%% file: PolyaTreeEnsemble_Randrianarisoa2020.tex
\documentclass[bj,preprint]{imsart}

\RequirePackage{amsthm,amsmath,mathtools,amsfonts,amssymb}
\RequirePackage[numbers]{natbib}
\RequirePackage[colorlinks,citecolor=blue,urlcolor=blue]{hyperref}
\RequirePackage{graphicx}

 \usepackage{enumitem}
 \usepackage{mdframed} 
 \usepackage{graphicx}
 \usepackage{dsfont} 
 \usepackage{bigints} 
 \usepackage{tikz}
 \usetikzlibrary{decorations.pathreplacing}
 \usepackage{float}   
 \usepackage{nccmath}

 \usepackage{subcaption}

\startlocaldefs

\makeatletter
\newcommand*\bigcdot{\mathpalette\bigcdot@{.5}}
\newcommand*\bigcdot@[2]{\mathbin{\vcenter{\hbox{\scalebox{#2}{$\m@th#1\bullet$}}}}}

\def\namedlabel#1#2{\begingroup
    #2%
    \def\@currentlabel{#2}%
    \phantomsection\label{#1}\endgroup
}

\newcommand{\pushright}[1]{\ifmeasuring@#1\else\omit\hfill$\displaystyle#1$\fi\ignorespaces}
\newcommand{\pushleft}[1]{\ifmeasuring@#1\else\omit$\displaystyle#1$\hfill\fi\ignorespaces}

\makeatother

\newcommand\given[1][]{\:#1\vert\:}

\newcommand\restr[2]{{
  \left.\kern-\nulldelimiterspace 
  #1 
  \vphantom{\big|} 
  \right|_{#2} 
  }}
 
\newcommand{\N}{\mathbb{N}}
\newcommand{\R}{\mathbb{R}}
\newcommand{\Z}{\mathbb{Z}}
\newcommand{\norm}[1]{\left\lVert#1\right\rVert}

\newtheorem{theorem}{Theorem}
\newtheorem{lemma}{Lemma}

\theoremstyle{remark}

\newcommand\independent{\protect\mathpalette{\protect\independenT}{\perp}}
\def\independenT#1#2{\mathrel{\rlap{$#1#2$}\mkern2mu{#1#2}}}

\endlocaldefs

\allowdisplaybreaks[4]

\begin{document}

\begin{frontmatter}
\title{Smoothing and adaptation of shifted Pólya Tree ensembles}
\runtitle{Smoothing and adaptation: Pólya Tree ensembles}

\begin{aug}
\author[A]{\fnms{Thibault} \snm{Randrianarisoa}\ead[label=e1]{thibault.randrianarisoa@sorbonne-universite.fr}}
\address[A]{LPSM,
Sorbonne Université, Paris, France.
\printead{e1}}

\end{aug}

\begin{abstract}
Recently,  S. Arlot and R. Genuer have shown that a random forest model outperforms its single-tree counterpart in estimating $\alpha-$Hölder functions, $1\leq\alpha\leq2$. This backs up the idea that ensembles of tree estimators are smoother estimators than single trees. On the other hand, most positive optimality results on Bayesian tree-based methods assume that $\alpha\leq1$. Naturally, one wonders whether Bayesian counterparts of forest estimators are optimal on smoother classes, just as observed with frequentist estimators for $\alpha\leq 2$.
We focus on density estimation and introduce an ensemble estimator from the classical (truncated) Pólya tree construction in Bayesian nonparametrics. Inspired by the work mentioned above, the resulting Bayesian forest estimator is shown to lead to optimal posterior contraction rates, up to logarithmic terms, for the Hellinger and $L^1$ distances on probability density functions on $[0;1)$ for arbitrary Hölder regularity $\alpha>0$. This improves upon previous results for constructions related to the Pólya tree prior, whose optimality was only proven when $\alpha\leq 1$. Also, by adding a hyperprior on the trees' depth, we obtain an adaptive version of the prior that does not require $\alpha$ to be specified to attain optimality.
\end{abstract}


\end{frontmatter}

\section{Introduction}

Tree-based methods and associated ensemble methods rank among the best performers in statistical learning and are, as such, widely used in practice. In nonparametric regression and classification problems, popular versions such as CART (Classification and Regression Trees)  \citep{cart84} and random forests \citep{RFBreiman} have sparked the interest of numerous researchers. While the former's theoretical behavior in an $L^2$-loss context is now quite well understood \citep{Tree_OracBnd, Donoho, geyNedelec} and is still at the heart of recent works \citep{klusowski2019analyzing}, the latter have been the object of fewer theoretical results. Indeed, it relies on recursive partitionings of the sample space and piecewise constant predictions, both data-dependent, which renders their analysis quite difficult. To this date, research has mainly focused on simplified versions of the original algorithm by Breiman (for which results are scarce \citep{SCORNET201672}) to prove properties such as consistency \citep{pmlr-v32-denil14, scornet:hal-00990008, 2014arXiv1407.3939A, 10.5555/1390681.1442799, wager2015adaptive, JMLR:v17:14-168, klusowski2018sharp, ishkog} or asymptotic normality \citep{JMLR:v17:14-168, doi:10.1080/01621459.2017.1319839}. We refer the interested reader to the review  \citep{reviewBiau} for a more comprehensive account of the topic's literature. 

The picture is even more striking with Bayesian versions of these forest estimators, as the theoretical development has just started to emerge. As for the advantages of these Bayesian counterparts, we cite their inherent ability to quantify uncertainty through posterior distributions and their propensity for adaptation. Popular algorithms are Bayesian CART \citep{10.1093/biomet/85.2.363, 10.2307/2669832} and BART (Bayesian Additive Regression Trees) \citep{chipman2010}, the latter being the prototype of Bayesian tree ensemble models. This motivated a myriad of works applying and adapting these methodologies in different domains (e.g., genomic studies \citep{LIU201017}, credit risk prediction \citep{ZHANG20101197}, predictions in medical experiments \citep{somaticBART}, etc.) and to different statistical problems beyond nonparametric regression (e.g., classification \citep{chipman2010, ZHANG20101197},  variable selection \citep{chipman2010, doi:10.1080/01621459.2016.1264957}, estimation of monotone functions \citep{mbart}, causal inference \citep{Hill07bayesiannonparametric}, Poisson processes inference\citep{Bartpoisson}, linear varying coefficient models\citep{vcbart}, heteroskedasticity \citep{doi:10.1080/10618600.2019.1677243, bleich2014bayesian}, etc.). A thorough presentation of the unified framework underlying BART and its extensions appears in \citep{tan2019bayesian}. The appeal also finds its roots in such algorithms' competitive practical performance: their great prediction abilities, robustness to ill-specified tuning parameters, and associated efficient posterior computation techniques \citep{f7b39f8e311043979b564404cff1206b, pratola2016, he2020stochastic, He2019XBARTAB}. At the other end of the spectrum, we are still at the dawn of their theoretical study. In the very recent years, researchers have come up with results in posterior contraction rate theory for trees and forests in $L^2$-loss \citep{rockova2017posterior, Theorybart}, in $L^\infty$-loss \citep{BCART_contrac}, and uncertainty quantification \citep{BCART_contrac, rockova2019semiparametric}. For more details, we refer the interested reader to the recent reviews \citep{LineroReview, doi:10.1146/annurev-statistics-031219-041110}.

Another famous prior in Bayesian nonparametrics that relies on recursive partitioning encoded into a dyadic tree structure is the Pólya tree process. The first related studies date back to the '60s \citep{freedman1963, freedman1965, ferguson1973, Ferguson} while the name appears for the first time in \citep{mauldinetal92, Lavine} (it originates in a link with Pólya urns, nicely explained in \citep{fundamentalsBNP}). As a prior on probability distributions and densities, it has since been used regularly and has recently inspired a new line of research coming up with modified versions \citep{WongMa, Hjort, NBM}. Lately, posterior contraction rates in $L^\infty$ and Bernstein-von Mises theorems were obtained for Pólya Trees and spike-and-slab Pólya Trees \citep{MR3729648, castillo2019spike}.

However, a significant shortcoming of these positive results of (almost-)optimal performances for tree-based methods (ensemble versions or not) is that they typically assume that the signal has limited regularity. Indeed, they study functional parameters that either lie in between step-functions \citep{NIPS2017_6804} and Lipschitz applications \citep{MR3729648, castillo2019spike} or belong to an additive model with Lipschitz components \citep{rockova2017posterior}. This comes from the fact that tree-based partitioning induces piecewise-constant estimators, which are usually too rough for efficient inference of smooth applications. Nevertheless, it has been noted that the aggregation of individual estimates may have a "smoothing" effect. This idea is already present in Breiman's bagging \citep{buhlmann2002}. It thus seems conceivable that a forest, i.e., a tree ensemble, may be more regular and enjoy optimal rates with even more restrictive smoothness assumptions. A few years ago, Arlot \& Genuer \citep{2014arXiv1407.3939A} indeed showed that this could be the case in a regression setting. Their method is described in more detail in Section \ref{aggPolya}. In addition, some works made links between random forests and kernel estimators \citep{scornet:hal-01255002, 2014arXiv1407.3939A}. This similarity was also recently established by \citep{oreilly2020stochastic} in the context of another tree ensemble model based on the Mondrian process \citep{NIPS2008_3622}, the Mondrian Forest \citep{lakshminarayanan2014mondrian}. The forest models developed there are such that the random construction of single trees is independent of the observed data. As mentioned above, such simplification has proved fruitful to come up with theoretical results. These are commonly known as \textit{Purely random forests}. They are also interesting for the study of Bayesian tree ensemble methods in that they too rely on the specification of a probability distribution on sample space recursive partitionings.

The present work builds upon ideas from \citep{2014arXiv1407.3939A}. The authors developed a \textit{Purely random forest} model that attains minimax convergence rates on the class of twice differentiable functions on $[0;1]$, for a modified $L^2$-loss (excluding points near the interval frontier). However, the related single-tree estimators are shown to be optimal only up to once-differentiable regularity. In the following, we will see how it is possible to adapt their aggregation of trees to forests to Pólya trees. We introduce a new prior on probability density functions, the Discrete Pólya Aggregate (DPA) prior, which is a toy random forest incorporating an aggregation step in its definition. Our contribution is threefold. First, we prove that the aggregation in DPA induces smoothing: the corresponding posterior distribution attains the optimal (up to log terms) Hellinger rates on classes of densities of arbitrary Hölder regularity $\alpha$ (no upper bound). Note that \citep{2014arXiv1407.3939A} and \citep{mourtada2018minimax} achieve this only up to $\alpha=2$. We demonstrate this smoothing through a link between DPA and spline densities. Second, our construction is adaptive: our results hold without knowledge of the regularity parameter through a hyperprior on the tree depth. Furthermore, we show how to handle smoothing at the domain's frontier by slightly modifying the prior close to the edges. Third, the DPA prior can be seen as a possible 'way' to smooth PT priors, a question left open in \citep{MR3729648, castillo2019spike}, at least here in terms of Hellinger rates. These results highlight the benefits that ensemble methods can have in Bayesian nonparametrics, smoothing the estimator to attain optimality on a broader class of problems and allowing adaptation. It is worth mentioning that \citep{3057a8ad4a924d30a262fc9586c26ee9} showed that a Bayesian forest made of 'smooth' decision trees adapts to high regularities in a regression setting. However, their individual 'tree predictors' are already smooth, whereas we work with 'hard' histogram-tree predictors, as in original random forests.

Deep neural networks (DNNs) have also proved to be pretty powerful in many problems. Their flexibility and capacity to learn feature representations place them among state-of-the-art methods. Thus, it aroused interest in the potential connections with random forest estimators. First, \cite{sethi, MR3297249} established that sparse shallow networks can realize these tree-based functions. Then, some recent papers also produced hybrid methods combining upsides of both, e.g., the representation learning properties of DNNs and the computational efficiency of random forests. For instance, we mention \emph{Deep Neural Decision Forests} \cite{7410529}, \emph{conditional networks} \cite{Ioannou2015}, \emph{Deep Neural Decision Trees} \cite{yang2018deep}, and \emph{Neural Random Forests} \cite{MR4043477}, for which consistency in nonparametric regression was proved, starting the theoretical study of these connections. There is also a surge in the development of the theoretical analysis of DNNs. Nonparametric methods relying on piecewise linear approximations are traditionally suboptimal with signals of smoothness $\alpha>2$. However, related to the type of results we present here, DNNs overcome this limitation. Even though sparse DNNs with ReLU activation function results in piecewise linear maps, \cite{MR4134774} showed that, in nonparametric regression, near minimax rates for arbitrary smoothness were attained. Our results on adaptation to smoothness are essentially similar, though we start with even rougher object that are not even continuous. Finally, whereas we link DPA to spline densities, \cite{pmlr-v80-balestriero18b} builds on the relation between DNNs and spline functions to obtain elements of approximation theory for DNNs. Also, \cite{ECKLE2019232} proved that DNN performances are competitive when compared with some usual spline methods (e.g., MARS \cite{MR1091842}).

This work outline is as follows: Section \ref{aggPolya} introduces our study framework and the aggregation ideas from \citep{2014arXiv1407.3939A} before presenting the DPA prior. Then, in Section \ref{results}, we expound on our results on the DPA posterior and other ones on priors inspired by the study of DPA. After a short discussion, the article body ends with the proofs of these results in Section \ref{proofs_theo}. The appendix presents helpful lemmas and elements used to derive our main results and a short numerical study illustrating our theoretical analysis.

\input{Agg.tex}

\input{results.tex}

\input{proofs.tex}

\bibliographystyle{imsart-number}
\bibliography{biblio}

\newpage

\input{Supplements.tex}

\end{document}

%% file: Agg.tex
\section{Aggregation of a Pólya Tree}
\label{aggPolya}
\subsection{Framework.}
Let's elaborate on the problem at hand, which is the one of density estimation. Below, $P_{f}$  is the probability distribution on $\Omega\coloneqq[0;1)$ with density $f$ w.r.t. Lebesgue measure $\lambda$. In a full Bayesian framework, one specifies a distribution on a pair $(X^{(n)}, f)$, determined by a prior on probability densities $f$, denoted $\Pi$, and the conditional distribution $X^{(n)}\given f \sim P_{^f}^{_{\otimes n}}$. From this, one obtains the posterior distribution, denoted $\Pi\left[\cdot\given X\right]$, omitting the superscript for conciseness.
We adopt a frequentist point of view in the analysis of the Bayesian procedure. Indeed, we assume that $X$ follows the distribution $P_{^{f_0}}^{_{\otimes n}}$ for a given $f_0$ instead of the marginal distribution of the pair $(X, f)$. Accordingly, we are interested in the asymptotical behaviour of $\Pi[\cdot\given X]$ under such conditions. In this paper, we introduce priors on probability densities such that the associated posteriors concentrate their masses on balls with center $f_0$ at optimal rates (up to logarithmic factors).\\

A central assumption for our subsequent analysis is that $f_0$ is $\alpha$-Hölderian ($\alpha>0$), i.e., it belongs to the Hölder class, with $\lfloor\alpha\rfloor$ the biggest integer strictly smaller than $\alpha$,
\[\Sigma(\alpha, [0,1) )\coloneqq\left\{f:[0,1)\mapsto\mathbb{R}\ \given \ ||f||_{\Sigma(\alpha)} \coloneqq \underset{x\neq y}{\text{sup}} \frac{|f^{(\lfloor\alpha\rfloor)}(x)-f^{(\lfloor\alpha\rfloor)}(y)|}{|x-y|^{\alpha-\lfloor\alpha\rfloor}} <+\infty\right\}.\]
In the following, we write, for real positive sequences $u_n,v_n$, $u_n \lesssim v_n$ whenever there exists a constant $C>0$ independent of $n$ such that for any $n$ large enough, $u_ n\leq C v_n$ ($\gtrsim$ is defined likewise). Also, if $u_n \lesssim v_n$ and $u_n \gtrsim v_n$, we write $u_n \asymp v_n$, while $u_n\propto v_n$ means that there exists a constant $C$ such that $u_n=Cv_n$. When comparing two quantities $a,b\in\R$, we write $a\vee b \coloneqq \max(a,b)$ and $a\wedge b \coloneqq \min(a,b)$. For random variables $X$ and $Y$, $X\independent Y$ means independence. $\mathbb{E}_f$ denotes the expectation under $P_f^{_{\otimes n}}$, as there will no ambiguity on $n$ in the following. Also, as pointed out, we denote $\lfloor a \rfloor$ the greater integer strictly smaller than $a$. As for the usual floor operator, it is written $\lfloor \cdot \rfloor_\text{f}$. For real univariate maps $f,g$, $f*g$ is their standard convolution.
The $n$-dimensional unit simplex is \[S^{n} \coloneqq \left\{(x_1,\cdots,x_{n})\in \R^{n} \given x_i \geq 0, \sum_{i=1}^{n} x_i =1\right\}\]
and the $\epsilon$-covering number $N(\epsilon,A,d)$, for some $\epsilon>0$ and $A$ a subset of metric space $\left(V,d\right)$, is the minimum number of $\epsilon$-balls with centers in $V$ needed to cover $A$. $\norm{\cdot}_1$ is the $L^1(\Omega)$ norm and $h$ is the Hellinger distance. 
\subsection{Smoothing of frequentist forest estimators.}
\label{freq_forests}

Since they entail piecewise-constant estimators with independent heights, inference methods that rely on single-tree constructions are usually limited in their performance. They are generally suboptimal on balls of Hölder classes with regularity $\alpha>1$. Nonetheless, there is hope that their ensemble methods are less prone to such problems and better suited for the inference of smoother parameters. In this section, we discuss a toy model from \cite{2014arXiv1407.3939A} which confirms this intuition.

\begin{figure}[!h]
\centering

\begin{tikzpicture}[every edge/.style={shorten <=1pt, shorten >=1pt}]

  \draw (0,0)  node [left] {0} -- (10,0) node [right] {1};
  \coordinate (p) at (0,4pt);
  \foreach \myprop/\mytext [count=\n] in {2.5/$A_{0}$,2.5/$A_{1}$,2.5/$A_{2}$,2.5/$A_{3}$}
  \draw [decorate,decoration={brace,amplitude=4}] (p)  edge [draw] +(0,-8pt) -- ++(\myprop,0) coordinate (p) node [midway, above=2pt, anchor=south] {\mytext} ;
  \path (10,4pt) edge [draw]  ++(0,-8pt);
  
  \draw[Bar->,] (0,-15pt) -- ( 1,-15pt ) node [midway, below] {U};
  \draw[dashed] (0,-15pt) -- ( 0,0 );
  \draw[Bar->,] (2.5,-15pt) -- ( 3.5,-15pt ) node [midway, below] {U};
  \draw[dashed] (2.5,-15pt) -- ( 2.5,0 );
  \draw[Bar->,] (5,-15pt) -- ( 6,-15pt ) node [midway, below] {U};
  \draw[dashed] (5,-15pt) -- ( 5,0 );
  \draw[Bar->,] (7.5,-15pt) -- ( 8.5,-15pt ) node [midway, below] {U};
  \draw[dashed] (7.5,-15pt) -- ( 7.5,0 );
  
    \draw (0,-30pt)  node [left] {0} -- (10,-30pt) node [right] {1};
   \draw (0,-26pt) -- (0,-34pt);
  \draw (1,-26pt) -- (1,-34pt);
  \draw (3.5,-26pt) -- (3.5,-34pt);
  \draw (6,-26pt) -- (6,-34pt);
  \draw (8.5,-26pt) -- (8.5,-34pt);
  \draw (10,-26pt) -- (10,-34pt);

  \draw[dashed] (1,-15pt) -- ( 1,-30pt );
   \draw[dashed] (3.5,-15pt) -- ( 3.5,-30pt );
    \draw[dashed] (6,-15pt) -- ( 6 ,-30pt );
    \draw[dashed] (8.5,-15pt) -- ( 8.5 ,-30pt );
    
    \coordinate (p) at (0,-34pt);
    \foreach \myprop/\mytext [count=\n] in {1/$B_{0}$,2.5/$B_{1}$,2.5/$B_{2}$,2.5/$B_{3}$, 1.5/$B_{4}$}
  \draw [decorate, decoration={brace,mirror,amplitude=4}] (p)  -- ++(\myprop,0) coordinate (p) node [midway, below=20pt, anchor=south] {\mytext} ;
\end{tikzpicture}

\caption{Random shift of a regular partition.}
\label{shift_fig}
\end{figure}
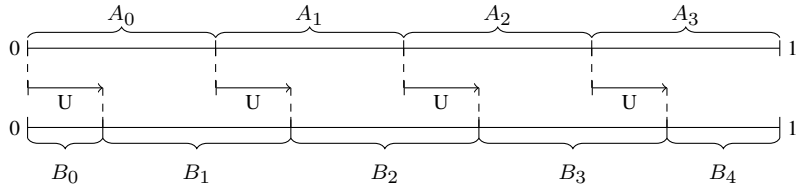

Let's assume that we are faced with estimating some map $f_0\colon [0;1) \to \R$ (let it be in density estimation, regression, Gaussian white noise problem, etc.). Tree-based methods build upon a recursive partitioning of the interval $[0;1)$ so that the estimator is piecewise constant on this given partition.  Some data average typically defines the values taken on each cell. Conversely, any partition in subintervals can be represented by a binary tree. Purely random forests then rely on the aggregation of random piecewise constant estimators, the cells of the partitions being random and independent of the observed data.

A particular toy distribution $\mathcal{P}_{\text{toy}}$ on partitions presented in \cite{2014arXiv1407.3939A} is sketched in Figure~\ref{shift_fig}. Given $k\in\N^*$, the partition $\mathbb{U}$ is defined starting with a regular partition of step $k^{-1}$ whose breakpoints are shifted to the right by $U/k$, with $U\sim\mathcal{U}[0;1]$. A tree estimator averages the data on the partition $P\sim \mathcal{P}_{\text{toy}}$ while a $q$-forest, $q\geq1$, is itself the average of $q$ tree estimators corresponding to the i.i.d. partitions $P_i\sim \mathcal{P}_{\text{toy}}$, $i=1,\dots,q$. In the context of nonparametric regression with a modified $L^2$-loss that only takes into account points far enough from the frontier of $[0;1]$, \cite{2014arXiv1407.3939A} shows that forests with a sufficient amount of trees and a well-chosen $k$ attain optimal convergence rates in the estimation of twice differentiable functions. An intuition for this result is that, in the limit $q\to\infty$, the forests actually mimic a triangular kernel estimator (see Proposition 4 from \cite{2014arXiv1407.3939A}). On the other hand, even with an optimal $k$, single-tree estimators cannot do better than usual histogram estimators and are optimal only when the function $f_0$ has at best Lipschitz regularity. While the approximation error of a tree is controlled by the $L^2$-projection of $f_0$ on a linear space of piecewise-constant maps, forests' one is controlled by the average of such projections on different spaces (the cells varying between spaces). It then appears that aggregation allows forests to borrow strength from a smoother object that enjoys nice approximation properties, bringing about the estimator's smoothing.

However, aggregating once still only allows the obtention of minimax convergence rates corresponding to twice differentiable regularity at most. It is unclear how the idea could be pushed further to adapt to arbitrary regularities in this context. Below, we see that it is possible to do so with Bayesian estimators.

\subsection{The DPA prior.}
\label{DPA_subsection}

Since we are focusing on density estimation, it is sensible to delve into the Pólya Tree prior, and more particularly into its finite version where the tree is truncated at a given depth. We talk about the Truncated Pólya Tree (TPT) prior to refer to the distribution on probability density functions defined in the following paragraph. 

\begin{figure}[!h]
\centering
\begin{tikzpicture}[
very thick,
level 1/.style={sibling distance=6.5cm},
level 2/.style={sibling distance=2.5cm},
level 3/.style={sibling distance=1cm},
every node/.style={circle,solid, draw=black,thin, minimum size = 0.5cm},
emph/.style={edge from parent/.style={dashed,black,thin,draw}},
norm/.style={edge from parent/.style={solid,black,thin,draw}}
]
\node [rectangle] (r){$\Omega=[0;1)$}
  child {
    node [rectangle] (a) {$I_0=[0;1/2)$}
    child {
      node [rectangle] {$I_{00}=[0;1/4)$}
      child[emph] {
        node [rectangle] {$g(x)=4Y_0Y_{00}$}
      }
       edge from parent node[left, draw=none]{$Y_{00}\sim \text{Beta}\left(\nu_{00}, \nu_{01}\right)$}
    }
    child {
      node [rectangle] {$I_{01}=[1/4;1/2)$}
      child[emph] {
        node [rectangle] {$g(x)=4Y_0Y_{01}$}
      }
      edge from parent node[right, draw=none]{$Y_{01}=1-Y_{01}$}
    }
    edge from parent node[left, draw=none]{$Y_0\sim \text{Beta}\left(\nu_{0}, \nu_{1}\right)\quad$}
  }
  child {
    node [rectangle] {$I_{1}=[1/2;1)$}
    child {
      node [rectangle] {$I_{10}=[1/2;3/4)$}
      child[emph] {
        node [rectangle] {$g(x)=4Y_1Y_{10}$}
      }
      edge from parent node[left, draw=none]{$Y_{10}\sim \text{Beta}\left(\nu_{10}, \nu_{11}\right)$}
    }
    child {
      node [rectangle] {$I_{11}=[3/4;1)$}
      child[emph] {
        node [rectangle] {$g(x)=4Y_1Y_{11}$}
      }
      edge from parent node[right, draw=none]{$Y_{11}=1-Y_{10}$}
    }
    edge from parent node[right, draw=none]{$\quad Y_1=1-Y_0$}
  };
\end{tikzpicture}
\caption{Truncated Pólya Tree at depth $L=2$.}
\label{TPT_fig}
\end{figure}
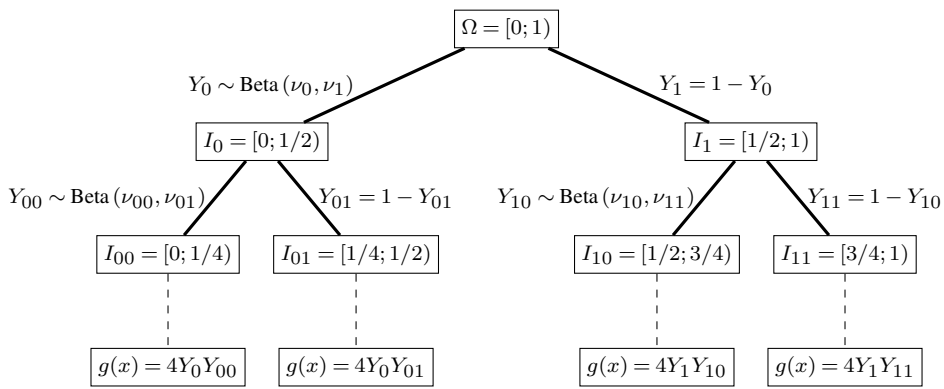

We introduce some notation and illustrate the construction of the TPT prior in Figure \ref{TPT_fig} (see also Chapters 3.5-3.7 of \cite{fundamentalsBNP}). For any $l>0$ and $0\leq k<2^l$, the dyadic number $r=k2^{-l}$ has a binary expansion $0.\kappa_1\dots\kappa_l$ (possibly padded with $0$'s on the right) with $\kappa_i\in\{0,1\}$ for $i=1,\dots,l$, i.e., $r=\sum_{j=1}^l\kappa_j2^{-j}$. Reciprocally, any $\kappa\in\{0,1\}^l$ with $l>0$ corresponds to the binary expansion of a dyadic number of the form $k2^{-l}$ for some $0\leq k<2^l$. Consequently, we write $\kappa(l,k)\in\{0,1\}^l$ the sequence of length $\left|\kappa(l,k)\right|=l$ corresponding to the dyadic $k2^{-l}$ (by convention $\kappa(0,0)=\varnothing$). Then, for $0<i\leq l$ and $\kappa=\kappa_1\dots\kappa_l\in\left\{0;1\right\}^l$, $\kappa^{[i]}\coloneqq \kappa_1\dots\kappa_i$. For any pair $(l,k)\in\N\times\Z$, we introduce $I_{l,k}\coloneqq\left[\frac{k}{2^l}, \frac{k+1}{2^l}\right)$ and, whenever $l\geq 0, 0\leq k<2^l$, $I_{\kappa(l,k)}\coloneqq I_{l,k}$. One sees that $I_\kappa=I_{\kappa0}\cup I_{\kappa1}$ and $\Omega=\cup_{\kappa:\ \left|\kappa\right|=l} I_\kappa$ so that the union of the sets of the form $I_\kappa$ defines a recursive partitioning of the unit interval. In Figure \ref{TPT_fig}, one sees that this partitioning consists of splitting each interval in its midpoint from one level to another. We refer to $I_{\kappa0}$ (resp. $I_{\kappa1}$) as the left (resp. right) child of $I_{\kappa}$. Therefore, $\kappa$ encodes the sequence of partitioned sets from $\Omega$ to $I_\kappa$ according to this relationship.
  
Then, for $L\in\N^*$ and a set of positive real parameters $\mathcal{A}=\left\{\upsilon_\kappa, \kappa\in\cup_{l=1}^L \{0,1\}^l\right\}$, we say that the Lebesgue probability density $g$ follows a Truncated Pólya Tree distribution $\text{TPT}_L\left(\mathcal{A}\right)$ for the above recursive partitioning scheme if there exist random variables $0\leq Y_\kappa \leq 1$ for any $0<|\kappa|\leq L$ such that

\begin{itemize}
  \item the variables $Y_{\kappa0}$ with $0\leq|\kappa|\leq L-1$ are mutually independent and $Y_{\kappa0}\sim \text{Beta}(\upsilon_{\kappa0},\upsilon_{\kappa1})$,
  \item $Y_{\kappa1}=1-Y_{\kappa0}$ for $0\leq|\kappa|\leq L-1$,
  \item $\forall \kappa\ \text{s.t.}\ |\kappa|=L,\ \forall x \in I_\kappa,\ g(x)=2^L\prod_{l=1}^{L}Y_{\kappa^{[l]}}$.
\end{itemize}

The link between this construction and dyadic trees is illustrated in Figure~\ref{TPT_fig} where we see that $L$ defines the depth of a tree that encodes a partition of $\Omega$. One sees that for any sequence $\kappa$ as above, $Y_{\kappa0}=P_g\left(I_{\kappa0}\right)/P_g\left(I_\kappa\right)=P_g\left(I_{\kappa0}\ | I_{\kappa}\right)$. Subsequently, we always assume for simplicity that the parameters $\upsilon_\kappa$ in the set $\mathcal{A}$ verify $\upsilon_\kappa=a_{|\kappa|}$, so that we rather write $\mathcal{A}=\left(a_l\right)_{0<l\leq L}$.

Setting $H_{L,i}\coloneqq 2^{L}\mathds{1}_{I_{L,i}}$, a probability density $g$ defined as above can be written
\begin{equation}\label{TPT_proba_expr} g = \sum_{i=0}^{2^L-1} \Theta_i H_{L,i},\qquad \Theta_i = \prod_{l=1}^L Y_{\kappa\left(l, \lfloor i2^{l-L}\rfloor_{\text{f}}\right)},\end{equation}
so that it belongs almost surely to
\[C_L\coloneqq\left\{h:[0,1)\mapsto\mathbb{R}^+\ \Big|\ \int h =1,\ h\text{ is constant on } I_\kappa, |\kappa|=L \right\}.\]
As pointed out before, the elements of $C_L$ are too rough to approximate efficiently smooth applications, so we would like to leverage ideas developed in the last section to obtain "smoother" prior samples.

The TPT prior defined above has samples that are piecewise constant on some dyadic partition of $\Omega$. However, it would be possible to develop a similar prior so that the samples are piecewise constant on a different partition. Figure \ref{shift_TPT_fig} illustrates such construction in the case of a dyadic partition shifted by some quantity $S$, just like in Section \ref{freq_forests}. We point out a slight difference from what we described in Section \ref{freq_forests}: to be encoded in a complete binary tree of depth $l\geq0$, the recursive partitions need to have $2^l$ elements at level $l$, while a shifted dyadic partition has $2^l+1$ elements (see Figure \ref{shift_fig}). The way to go here is to merge the external cells, corresponding to sets $B_0$ et $B_4$ in Figure~\ref{shift_fig} or $\left[0;S\right)$ and $\left[3/4+S;1\right)$ in Figure~\ref{shift_TPT_fig}. Following the ideas from Section \ref{freq_forests}, for some $L>0$, we could then define a new prior whose samples are the averages of $q\geq 0$ independent maps. Each of these maps would follow an independent TPT prior with depth $L$, with their underlying partitions that are dyadic partitions shifted by $q$ independent uniform random variables.

As we will see in the following sections, the priors must allocate some of their mass to subsets of limited complexity to obtain posterior contraction rates. It is not the case of the construction we just proposed, as the samples belong to a functional class that is too rich, and the prior mass is overly spread out. Consequently, we need to impose some correlation between the $q$ tree maps. We propose a modified prior in which the aggregated trees are not independent and the shifts of their dyadic partitions are deterministic. Indeed, as we do not want the prior to have its samples that are much more complex than the ones of the TPT distribution (for a given $L>0$), we would like it to involve just as many $Y$'s Beta random variables.

First, for $f: \mathbb{R} \to \mathbb{R}$, $q\in \mathbb{N}^*$ and $s>0$, we define a finite aggregation step as the effect of the map $f\to f_{q,s}^{1} $ defined by

\begin{align}
\label{aggregating_map3}
\begin{split}
  f_{q,s}^{1} \colon \mathbb{R} &\to \mathbb{R}\\
  x &\mapsto \frac{1}{q}\sum_{i=0}^{q-1} f\left(x - \frac{is}{q}\right)
 \end{split}
\end{align}
for any application $f\colon \R \to \R$. If $f$ is the $1$-periodic extension of a $TPT_L$ sample and $s=2^{-L}$, $L>0$, the restriction to $[0;1)$ of $f_{q,s}^{1}$ is the aggregation of $q\geq0$ piecewise constant maps as described above. More precisely, it is the aggregation on $[0;1)$ of $q$ maps constructed in the same way as in Figure \ref{shift_TPT_fig}, and that share the same values (the $Y$'s variable are identical) on their respective underlying partitions, which are dyadic partitions shifted by $S_i=iq^{-1}2^{-L}$ for $i=0,\dots,q-1$. The resulting map could then be viewed as the sample of a simplified prior on forests, based on TPTs, which is the pushforward measure of the $TPT_L$ measure by the 'aggregating' map \eqref{aggregating_map3}. However, this restriction on $[0;1)$ of a $1$-periodic map is 'cyclical' over the frontier of the interval, as most of the trees have the same value near $0$ and $1$ (cf. Figure~\ref{shift_TPT_fig}). We discuss this after the complete definition of our new priors.

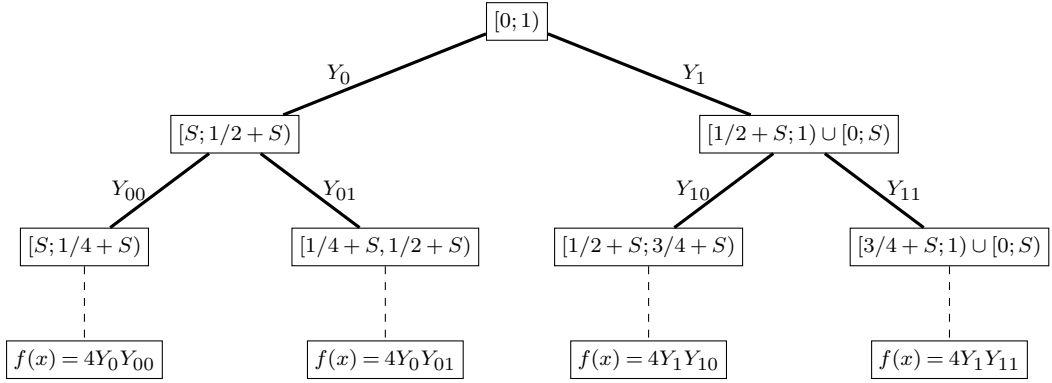
\begin{figure}

\centering
\begin{tikzpicture}[
very thick,
level 1/.style={sibling distance=7.5cm},
level 2/.style={sibling distance=4cm},
level 3/.style={sibling distance=2.5cm},
every node/.style={circle,solid, draw=black,thin, minimum size = 0.5cm},
emph/.style={edge from parent/.style={dashed,black,thin,draw}},
norm/.style={edge from parent/.style={solid,black,thin,draw}}
]
\node [rectangle] (r){$\left[0;1\right)$}
  child {
    node [rectangle] (a) {$\left[S;1/2+S\right)$}
    child {
      node [rectangle] {$\left[S;1/4+S\right)$}
      child[emph] {
        node [rectangle] {$f(x)=4Y_0Y_{00}$}
      }
       edge from parent node[left, draw=none]{$Y_{00}$}
    }
    child {
      node [rectangle] {$\left[1/4+S, 1/2+S\right)$}
      child[emph] {
        node [rectangle] {$f(x)=4Y_0Y_{01}$}
      }
      edge from parent node[right, draw=none]{$Y_{01}$}
    }
    edge from parent node[left, draw=none]{$Y_0\quad$}
  }
  child {
    node [rectangle] {$\left[1/2+S; 1\right)\cup\left[0;S\right)$}
    child {
      node [rectangle] {$[1/2+S;3/4+S)$}
      child[emph] {
        node [rectangle] {$f(x)=4Y_1Y_{10}$}
      }
      edge from parent node[left, draw=none]{$Y_{10}$}
    }
    child {
      node [rectangle] {$\left[3/4+S; 1\right)\cup\left[0;S\right)$}
      child[emph] {
        node [rectangle] {$f(x)=4Y_1Y_{11}$}
      }
      edge from parent node[right, draw=none]{$Y_{11}$}
    }
    edge from parent node[right, draw=none]{$\quad Y_1$}
  };
\end{tikzpicture}
\caption{Shifted Truncated Pólya Tree at depth $L=2$, with shift S.}
\label{shift_TPT_fig}
\end{figure}

Before going further, let's discuss some aspects of our aggregation scheme and the definition of the 'Pólya forest' above. It has the benefit that it can be generalized to higher aggregation orders via the following recurrence relationship, which involves some weights depending on the degree of aggregation $k\in\N^*$:

\[ f_{q,s}^{k+1}(\cdot) \coloneqq  \Big(f_{q,s}^{k}\Big)_{q,s}^{1}(\cdot) = \sum_{i=0}^{(q-1)(k+1)} \sharp\left\{(j_1,\dots,j_{k+1})\in[\![0;q-1]\!]^{k+1},\  j_1+\dots+j_{k+1}=i\right\}  f\left(\cdot - \frac{is}{q}\right). \]
Aggregating $k$ times, we obtain "forests of forest", which are more general forests of weighted trees, with non-uniform weights. 
To see the actual effect of this operation, it is useful to have a look at what is happening in the limit $q\to\infty$ (i.e. for an "infinite forest"). If $f$ is Riemann integrable on any interval of length $s$, then letting $q\to\infty$ results in $f_{q,s}^{1}$ converging pointwise to
\begin{align*}
\begin{split}
  f_{\infty,s}^{1} \colon \mathbb{R} &\to \mathbb{R}\\
  x &\mapsto s^{-1}\int_{x-s}^{x} f(t) dt = s^{-1}(\chi_s * f)(x)
 \end{split}
\end{align*}
where $\chi_s(t) \coloneqq \mathds{1}_{[0;1]}(t/s)$. For simplicity, we write $\chi \coloneqq\chi_1$. This defines a continuous aggregation that can also be iterated, for $l\in\N^*$, \begin{equation}\label{rec_def}f_{\infty,s}^{l+1} \coloneqq  \Big(f_{\infty,s}^{l}\Big)_{\infty,s}^{1}= s^{-(l+1)} \chi_s^{*(l+1)} * f\end{equation}with $\chi_s^{*(l+1)}$ the $(l+1)$-th iterated convolution of $\chi_s$ with itself. Such an "infinite forest" (if $f$ is a tree function) is the continuous aggregation of a continuum of 'tree' maps, possibly with non-uniform weights for higher degrees of aggregation. Besides, we take the convention $f_{q,s}^{0}\coloneqq f_{\infty,s}^{0}\coloneqq f $. Thanks to Lemma \ref{compact_formula}, for $m\geq 1$, we also have the more explicit formula for the continuous aggregation

\begin{equation}\label{convol_cont_aggreg} f_{\infty,s}^{m}(\cdot) = s^{-1}\  \chi^{*m}(\cdot/s)*f .\end{equation}
As we will see later, the weighted aggregation of trees allows obtaining even smoother forests.
Finally, setting, for $0< i \leq m+1$, \begin{equation}\label{omegas_def}\omega_{m,i} \coloneqq \int_0^{i} \chi^{*(m+1)}\left(t\right) dt, \end{equation}we define our new prior, the Discrete Pólya Aggregation (DPA) distribution.  Fixing $m\in\N$, $L\in\N^*$ such that $2^{L-1}>m$, $q\in\N^*$, $U>0$ and a set of hyperparameters $\mathcal{A}$, the samples of  $\text{DPA}(m,L,q,\mathcal{A},U)$ are generated sequentially as follows:

\begin{enumerate}
\item\label{s1} \underline{Trees definition and handling of the frontier}
	\begin{enumerate}
		\item\label{a} Draw $g$ such that $g\sim TPT_L\left(\mathcal{A}\right)$. One writes \[g(\cdot) = \sum_{i=0}^{2^L-1} \Theta_i H_{Li}\] for some sequence $\left( \Theta_i \right)_{i=0}^{2^L-1}$ whose elements are positive and sum up to $1$.
		\item\label{b} Given $\Theta_i$, $0\leq i\leq 2^L-1$,
		\[
   \theta_i=
\begin{cases}
   \Theta_i& \text{if } 0\leq i \leq 2^L-m-1, \\
    v_i \sim \mathcal{U}\left[0\vee\frac{\Theta_i-(1-\omega_{m,2^L-i})U}{\omega_{m,2^L-i}} ;  U\wedge\frac{\Theta_i}{\omega_{m,2^L-i}} \right]   & \text{if } 2^L-m \leq i \leq 2^L-1, \\
   \theta_i=\frac{\Theta_{i-m}-\omega_{m,2^L+m-i}v_{i-m}}{1-\omega_{m,2^L+m-i}} & \text{if } 2^L \leq i \leq 2^L+m-1 \\
\end{cases}
\]where the uniform variables above are mutually independent.
		\item\label{c} Define the $(2^L+m)$-periodic sequence $\left(u_i\right)_{i\in\Z}$ such that $u_i=\theta_j$, $j\equiv i \mod 2^L+m$.
	\end{enumerate}

\item\label{s2} \underline{Aggregation}\\
 Set \[f=\sum_{i\in\Z}u_i H_{L,i}\] as the base tree and output, as the aggregation of shifted trees, the restriction on $[0;1)$ of \[f^m_{q,2^{-L}}/\int_{[0;1)} f^m_{q,2^{-L}}(v)dv.\] 

\end{enumerate}

Step \ref{a} and Step \ref{s2} gather the ideas we have presented up until now, namely the definition and the aggregation of a finite number of TPT samples, with shifted underlying partitions. Another related distribution that we name the Continuous Pólya Aggregation (CPA) prior and denote $\text{CPA}(m,L,\mathcal{A},U)$, corresponds to the situation where a continuum of Pólya Tree samples are aggregated with the operation \eqref{rec_def}. It is defined by a similar algorithm to the one above, except that Step \ref{s2} is replaced by

\begin{enumerate}
\item[\namedlabel{s2_prime}{2'.}]\underline{Aggregation}\\
 Set \[f=\sum_{i\in\Z}u_i H_{L,i}\] as the base tree and output the restriction of $f^m_{\infty,2^{-L}}$ on $[0;1)$ as the aggregation of shifted trees.
\end{enumerate}

Note that, according to Lemma \ref{density_prior_int}, it is no longer necessary to normalize the output function as the $u_i$ are defined so that $f^m_{\infty,2^{-L}}$ is almost surely a probability density.
Since, as mentioned before, the external sets of shifted partitions are merged (see Figure~\ref{shift_TPT_fig}), a 'Pólya forest' based on Steps \ref{a} and \ref{s2} (or \ref{s2_prime}) only would have the side effect of cycling over the frontier (i.e., it tends the same limits toward the frontier), as shown in Figure~\ref{fpf_draw}. It presents a sample from a 'naive' construction, aggregating shifted trees without treatment on the frontier. That is why the above definitions of the DPA and the CPA prior also feature Steps \ref{b} and \ref{c} to modify the frontier behavior. It is possible to appreciate the advantage of this treatment on the frontier in Figures \ref{fpf_draw2} and \ref{fpf_draw3}, where similar samples are plotted, with and without this modification. Without it, the samples would not be flexible enough to approximate general densities at a reasonable rate.
The figures we have just mentioned also help understand the role of each parameter in the definition of DPA and CPA. The $L$ parameter controls the depth of the trees (i.e., how refined the underlying partitions are), and, along with the degree of aggregation $m$, they define the smoothness of the prior. The number of trees $q$ controls the distance between the samples from DPA and CPA. Indeed, the latter acts as a 'limit' prior for the former, and we most naturally found our theoretical analysis on it. As we will see, properties satisfied by CPA are shared with DPA, up to some discretization effects to be controlled. As one can observe on the figures, samples from DPA are piecewise constant so that it is still a histogram prior. However, in Section \ref{proofs_theo}, it is shown that these samples are discrete approximations of spline functions, which are themselves sampled by CPA, and this added structure accounts for increased posterior performance (see Section \ref{results}). 

Rather useful for adaptive estimation, $U$ is a technical parameter related to our method to modify the samples on the frontier. A small value of $U$ allows the samples to have limited complexity. This point is a prerequisite for our approach to derive posterior contraction rate. It is also the reason for the simplifications highlighted above, with the partitions' shifts being deterministic and the trees all derived from a single TPT sample.

%% file: results.tex
\section{Main results}
\label{results}

\subsection{Posterior contraction rates for DPA.}
We now present our main results on the asymptotic behavior of our density estimation procedure based on the DPA prior. In the following Theorem \ref{Master_theorem2}, we see that for any fixed arbitrary degree of Hölder regularity of the true density $f_0$, the posterior distribution attains minimax contraction rates (up to a logarithmic factor). A critical parameter that needs to be adequately defined to obtain the right degree of smoothness is the trees' depth $L$. In the following theorem setting, when the regularity of the true density $f_0$ is given, we let the depth depend on the sample size and $\alpha$. Namely, we use the depth $L_n$ that is the closest integer to the solution $x$ of
\begin{equation*}\label{cutoff_eq} 2^x =  \left(\frac{n}{\log n}\right)^{\frac{1}{2\alpha+1}}. \end{equation*}

\begin{theorem}
\label{Master_theorem2}
Suppose $f_0\in  \Sigma(\alpha, [0,1))$, $\alpha>0$ and $f_0\geq\rho$ for some $\rho>0$. Let us endow $f$ with a $DPA\left(\lfloor \alpha \rfloor,L_n,n,(a_l)_{0<l\leq L_n},U_n\right)$ prior which we write $\Pi$, where $U_n\to\infty$ is an arbitrary sequence and such that, for some $\beta>0,R\geq1, \delta>0$, for any $0<l\leq L_n$,
\[\ a_l \in\Bigg[\delta n^{-\beta};R\Bigg].\]Then, for $M>0$ depending on $\rho, \alpha$, $\norm{f_0}_{\Sigma(\alpha)}$, $\beta$ and $R$, and $d(f,g)=\norm{f-g}_1$ or $d(f,g)=h(f,g)$, as $n\rightarrow\infty$, \[\mathbb{E}_{f_0}\Pi\Bigg[d(f_0,f)>M\left(\frac{\log n}{n}\right)^{\frac{\alpha}{2\alpha +1}}\given X\Bigg]\rightarrow0.\]
\end{theorem}

Through the use of $L_n$ and $\lfloor\alpha\rfloor$, this result relies on the assumption that the regularity $\alpha$ is fixed and known. The issue is that, in practice, we do not know this characteristic of the problem beforehand. Adding a prior on the depth of the trees $L$ (which will be linked by the below functional to the degree of aggregation used) allows a new hierarchical prior attaining the minimax contraction rate, with no requirement for the knowledge of any fine property of $f_0$ anymore. We introduce the following functional, defined for $l,n\in\N^*$,
\[ \xi(l,n) = \left\lfloor \frac{1}{2}\Bigg[ \frac{1}{l}\log_2\Bigg(\frac{n}{\log n}\Bigg)-1\Bigg] \right\rfloor.\]
For any depth value and sample size, it gives an estimate of the smoothness of the signal to be recovered. The result below shows that it is possible to define an adaptive prior that leads to optimal contraction rates for an arbitrary regularity $\alpha>0$. The idea is to add a hyperprior on different models, characterized by the depth $l$ and the corresponding estimated smoothness.

\begin{theorem}
\label{Adapt_smooth_theorem2}
Suppose $f_0\in  \Sigma(\alpha, [0,1))$, $\alpha>0$ and $f_0\geq\rho$ for some $\rho>0$. Let us endow $f$ with the following hierarchical prior which we write $\Pi$,
\begin{align*} l &\sim \Pi_L \\
f\given l &\sim DPA\left(\xi(l,n),l,v_nn\log^3 n,(a_l)_{0<l\leq L_n},\log n\right), \end{align*}
$v_n\to\infty$, and such that for $l>0$, $\Pi_L$  and the sequence $(a_l)_{l\in \N^*}$ satisfy, for some $\beta>0,R\geq1, \delta>0$,
\[\Pi_L \left[\{l\}\right]\propto 2^{-l2^l} \qquad\text{and}\qquad a_l \in\Bigg[\delta n^{-\beta};R\Bigg].\]Then, for $M>0$ depending on $\rho, \alpha$, $\norm{f_0}_{\Sigma(\alpha)}$, $\beta$ and $R$, and $d(f,g)=\norm{f-g}_1$ or $d(f,g)=h(f,g)$, as $n\rightarrow\infty$, \[\mathbb{E}_{f_0}\Pi\Bigg[d(f_0,f)>M\left(\frac{\log n}{n}\right)^{\frac{\alpha}{2\alpha +1}}\given X\Bigg]\rightarrow0.\]
\end{theorem}
 
The exponential decay of the atom probabilities in $\Pi_L$ is fast enough so that the prior still concentrates on a small number of models, but also slow enough so that it selects with high probability the model specified in Theorem \ref{Master_theorem2} for a given $\alpha$. From the point-of-view of Theorem \ref{GGV_theo}, an essential tool to prove our theorems, this hyperprior is a good tradeoff between the requirement that the prior concentrates its mass on a low-dimensional set and the one that it gives sufficient probability to small balls centered on the signal $f_0$. Therefore, our hierarchical prior behaves just like our non-adaptive one, attaining optimal contraction rates for any $\alpha>0$. 
 
A slight difference between Theorems \ref{Master_theorem2} and \ref{Adapt_smooth_theorem2} is that the sequence $U_n$ is replaced by $\log n$. Indeed, as we seek to use Theorem \ref{GGV_theo} to prove the above result, the slow growth of the logarithmic function ensures that the adaptive prior concentrates its mass on sieves of moderate complexity. At the same time, neighbourhoods of $f_0$ still have sufficient mass asymptotically. Also, the number of trees is of a higher order since, in the adaptive setting, some further work is necessary to handle discretization effects of finite forests. 

The novelty in these results is that these are the first tree-based priors that enjoy (almost-)optimal posterior contraction rates on classes of arbitrary smoothness to the best of our knowledge. It highlights how incorporating aggregation operations in priors leads to smoother forest samples when compared with tree priors. Previous results on related constructions were usually limited to Hölder classes of regularity $\alpha\leq1$ at most. The link between single tree structures and piecewise constant functions makes them too rough to estimate smooth signals. In comparison with the original toy model of Arlot et al. \cite{2014arXiv1407.3939A}, we extend the aggregation process so that the smoothing of the estimator occurs on regularity classes of orders even larger than $2$. Also, we have shown that forest aggregation is compatible with the definition of adaptive priors, as these two aspects do not come at the price of a loss in posterior contraction rate. All in all, these are the first results of adaptivity and general smoothing for Bayesian forest estimators. 
To conclude on the advances made here, we mention that the original results on forest estimators from \cite{2014arXiv1407.3939A} put aside the effect of their construction on the frontierv. The frequentist framework let them focus on a localized loss (namely, the mean integrated square error with integration on an interval strictly included in the unit interval), making the behavior on the frontier of $[0; 1)$ irrelevant. On the contrary, the constructions we come up with here deal with those side effects by slightly modifying the samples near the frontier of $\Omega$.

\subsection{Extension to other priors.}
\label{ext_new_priors}

Since the DPA prior is a discretized version of the CPA prior, the results from the previous section stem from the fact that it is possible to obtain similar ones for the CPA prior. Theorem \ref{adapt_CPA}  below is a version of Theorem \ref{Adapt_smooth_theorem2} for CPA.

\begin{theorem}
\label{adapt_CPA}
Suppose $f_0\in  \Sigma(\alpha, [0,1))$, $\alpha>0$ and $f_0\geq\rho$ for some $\rho>0$. Let us endow $f$ with the following hierarchical prior which we write $\Pi$ 
\begin{align*} l &\sim \Pi_L \\
f\given l &\sim CPA(\xi(l,n),l,(a_i)_{0<i\leq l}, \log n) \end{align*}
such that for $l>0$, $\Pi_L$  and the sequence $(a_l)_{l\in \N^*}$ satisfy, for some $\beta>0,R\geq1, \delta>0$,
\[\Pi_L \left[\{l\}\right]\propto 2^{-l2^l} \qquad\text{and}\qquad a_l \in\Bigg[\delta n^{-\beta};R\Bigg].\]Then, for $M>0$ depending on $\rho, \alpha$, $\norm{f_0}_{\Sigma(\alpha)}$, $\beta$ and $R$, and $d(f,g)=\norm{f-g}_1$ or $d(f,g)=h(f,g)$, as $n\rightarrow\infty$, \[\mathbb{E}_{f_0}\Pi\Bigg[d(f_0,f)>M\left(\frac{\log n}{n}\right)^{\frac{\alpha}{2\alpha +1}}\given X\Bigg]\rightarrow0.\]
\end{theorem}

The CPA prior is a prior on spline densities which involves a randomized step to define the sample near the frontier of $\Omega$. It is also possible to apply instead a deterministic correction to the infinite forest after the aggregation step. We now define the Spline Pólya Tree (SPT) prior which does so and further highlights the link between forests of Pólya Trees and spline priors.

Let's assume that $g\sim TPT_L\left(\mathcal{A}\right)$ for some $L>0$. Then, as $g$ has support on $\Omega$, we extend it on $\mathbb{R}$ by $1$-periodicity, giving rise to the application $\Tilde{g}$. For $m\geq0$, let's define the map $A^1_{m, 2^{-L}}$, which operates as a smoothing/aggregation of $g$, such that
\begin{equation}\label{first_step_map} A^1_{m, 2^{-L}}(g) \coloneqq  \restr{\Tilde{g}_{\infty,2^{-L}}^{m} }{[0;1)}. \end{equation}We now define the correction of this infinite forest. Following Lemma \ref{map_to_splines}, since $\Tilde{g}$ is a piecewise constant map with breaks at dyadic numbers $k2^{-L}$, $k\in \mathbb{Z}$, $A^1_{m, 2^{-L}}(g)$ is a spline function of order $m +1$ and knots $\left(k2^{-L}\right)_{0\leq k\leq2^{L}}$ (more about this in Section \ref{proofs_theo}). Therefore, there exists polynomials $P_1,\ P_2$ of degree $m$ such that, for $u\in\left[m2^{-L}; (m +1)2^{-L}\right)$ and $v\in \left[1- (m +1)2^{-L}; 1 - m2^{-L}\right)$ \[A^1_{m 2^{-L}}(g)(u) = P_1(u) \text{ and } A^1_{m, 2^{-L}}(g)(v) = P_2(v).\]
As seen with CPA, the samples $A^1_{m, 2^{-L}}(g)$ gives good estimates of a density on the interior of $\Omega$ but not near the frontier. An idea could be to modify the samples using only the information from $ A^1_{m, 2^{-L}}(g)$ away from the frontier, so that we then define the map

\begin{equation}
\label{sd_map_def}
    A^2_{m, 2^{-L}}(g)(x)\coloneqq 
\begin{cases}
    A^1_{m, 2^{-L}}(g)(x),& \text{if } x\in \left[m2^{-L}; 1-m2^{-L}\right), \\
    P_1(x),              & \text{if } x< m2^{-L},\\
    P_2(x),              & \text{if } x\geq 1-m2^{-L}.\\
\end{cases}
\end{equation}
It is a spline function with no undesired continuity/cyclicality property. Finally, for $\tau>0$, the following map is a density function
\begin{equation}\label{sd_map_def2} \Tilde{f}= \frac{A^2_{m, 2^{-L}}(g)_++\tau}{\int_0^1\left(A^2_{m, 2^{-L}}(g)_+(t)+\tau\right)dt}. \end{equation}
To sum up the above construction, if we write $SD_{^{\tau, m, 2^{-L}}}$ the application that associates the function $\Tilde{f}$ to $g$, our SPT prior written $SPT\left( m ,L, \mathcal{A}, \tau\right)$, is the image prior of a $TPT_L\left(\mathcal{A}\right)$ prior by this map. Once again, this construction leads to adaptive (almost-)optimal contraction rate, for any arbitrary Hölder regularity, as is shown in the following theorem. However, the definition of samples near the frontier are less flexible than in CPA, which leads to an additional $\log$ factor in the rate. Still, it is not clear whether it is a shortcoming of $SPT$, or simply a byproduct of our proof.

\begin{theorem}
\label{Adapt_smooth_theorem}
Under the same assumptions on $f_0$ as in Theorem  \ref{adapt_CPA}, for  $\tau_n = \sqrt{n}^{-1}$, let's endow $f$ with the following hierarchical prior which we write $\Pi$
\begin{align*} l &\sim \Pi_L \\
f\given l &\sim SPT\left(\xi(l,n) , l,\mathcal{A}, \tau_n\right) \end{align*}
such that, for $l>0$, $\Pi_L$  and the sequence $(a_l)_{l\in \N^*}$ satisfy, for some $\beta>0,R\geq1, \delta>0$,
\[\Pi_L \left[\{l\}\right]\propto 2^{-l^{3/2}2^l} \qquad\text{and}\qquad a_l \in\Bigg[\delta n^{-\beta};R\Bigg].\]
Then, for $M>0$  depending on $\rho, \alpha$, $\norm{f_0}_{\Sigma(\alpha)}$, $\beta$ and $R$ large enough and $d(f,g)=\norm{f-g}_1$ or $d(f,g)=h(f,g)$, as $n\rightarrow\infty$, \[\mathbb{E}_{f_0}\Pi\Bigg[d(f_0,f)>Mn^{-\frac{\alpha}{2\alpha +1}}\log^{\frac{\alpha}{2\alpha +1}+1/2} n\given X\Bigg]\rightarrow0.\]
\end{theorem}

\section{Discussion.}

One main take-away message is that a well-chosen histogram prior (in the form of a forest) achieves adaptation to arbitrary regularities $\alpha>0$. We have shown that aggregating in a suitable way single (truncated) Pólya trees to define a forest prior allows the induced Hellinger and $L^1$ posterior contraction rates to be optimal on Hölder balls of densities of arbitrary smoothness coefficient. It then bypasses the apparent limitation to $\alpha\leq 1$ of single tree-based priors in previous works. This result highlights the benefits that aggregation operations have for Bayesian estimators. This also improves on previous results in the literature on forests estimators that assumed either $\alpha\leq2$ or $\alpha$ fixed.

This work is a new step in the understanding of the theoretical behavior of Bayesian forest estimators. As noted above, it is still a "toy" model as it involves some simplification, in comparison with usual forest methods such as BART. Here, despite its definition as a sum-of-tree prior, the different trees in the forest are almost the same Pólya tree sample, with the difference that  their underlying partitions are deterministically shifted. Whether it is possible to obtain results similar to those in this document with aggregation schemes and Bayesian forests that are more general is a matter for further investigation. A first extension would be to allow the shifts to be random. Or, one could allow the tree components to be defined on deterministic partitions of the sample space but with different values in corresponding cells across the trees. Our results seem to pertain more particularly to priors on forests of many well-correlated trees, and it is natural to try to lighten this imposed correlation. Doing so, we would obtain priors closer to those used in practice.

Here we focused on the density model, but the ideas developed should more generally apply to other settings, e.g., nonparametric regression, which is left for further work. The framework of Theorem \ref{GGV_theo} limits our analysis of the Bayesian density estimation to Hellinger and $L^1$ rates. Other distances require more complex arguments involving additional technicalities. However, we mention that this framework extends to the $L^2$ distance in regression models. Then, we expect our construction and the rates we obtained to apply to these models as well.

Another intriguing research direction is extending this work to higher dimensions, where the sample space is instead $[0;1)^d$, with $d$ potentially large. Though Pólya Trees in this setting exists, the definition of forest aggregations is not straightforward. It is necessary to define the partition shifts carefully so that the ’limiting continuous’ prior on ’continuous forests’ enjoys good properties. We see in the next section that the elements of a B-spline basis appear naturally in the limit with our aggregation of piecewise constant maps. In higher dimensions, with well-defined shifts, it should be possible to recreate a tensor-product basis so that our analysis applies as well.

Also, we have left aside the question of the computation of the posterior induced by our new priors. It was not the primary purpose of this article, and consequently, we do not investigate this further here. Even though the conjugacy property of single Pólya Trees is lost through aggregation, usual methods such as MCMC should apply. Although we discussed its shortcomings, a simplified prior with no treatment near the frontier of $\Omega$ has the convenient upside that its posterior is explicit. To illustrate the behavior of such a Bayesian forest, we present some numerical experiments based on this prior in Section \ref{simulations} of the Appendix. 

Finally, in the present work, we adapt ideas developed in the study of frequentist estimators in \cite{2014arXiv1407.3939A} to obtain optimal Bayesian posterior rates for arbitrary regularities. Yet, extending these new optimality results to frequentist random forests is not straightforward. Interestingly, the Bayesian framework seems more conducive to developing methods that attain optimal rates for high regularities: for instance, Bayesian mixtures of Gaussians can adapt to arbitrary regularities (see \cite{MR2735885}), even though the Gaussian kernel has order $1$. This contrasts with frequentist Gaussian kernel estimators, which are suboptimal for higher regularities. Therefore, Bayesian forest posteriors could also have an advantage over frequentist random forests.

%% file: proofs.tex
\section{Proofs}
\label{proofs_theo}

\subsection{Link with spline spaces}
\label{spline_introduction}

First, for $A$ a real interval,  we denote $\Pi_{k, \mathbf{t}}(A)$ the space of splines of order $k$ and knot sequence $\mathbf{t}=(t_i)_{i\in I} $, with $I\subset \mathbb{Z}$ such that $\forall i,\ t_i\in A\subset \mathbb{R}$. Also, we assume $i\leq j\implies t_i\leq t_j$ as well as, $\inf \left\{t_i\given i\in I\right\}= \inf A$ and $\sup \left\{t_i\given i\in I\right\}= \sup A$. $\Pi_{k, \mathbf{t}}$ is the subset of maps in $\mathcal{C}^{k-2}(A)$ whose restriction on any interval of the form $[t_i; t_{i+1}[$ is a polynomial of degree strictly smaller than $k$. It coincides with the span of B-splines $B_{j,k}$ of order $k$ and knots $\mathbf{t}$ (see \cite{de1986b} or \cite{fundamentalsBNP} for definition and more details on this), with $j=1,\dots,k+\#\mathbf{t}-1$ if $\#\mathbf{t}<\infty$, $j\in\Z$ otherwise (dropping the dependence on the sequence of $t_j$'s as there will be no ambiguity in the following).
Now, as shown in \cite{aubin2011applied}, for $L\in\N$, a real sequence $\left(u_i\right)_{i\in\Z}$ and the piecewise constant map

\begin{align}
\begin{split}
\label{step_func_can}
  \widetilde{H}_L \colon \mathbb{R} &\to \mathbb{R}\\
  x &\mapsto \sum_{j\in \mathbb{Z}} u_j H_{Lj},
\end{split}
\end{align}
we write for $m\in\mathbb{N}$, by linearity and from \eqref{convol_cont_aggreg},

\begin{align}
\label{spline_new_def}
\left(\widetilde{H}_L\right)_{\infty,2^{-L}}^{m}(x) &=  \sum_{j\in \mathbb{Z}} 2^{L} u_j      \chi^{*m}(2^L\cdot)*  H_{Lj}(x) \nonumber\\
&=  \sum_{j\in \mathbb{Z}} u_j  2^{2L}    \chi^{*m}(2^{L}\cdot)*  \chi(2^{L}\cdot - j)(x) \nonumber\\
&= \sum_{j\in \mathbb{Z}} u_j  2^{2L}  \int_\R  \chi^{*m}(2^{L}s) \chi\left(2^{L}(x-s) - j\right) ds\nonumber \\
&= \sum_{j\in\mathbb{Z}} u_j 2^{L} \chi^{*(m+1)}\left(2^{L}x-j\right).
\end{align}
Also, as shown in \cite{de1986b} (Section 10), the maps $\{\chi^{*(m+1)}(\cdot-i), i\in\mathbb{Z}\}$ are the Cardinal splines of order $m+1$, i.e., the B-Splines of order $m+1$ corresponding to the biinfinite knot sequence $\mathbf{t}=\mathbb{Z}$. Hence, it is a basis for $\Pi_{m+1, \mathbb{Z}}(\mathbb{R})$ and similarly, one shows that $\{2^L\chi^{*(m+1)}(2^L\cdot-i)\given i\in\mathbb{Z}\}$ is a basis for $\Pi_{m+1, 2^{-L}\mathbb{Z}}(\mathbb{R})$. Therefore, $(\widetilde{H}_L)_{\infty,2^{-L}}^{m}$ from \eqref{spline_new_def} belongs to this linear space and $\left(u_i\right)_{i\in \mathbb{Z}}$ is the sequence of its coordinates in this basis.

We now remark that the map from \eqref{spline_new_def} is $1$-periodic if and only if $\left(u_i\right)_{i\in \mathbb{Z}}$ is a $2^L$-periodic sequence (see Lemma \ref{period_splines_seq}). It follows that $\Tilde{\Pi}_{m+1, 2^{-L}\mathbb{Z}}(\mathbb{R})\coloneqq\Pi_{m+1, 2^{-L}\mathbb{Z}}(\mathbb{R})\cap \left\{f: \mathbb{R} \to \mathbb{R}, \ f\ 1-\text{periodic}\right\}$ is a linear space with basis
\begin{equation}\label{period_spline_bas_def} \left\{S_{i,2^L,m} \coloneqq 2^L \sum_{j\in\mathbb{Z}}\chi^{*(m+1)}(2^L\cdot-2^Lj-i)\given 0\leq i\leq 2^L-1\right\}.\end{equation}

In Section \ref{approx_theory}, we see that this space has good approximation properties, which we use below in the proofs. Indeed, for each Hölderian density function on $\Omega$, it contains an element whose restriction to $\Omega$ is sufficiently close to the density function in the interior of the interval. We also stress that, via a modification near the frontier of $\Omega$, we can recover a spline function that approximates the density well on the whole interval.

This observation accounts for the performance of our priors. Looking back at the algorithmic definitions of DPA and CPA priors in Section \ref{DPA_subsection}, and setting aside the discretization in DPA, it is possible to interpret them as follows. The coordinates of an element $\Tilde{\Pi}_{m+1, 2^{-L}\mathbb{Z}}(\mathbb{R})$ are sampled via a TPT distribution, and uniform random variables transform it into a similar spline from $\Pi_{m+1, 2^{-L}\mathbb{Z}}(\mathbb{R})$ (or rather a restriction to $\Omega$). This point of view underlies the proofs in the following subsections.

\subsection{Proofs of main results}
In this subsection, we provide the proofs for the theorems presented in the last section. The adaptive results involve the derivation of intermediary points that we first demonstrate in the proof of non-adaptive results (for fixed regularity). Therefore, we first analyze the case of fixed regularities for CPA and DPA before delving into the proofs of adaptive results. Our arguments rely on Theorem \ref{GGV_theo}, whose conditions are investigated in lemmas following it. These lemmas build on the approximation properties of spline functions and their parametric representations (see Section \ref{approx_theory}).  Also, we primarily focus on CPA, as the similar properties of DPA only then require the control of additional discretization terms. 

\begin{proof}[Proof of Theorem \ref{Master_theorem2} (and extension to CPA for fixed regularity)]

It is sufficient to verify that the conditions of Theorem \ref{GGV_theo}  are satisfied. One shows that the prior puts sufficient mass in some Kullback-Leibler neighborhoods of the true density. We use results in Approximation Theory (see Lemmas  \ref{approx_cyclicSplines} and \ref{KL_cont_unif}) that we develop in Section \ref{approx_theory}. Besides, one also has to prove that the priors allocate most of their mass to subsets of limited complexity. It ensues from the priors generating draws that belong to spaces that resemble spaces of splines, whose dimensions are not too large (see Section \ref{spline_introduction} and Lemma \ref{Hell_unif_cont}). In the following, we write  $m=\lfloor\alpha\rfloor$ and, for $c_0$ a constant to be defined below,
\begin{equation}\label{rate_eps}\epsilon_n= c_0\Bigg(\frac{\log n}{n}\Bigg)^{\frac{\alpha}{2\alpha +1}}. \end{equation}
\textit{1) Complexity of the prior: } Let us define 
\begin{align}\label{sieve_CPA} \begin{split}\mathcal{H}_{m, l} \coloneqq \Bigg\{&(\theta_i)_{i\in\Z}\in\R_+^\Z,\quad \forall i\in \Z,\ \theta_i = \theta_{i+m+2^{l}},\\
&\quad \sum_{i=0}^{2^{l}-m-1}\theta_i + \sum_{i=2^{l}-m}^{2^{l}-1}\Big(\omega_{m,2^{l}-i}\theta_i+\omega_{m,i-(2^{l}-m-1)}\theta_{i+m}\Big)=1\Bigg\} \end{split}.\end{align}
According to Lemma \ref{density_prior_int} and the discussion following it, the following sets have probability $1$ under the DPA prior (resp. CPA):
\begin{align}
\label{sieves_definition}
\mathcal{F}_n&\coloneqq \left\{\restr{\frac{f^{m}_{n^,2^{-L_n}}}{\int_0^1 f^{m}_{n,2^{-L_n}}(t)dt}}{[0;1)},\quad f=\sum_{i\in\Z} \theta_i H_{L_ni},\ (\theta_i)_{i\in\Z}\in  \mathcal{H}_{m, L_n} \right\}\\
&\left(\text{resp. }\coloneqq \Bigg\{\restr{f^{m}_{\infty,2^{-L_n}}}{[0;1)},\quad f=\sum_{i\in\Z} \theta_i H_{L_ni},\ (\theta_i)_{i\in\Z}\in  \mathcal{H}_{m, L_n}\Bigg\}\right).\nonumber
\end{align}
In \eqref{omegas_def}, the positivity of $\chi$ ensures that $\inf_{1\leq l\leq m+1} \omega_{m,l}=\omega_{m,1}=1/(m+1)!$ (see Proposition 6.7.1., p.136, in \cite{aubin2011applied}). So, any sequence in $\mathcal{H}_{m, n}$ has its coordinates bounded by $(m+1)!$ because of their positivity and the constraint from the definition. Then, from Lemma \ref{Hell_unif_cont} with $M=(m+1)!$, $q=n$ and $L=L_n$, there exists an absolute constant $C$ such that for $n$ large enough and $B_{\mathbb{R}^{D}}(0, r)$ the $L^2$ closed ball in $\R^D$ of radius $r$,
\begin{align*}\begin{split}
 N\left(\epsilon_n, \mathcal{F}_n, h\right)&\leq N\left(C\left(2^{L_n}+m\right)^{-1/2} \epsilon_n^2, \mathcal{H}_{m, n}, \norm{\cdot}_2 \right)\\
 &\leq N\left(C\left(2^{L_n}+m\right)^{-1/2} \epsilon_n^2, B_{\R^{2^{L_n}+m}}\big(0,\sqrt{2^{L_n}+m}\omega_{m,1}^{-1}\big), \norm{\cdot}_2 \right).
\end{split}\end{align*}
The first inequality is valid in the discrete case since the remainder term from Lemma \ref{Hell_unif_cont} is of the order $o(\epsilon_n)$ for our values of $L_n$ and the number of trees $q=n$. It is a general fact that there exists a universal constant $C>0$ such that \[N\left(\delta, B_{\mathbb{R}^{K}}(0, M), ||.||_2\right)\leq \left(\frac{CM}{\delta}\right)^K.\]Therefore, one concludes that, for any $D>c_0^{-1}$ and $n$ large enough
\begin{equation}\label{comp_siv}  N\left(\epsilon_n, \mathcal{F}_n, h\right)\leq \left(C\omega_{m,1}^{-1}\frac{2^{L_n}+m}{\epsilon_n^2}\right)^{2^{L_n}+m}\leq e^{Dn\epsilon_n^2}.\end{equation}
\textit{2) Small ball probability condition: }For the last condition of Theorem \ref{GGV_theo}, we introduce the sequence $\left(\eta_i\right)_{0\leq i\leq 2^{L_n}+m-1}$ from Lemma \ref{approx_cyclicSplines} such that
\[  1= \sum_{i=m}^{2^{L_n}-1} \eta_i +\sum_{i=0}^{m-1}  \left(\omega_{m,i+1}  \eta_i +(1-\omega_{m,i+1}) \eta_{2^{L_n}+i} \right).\] 
We define $\Tilde{\eta}_i=\eta_i$ for $i=m,\dots,2^{L_n}-1$ and  for $i=0,\dots,m-1$,
\begin{equation}\label{intro_tilde_true}
\Tilde{\eta}_i=\Tilde{\eta}_{2^{L_n}+i}=\omega_{m,i+1}\eta_i + (1-\omega_{m,i+1})\eta_{2^{L_n}+i}=(1-\omega_{m,m-i})\eta_i + \omega_{m,m-i}\eta_{2^{L_n}+i},
\end{equation}this being consistent according to Lemma \ref{symmetry_convol}. This guarantees that $\left(\Tilde{\eta}_i\right)_{0\leq i\leq 2^{L_n}-1}\in S^{2^{L_n}}$.
First, from Lemma \ref{KL_cont_unif}, for $n$ and $c_0$ large enough and $C$ small enough, depending on $\rho, \alpha$ and $\norm{f_0}_{\Sigma(\alpha)}$, we have the inequality 
\[ \Pi \left[B_{KL}\Bigg(f_0,\epsilon_n\Bigg)\right] \geq \Pi \Bigg[ \underset{0\leq i \leq 2^{L_n}+m-1}{\max}\left| u_{i-m}-\eta_i \right| \leq C\epsilon_n2^{-L_n} \Bigg] \] 
following from the fact that the terms depending on $L=L_n$ and the number of 'trees' $q=n$ in Lemma \ref{KL_cont_unif} are of order $O(\epsilon_n^2)$.
One controls the different random variables from the sequential definition of our prior in Step \ref{s1} so that one obtains a lower bound on the above event probability.
With notation from Part \ref{DPA_subsection} and $\iota(i)\equiv i+m \mod 2^{L_n}$, for any $r>0$, on the event
\[ \Big\{ \underset{0\leq i \leq 2^{L_n}-1}{\max}\left| \Theta_i-\Tilde{\eta}_{\iota(i)} \right| \leq  \frac{\omega_{m,1}^2}{8}r\leq  \frac{\omega_{m,1}}{8}r\ \text{ and }\  \underset{2^{L_n}-m \leq i \leq 2^{L_n}-1}{\max}\left| \theta_i-\eta_{i+m} \right| \leq \omega_{m,1}\frac{r}{4} \Big\},\]
we have that $ \underset{m\leq i \leq 2^{L_n}+m-1}{\max}\left| u_{i-m}-\eta_i \right| \leq r$ and, using the periodicity of $(u_i)_{i\in\Z}$, for $ i=0,\dots,m-1$,
\begin{align*}
\left| u_{i-m}-\eta_i \right|&=\left| u_{i+2^{L_n}}-\eta_i \right|\\
&=\left| \theta_{i+2^{L_n}}-\eta_i \right|\\
&=\left| \frac{\Theta_{i+2^{L_n}-m}-\omega_{m,m-i}\theta_{i+2^{L_n}-m}}{1-\omega_{m,m-i}} -\frac{\Tilde{\eta}_i-\omega_{m,m-i}\eta_{2^{L_n}+i}}{1-\omega_{m,m-i}}  \right|\\
&\leq \left(1-\omega_{m,m-i}\right)^{-1}\left(\frac{\omega_{m,1}}{8}r+\omega_{m,m-i}\omega_{m,1}\frac{r}{4}\right)\\
&\leq r(1/8+\omega_{m,m-i}/4)\leq r.
\end{align*}
This ultimately implies that $\underset{0\leq i \leq 2^{L_n}+m-1}{\max}\left| u_{i-m}-\eta_i \right| \leq r$ since $\omega_{m,1}\leq \omega_{m,l}\leq 1$ for $l\geq 1$. 
Therefore, it remains to study the factors in the lower bound
\begin{align*}
\begin{split}
&\Pi \left[B_{KL}\Bigg(f_0,\epsilon_n\Bigg)\right] \geq \Pi \Bigg[ \underset{0\leq i \leq 2^{L_n}-1}{\max}\left| \Theta_i-\Tilde{\eta}_{\iota(i)} \right| \leq C\frac{\omega_{m,1}^2}{8}\epsilon_n2^{-L_n} \Bigg] \times\\
&\qquad\Pi\Bigg[ \underset{2^{L_n}-m \leq i \leq 2^{L_n}-1}{\max}\left| \theta_i-\eta_{i+m} \right| \leq C\omega_{m,1}\frac{\epsilon_n}{4}2^{-L_n}\Bigg|\underset{0\leq i \leq 2^{L_n}-1}{\max}\left| \Theta_i-\Tilde{\eta}_{\iota(i)} \right| \leq C\frac{\omega_{m,1}^2}{8}\epsilon_n2^{-L_n} \Bigg].
\end{split}
\end{align*}
This translates the fact that, to obtain a good prior mass, it is sufficient to control the mass of the TPT so that the associated forests are close to the density $f_0$ on the interior of $\Omega$, and then control the behavior near the frontier to extend the result to the whole of $\Omega$.

Now, for $0\leq i\leq 2^{L_n}-1$, we decompose $\Theta_i = \prod_{j=1}^{L_n} Y_{\kappa(L_n,i)^{[j]}}$ and $\Tilde{\eta}_i= \prod_{j=1}^{L_n} y_{\kappa(L_n,i)^{[j]}} $ where, for $1\leq j<L_n$,\[ y_{\kappa(L_n,i)^{[j]}0} \coloneqq \frac{\sum_{s, \lfloor s2^{j+1-L_n} \rfloor_{\text{f}} = \lfloor i2^{j+1-L_n} \rfloor_{\text{f}}}\ \Tilde{\eta}_s}{\sum_{s, \lfloor s2^{j-L_n} \rfloor_{\text{f}} = \lfloor i2^{j-L_n} \rfloor_{\text{f}}}\Tilde{\eta}_s},\qquad y_{\kappa(L_n,i)^{[j]}1}  \coloneqq 1-y_{\kappa(L_n,i)^{[j]}0} \]belong to $[0;1]$. Also, $y_{0}$ and $y_{1}$ satisfy the same formula, with $j=0$. Withh $e_j  = Y_{\kappa(L_n,i)^{[j]}}$ and $ t_j= y_{\kappa(L_n,\iota(i))^{[j]}}$ for sake of clarity, then, as for all $j=1,...,L_n,\ |e_{j}|\leq 1 \text{ and }|t_{j}|\leq 1$, we have

\begin{align*}
\left| \Theta_i-\Tilde{\eta}_{\iota(i)} \right|&= \left|\sum_{j=1}^{L_n} e_1...e_{j-1}(e_j-t_j)t_{j+1}...t_{L_n}\right|\\
&\leq \sum_{j=1}^{L_n} \left|e_1...e_{j-1}(e_j-t_j)t_{j+1}...t_{L_n}\right|\\
&\leq \sum_{j=1}^{L_n} \left|e_j-t_j\right|.
\end{align*}
This finally gives us that, using that the $Y$'s variables are independent in the TPT,

\begin{align*}
\begin{split}
\Pi \left[ \underset{0\leq i \leq 2^{L_n}-1}{\max}\big| \Theta_i-\Tilde{\eta}_{\iota(i)} \big| \leq \frac{C\omega_{m,1}^2\epsilon_n}{8*2^{L_n}} \right]&\geq \Pi\left[ \bigcap_{\substack{0\leq i\leq 2^{L_n},\\ 1\leq j\leq L_n}}\Big\{\big|Y_{\kappa(L_n,i)^{[j]}} - y_{\kappa(L_n,\iota(i))^{[j]}} \big|\leq  \frac{C\omega_{m,1}\epsilon_n^2}{8L_n2^{L_n}}  \Big\}\right]\\
&= \prod_{j=1}^{L_n} \prod_{|\kappa|=j-1} P_{X\sim \text{Beta}(a_j,a_j)}\left[|X-y_{\kappa0}|\leq C\frac{\omega_{m,1}^2}{8L_n}\epsilon_n2^{-L_n} \right].
\end{split}
\end{align*}
Let's write $\xi_n=C\frac{\omega_{m,1}^2}{8L_n}\epsilon_n2^{-L_n} $. Since for any $j$, $a_j\Gamma(a_j)=\Gamma(a_j+1)\leq \Gamma(R+1)\eqqcolon\Tilde{R}$ and $\Gamma$ is lower bounded by some constant $\psi>0$ on the set of real positive numbers, $\Gamma(2a_j)\Gamma(a_j)^{-2}\geq \psi a_j^2 \Tilde{R}^{-2}\geq \psi \Tilde{R}^{-2} \delta^{2}n^{-2\beta}$. Also, for $n$ large enough, if $R\geq a_j>1$, 
\[ \int_{(y_{\kappa0}-\xi_n)\vee 0}^{(y_{\kappa0}+\xi_n)\wedge 1} t^{a_j-1}(1-t)^{a_j-1}dt\geq \int_{0}^{\xi_n} t^{a_j-1}(1-t)^{a_j-1}dt\geq \left(1-\xi_n\right)^{R-1}\frac{\xi_n^R}{R}\geq \frac{\xi_n^R}{2^{R-1}R}, \]while, for $a_j\leq1$, the bound can just be replaced by $\xi_n$ when $n$ is large enough. Finally, for some $C>0$, depending on $\beta$, $R$, and $c_0$,

\begin{equation}
\label{control_proba_proba}
\begin{array}{rcl}
\Pi \left[ \underset{0\leq i \leq 2^{L_n}-1}{\max}\big| \Theta_i-\Tilde{\eta}_{\iota(i)} \big| \leq \frac{C\omega_{m,1}^2\epsilon_n}{8*2^{L_n}} \right] &\geq& \prod_{j=1}^{L_n} \prod_{|\kappa|=j-1} \frac{\psi\delta^{2} }{\Tilde{R}^2} \left(\frac{1}{2^{R-1}R} \wedge 1\right) n^{-2\beta} \xi_n^{R\vee 1} \\
&=& \left(\frac{\psi\delta^2 }{\Tilde{R}^2}  \big(\frac{1}{2^{R-1}R} \wedge 1\big) n^{-2\beta}\xi_n^{R\vee 1} \right)^{2^{L_n}-1}\\
&\geq& e^{-Cn\epsilon_n^2}.
\end{array}
\end{equation}
Then, for $U_n$ large enough, i.e. $n$ large enough, for any $2^{L_n}-m \leq i \leq 2^{L_n}-1$, the uniform random variables $\theta_i$ verify for any $\Tilde{C}>0$
\begin{equation*}
\begin{array}{rcl}
\Pi\Bigg[ \left| \theta_i-\eta_{i+m} \right| \leq C\omega_{m,1}\frac{\epsilon_n}{4}2^{-L_n} &\Bigg|&     \underset{0\leq i \leq 2^{L_n}-1}{\max}\left| \Theta_i-\Tilde{\eta}_{\iota(i)} \right| \leq C\frac{\omega_{m,1}^2}{8}\epsilon_n2^{-L_n} \Bigg]\\
&\geq& \left(\frac{\Theta_i}{\omega_{m,\iota(i)}}\right)^{-1}C\omega_{m,1}\frac{\epsilon_n}{8}2^{-L_n}\\
&\geq& C\omega^2_{m,1}\epsilon_n2^{-L_n}/8\geq e^{-\Tilde{C}n\epsilon_n^2}.
\end{array}
\end{equation*}
The first inequality is due to $0\leq \eta_{i+m} \leq \omega_{m,\iota(i)+1}^{-1}\Tilde{\eta}_{\iota(i)}\leq \omega_{m,\iota(i)+1}^{-1}\left(\Theta_i+C\omega_{m,1}^2\epsilon_n2^{-L_n}/8\right)$ on the conditioning event, which follows by positivity and \eqref{intro_tilde_true}. Finally, we use the conditional independence of the random variables $\theta_i$ to obtain the lower bound,

\begin{equation}\label{ball_prob} 
\Pi \left[B_{KL}\Bigg(f_0,\epsilon_n\Bigg)\right] \geq e^{-(C+m\Tilde{C})n\epsilon_n^2}.
\end{equation}
We now conclude with Theorem \ref{GGV_theo} and equations \eqref{comp_siv} and \eqref{ball_prob}, recalling that $\mathcal{F}_n$ is an almost sure event under our prior, the constant $M>0$ depending on $\rho, \alpha$, $\norm{f_0}_{\Sigma(\alpha)}$, $\beta$ and $R$.
\end{proof}

\begin{proof}[Proof of Theorems \ref{Adapt_smooth_theorem2} and \ref{adapt_CPA}]
We proceed as in the proof of Theorem \ref{Master_theorem2}, with $\epsilon_n$ as in $\eqref{rate_eps}$. The main difference is that the distributions now allocate positive mass to different depth values $L$ so that we adapt the sieves. Below, we take the union of sieves similar to those introduced in the above proof, from low resolution $L=1$ up to some threshold. To this effect, we introduce the sequences of depth $L_{1,n}$ and $L_{2,n}$ such that $2^{L_{i,n}}\asymp C_i \left(\frac{n}{\log n}\right)^\frac{1}{2\alpha+1}$ (i.e., it is the closest integer to the solution of this equation) for some constants $C_1$ and $C_2=1$. The proof then again uses the Theorem \ref{GGV_theo} with an additional term to be controlled, corresponding to the prior mass on the hyperparameter $L$.

\textit{1) Complexity of the prior: } Let, for $0\leq k\leq2^{L_n}-1$,

\begin{align*} \mathcal{I}_{k, n} \coloneqq \Bigg\{&(\theta_i)_{i\in\Z}\in\R_+^\Z,\  \forall i\in \Z,\ \theta_i = \theta_{i+k+2^{L_{1,n}}}\text{ and } 0\leq \theta_i \leq \log n,\\
&\quad \sum_{i=0}^{2^{L_{1,n}}-k-1}\theta_i + \sum_{i=2^{L_{1,n}}-k}^{2^{L_{1,n}}-1}\Big(\omega_{k,2^{L_{1,n}}-i}\theta_i+\omega_{k,i-(2^{L_{1,n}}-k-1)}\theta_{i+k}\Big)=1\Bigg\} \end{align*}
and define the sieves for the DPA prior (resp. CPA)
\begin{align}
\label{sieves_definition2}
\mathcal{F}_n&\coloneqq \bigcup_{l=1}^{L_{1,n}}\left\{f^{\xi(l,n)}_{\infty,2^{-l}}{[0;1)},\ f=\sum_{i\in\Z} \theta_i H_{li},\ (\theta_i)_{i\in\Z}\in  \mathcal{I}_{\xi(l,n), n} \right\}\\
&\left(\text{resp. }\coloneqq \bigcup_{l=1}^{L_{1,n}} \Bigg\{\restr{\frac{f^{\xi(l,n)}_{n\log^3n,2^{-l}}}{\int_0^1 f^{\xi(l,n)}_{n\log^3n,2^{-l}}(t)dt}}{[0;1)},\  f=\sum_{i\in\Z} \theta_i H_{li},\ (\theta_i)_{i\in\Z}\in  \mathcal{I}_{\xi(l,n), n} \Bigg\}\right).\nonumber
\end{align}
In the definition of the prior, the sequence $(u_i)_{i\in\Z}$ lies in $\left[0;\log n\right]^{\Z}$ almost surely for $n$ large enough. Therefore, following Lemma \ref{density_prior_int} and the discussion after its proof, we now have, for $n$ large enough,

\begin{align}
\label{comp_ad2}
\Pi[\mathcal{F}_n^c] = \Pi[l>L_{1,n}] \propto \sum_{l=L_{1,n}+1}^{+\infty} 2^{-l 2^l} &\lesssim 2^{-L_{1,n}2^{L_{1,n}}}\nonumber\\
&=e^{-\log(C_1n/\log n)\frac{2^{L_{1,n}}}{2\alpha+1}}\nonumber\\
&\leq e^{-C_1(c_0^{2}/(4\alpha-2))^{-1}n\epsilon_n^2}.
\end{align}
Also, using Lemma \ref{Hell_unif_cont} with $M=\log n$, $q=v_nn\log^3 n$, $m\leq\xi(1,n)\leq \log(n)/2$ and $l\leq L_{1,n}$, we use similar arguments as the ones preceding \eqref{comp_siv} to derive, for $C,C'$ absolute constants and $D$ depending on $C_1$ and $c_0$,

\begin{align}
\label{couv_ad2}
N(\epsilon_n, \mathcal{F}_n,h)  &\leq  \sum_{l=1}^{L_{1,n}} \left(\frac{C(2^l+\xi(l,n))\log n}{\epsilon_n^2}\right)^{2^l+\xi(l,n)}\nonumber\\
&\leq \sum_{l=1}^{L_{1,n}} \left(\frac{C(2^{L_{1,n}}+\log(n)/2)\log n}{\epsilon_n^2}\right)^{2^l+\log(n)/2}\nonumber\\
&\lesssim  \left(\frac{C(2^{L_{1,n}}+\log(n)/2)\log n}{\epsilon_n^2}\right)^{2^{L_{1,n}}+\log n}\leq e^{C'\log n 2^{L_{1,n}}}\leq e^{Dn\epsilon_n^2}.
\end{align}
In particular, we have used that, with the sequences from the theorem, the term depending on $q$ in Lemma \ref{Hell_unif_cont} is of order $o(\epsilon_n)$.\\
\textit{2) Prior mass condition: }Since  $\xi(L_{2,n},n)=\lfloor\alpha\rfloor$, it is possible to use the same arguments that led to \eqref{ball_prob}, for $n$ large enough, $c_0$ large enough depending on $\rho, \alpha$ and $\norm{f_0}_{\Sigma(\alpha)}$ and $C$ large enough depending on $\beta$, $R$ and $c_0$, to obtain

\begin{align}
\label{ball_ad2}
\Pi\left[B_{KL}(f_0,\epsilon_n)\right] &\gtrsim \Pi\left[B_{KL}(f_0,\epsilon_n)\given l=L_{2,n}\right]2^{-L_{2,n}2^{L_{2,n}}}\nonumber\\
&\geq  e^{-Cn\epsilon_n^2} e^{-(c_0^{2}/(2\alpha-1))^{-1}n\epsilon_n^2}.
\end{align}
Indeed, in the argument invoking Lemma \ref{KL_cont_unif}, the terms controlled with $q=v_nn\log^3 n$ are of order $o(\epsilon_n^2)$.\\

We conclude using Theorem \ref{GGV_theo} along with equations \eqref{comp_ad2},\eqref{couv_ad2} and \eqref{ball_ad2}, since for $C_1$ large enough, $C_1(c_0^{2}/(4\alpha-2))^{-1}> C+(c_0^{2}/(2\alpha-1))^{-1}+4$. Then, the theorem is valid for $M$ large enough, depending on $\rho, \alpha$, $\norm{f_0}_{\Sigma(\alpha)}$, $\beta$ and $R$.
\end{proof}

\begin{proof}[Proof of Theorem \ref{Adapt_smooth_theorem}]

Within this proof, let us set, for $c_0$ to be precised,
\[\epsilon_n= c_0n^{-\frac{\alpha}{2\alpha +1}}\log^{\frac{\alpha}{2\alpha +1}+1/2} n. \] Let's introduce the sequences of depth $L_{1,n}$ and $L_{2,n}$ such that $2^{L_{i,n}}\asymp C_i \left(\frac{n}{\log n}\right)^\frac{1}{2\alpha+1}\log^{1/2} n$ for some constants $C_1$ and $C_2=1$, and introduce the subsets

\[\mathcal{F}_n = \cup_{l=1}^{L_{1,n}} G_{l,  \xi(l,n) , \tau_n}\]
where
\[ G_{l,  k , \tau_n}\coloneqq\left\{	SD_{\tau_n, k, 2^{-l}}(g),\ g=  \sum_{i=0}^{2^l-1}\Theta_i H_{li}, (\Theta_i)_{0\leq i\leq 2^{l}-1}\in S^{2^{l}}	\right\}. \]
\textit{1) Complexity of the prior: }On the one hand, we have 

\begin{align}
\label{comp_ad}
\Pi[\mathcal{F}_n^c] = \Pi[l>L_{1,n}] &\lesssim 2^{-L_{1,n}^{3/2}2^{L_{1,n}}}\nonumber \\
&\leq e^{-\log^{-1/2}2 \log^{3/2}(C_1n/\log n)\frac{2^{L_{1,n}}}{2\alpha+1}}\nonumber\\
&\leq e^{-\frac{C_1c_0^{-2}}{(4\alpha+2)\sqrt{\log 2}}n\epsilon_n^2}.
\end{align}
On the other hand, Lemma \ref{Hellinger_control} implies

\[ N\left(\epsilon_n, G_{l,  \xi(l,n) , \tau_n}, h\right)\leq N\left(C_{l,n}^{-2} \tau_n \epsilon_n^2, \Bigg\{	\sum_{i=0}^{2^l-1}\Theta_i H_{li},\ (\Theta_i)_{0\leq i\leq 2^{l}-1}\in S^{2^{l}}	\Bigg\}, h \right) \]
where $C_{l,n}$ is the multiplicative constant from Lemma \ref{Hellinger_control} when $m= \xi(l,n)$. For $f,g$ in \[\Bigg\{	\sum_{i=0}^{2^l-1}\Theta_i H_{li},\quad (\Theta_i)_{0\leq i\leq 2^{l}-1}\in S^{2^{l}}	\Bigg\},\] we see that
\[h(f,g)=\left(\sum_{i=0}^{2^{l}-1} 2^{-l} (\sqrt{f(i2^{-l})}-\sqrt{g(i2^{-l})})^{2} \right)^{1/2}=  \norm{\sqrt{\mathbf{f}}-\sqrt{\mathbf{g}}}_2
\]
where $\mathbf{f},\mathbf{g}\in [0,1]^{2^{l}}$ are the sequences in $S^{2^{l}}	$ defining $f$ and $g$. It follows that

\[
N\left(\epsilon_n, G_{l,  \xi(l,n) , \tau_n}, h\right)\leq N\left(C_{l,n}^{-2} \tau_n \epsilon_n^2, S^{2^{l}}, \norm{\cdot}_2\right)\leq N\left(C_{l,n}^{-2} \tau_n \epsilon_n^2, B_{\mathbb{R}^{2^{l}}}(0, 1),\norm{\cdot}_2\right)\leq  \left(\frac{C}{C_{l,n}^{-2} \tau_n \epsilon_n^2}\right)^{2^{l}}.
\]
This finally gives, since $\xi(l,n)\leq\xi(1,n)\leq \log(n)/2$, for $C'$ an absolute constant and using the explicit formula for $C_{l,n}$ from Lemma \ref{Hellinger_control},
\begin{align}
\label{couv_ad}
N(\epsilon_n, \mathcal{F}_n,h)  &\leq \sum_{l=1}^{L_{1,n}} N(\epsilon_n, G_{l,  \xi(l,n) , \tau_n}, h) \nonumber\\
&\leq \sum_{l=1}^{L_{1,n}} \left(\frac{4C\left(1+\sqrt{1+2\left( \xi(l,n) +1\right)^3 e^{\sqrt{6\left( \xi(l,n) +1\right)} \xi(l,n)}}\right)^2}{\tau_n \epsilon_n^2}\right)^{2^l} \nonumber\\
&\leq  \sum_{l=1}^{L_{1,n}} \left(\frac{8C\left(1+\left( \xi(l,n) +1\right)^3 e^{\sqrt{6\left( \xi(l,n) +1\right)} \xi(l,n)}\right)}{\tau_n \epsilon_n^2}\right)^{2^l} \nonumber\\
&\leq \sum_{l=1}^{L_{1,n}} \left(\frac{9Ce^{C'\log^{3/2} n}\log^{3} n}{\tau_n \epsilon_n^2}\right)^{2^l} \nonumber\\
&\lesssim \left(\frac{C n^{C'\sqrt{\log n}}\log^{3} n}{\tau_n\epsilon_n^2}\right)^{2^{L_{1,n}}+1}\nonumber\\
&\leq n^{2C'\sqrt{\log n}2^{L_{1,n}}}\leq e^{2C'C_1c_0^{-2}n\epsilon_n^2}.
\end{align}
\textit{2) Prior mass condition: }
Lemma \ref{KL_ball_control} ensures the existence of a sequence $\left(\eta_i\right)_{0\leq i\leq 2^{L_{2,n}}-1}\in S^{2^{L_{2,n}}}$ such that, with $\Theta_i$ the sequence drawn by the TPT distribution as in \eqref{TPT_proba_expr}, $\underset{0\leq i \leq 2^{L_{2,n}}-1}{\max} \left| \eta_{i} - \Theta_{i}\right| \leq \left(\log n\right)^\frac{\alpha+2}{2\alpha+1}n^{-\frac{3\alpha+3}{2\alpha+1}}$ and

\[ K\left(f_0, SD_{\tau_n, \lfloor \alpha \rfloor, 2^{-L_{2,n}}}\Bigg(\sum_{i=0}^{2^{L_{2,n}}-1} \Theta_i H_{L_{2,n}i}\Bigg)\right)\vee V\left(f_0, SD_{\tau_n, \lfloor \alpha \rfloor, 2^{-L_{2,n}}}\Bigg(\sum_{i=0}^{2^{L_{2,n}}-1} \Theta_i H_{L_{2,n}i}\Bigg)\right)\]
is smaller than $\epsilon_n^2$ if $c_0$ is large enough, depending on $\rho,\ \alpha$, $\norm{f_0}_{\infty}$ and $\norm{f_0}_{\Sigma(\alpha)}$. Indeed, under this condition, every term in the lemma depending on $L=L_{2,n}$ and $\tau=\sqrt{n}^{-1}$ is of the right order. Consequently,
\begin{equation*}
\label{prob_ball}
\Pi \left[B_{KL}\Bigg(f_0,\epsilon_n\Bigg) \Big| l=L_{2,n} \right] \geq \Pi \left[ \underset{0\leq i \leq 2^{L_{2,n}}-1}{\max} \Big| \eta_{i} - \Theta_{i}\Big| \leq  \Big(\log n\Big)^\frac{\alpha+2}{2\alpha+1}n^{-\frac{3\alpha+3}{2\alpha+1}} \Big| l=L_{2,n} \right].
\end{equation*}
The same arguments underlying \eqref{control_proba_proba} then ensures that, for some $C>0$, depending on $\beta$, $R$, $\alpha$ and $c_0$, $\Pi \left[B_{KL}\big(f_0,\epsilon_n\big) \Big| l=L_{2,n} \right]\geq e^{-Cn\epsilon_n^2}$. Therefore,

\begin{align}\label{ball_ad}
\Pi\left[B_{KL}(f_0,\epsilon_n)\right] &\gtrsim \Pi\left[B_{KL}(f_0,\epsilon_n)|l=L_{2,n}\right]2^{-L_{2,n}^{3/2}2^{L_{2,n}}}\nonumber\\
&\geq e^{-Cn\epsilon_n^2} e^{-\log^{3/2} n2^{L_{2,n}}/\sqrt{\log 2}}\nonumber\\
&= e^{-\left(C+\frac{1}{c_0^2\sqrt{\log 2}}\right)n\epsilon_n^2}.
\end{align}

We conclude using Theorem \ref{GGV_theo} along with equations \eqref{comp_ad},\eqref{couv_ad} and \eqref{ball_ad} , since for $C_1$ large enough, $\frac{C_1c_0^{-2}}{(4\alpha+2)\sqrt{\log 2}} > C+\frac{1}{c_0^2\sqrt{\log 2}}+4$. Then, the theorem is valid for $M$ large enough, depending on $\rho, \alpha$, $\norm{f_0}_{\Sigma(\alpha)}$, $\norm{f_0}_{\infty}$, $\beta$ and $R$.

\end{proof}

%
%
%
%
%

\subsection{Approximation theory for periodic splines.}
\label{approx_theory}

In the constructions of CPA and DPA distributions, the aggregating operation we have defined transforms a TPT sample into a periodic spline density (or a piecewise constant approximation of it for DPA). It is convenient as these periodic splines have good approximation properties according to the following lemma. It is the result of \eqref{spline_new_def}, Lemmas \ref{period_splines_seq} and \ref{map_to_splines}, and the prior definitions. Below, we prove that it is possible to approximate any Hölder density with such a spline, as long as we focus on an interval far enough from the frontier of $\Omega$. 

In order to extend this result near the frontier of $\Omega$, we also see that we can recover an approximating spline on the whole of $\Omega$ from the periodic spline of Lemma \ref{approx_cyclicSplines}. Consequently, CPA and DPA include a stochastic step to simulate the modifications needed to obtain this last spline density.

In the end, the link with splines explains why the aggregation part of our priors results in the almost optimal contraction rates that we obtained in Section \ref{results}.

\begin{lemma}
\label{approx_cyclicSplines}
Suppose $m+1 \geq \alpha > 0$ and $L\geq1$. There exists a constant $C$ depending only on $m$ and $\alpha$ such that for every $f_0 \in \Sigma(\alpha, [0,1) )$, there exists $g\in \Tilde{\Pi}_{m+1, 2^{-L}\mathbb{Z}}(\mathbb{R})$ such that
\[ \norm{ \restr{f_0}{ \left[2^{-L}m;1-2^{-L}m\right) }-\restr{g}{ \left[2^{-L}m;1-2^{-L}m\right) }}_{\infty} \leq C 2^{-\alpha L}\left( \norm{f_0^{(\lfloor\alpha\rfloor)}}_{\infty}+\norm{f_0}_{\Sigma(\alpha)}\right) \] for $L$ large enough.\\
Let $f_0$ be a probability density such that $f_0>\rho$ for some $\rho >0$. For $L$ large enough, replacing the above bound by $C 2^{-\alpha L}$ with $C$ a constant depending on $m$, $\alpha$, $\norm{f_0^{(\lfloor\alpha\rfloor)}}_\infty$ and $\norm{f_0}_{\Sigma(\alpha)}$, we can choose $g$ above of the form, for $S_{i,2^L,m}$ as in \eqref{period_spline_bas_def}, \[g = \sum_{i=0}^{2^L-1}\theta_i S_{i,2^L,m} \] with $\mathbf{\theta}=(\theta_i)_{0\leq i\leq 2^L}$ in the $2^L$-dimensional unit simplex $S^{2^L}$, such that there exists \[(\eta_i)_{0\leq i\leq 2^L+m-1}\in\left[0; 2\big(\norm{f_0}_{\Sigma(\alpha)}+2\norm{f_0^{(\lfloor\alpha\rfloor)}}_{\infty}\big)\right]^{2^L+m}\]
satisfying

\[
   \theta_k=
\begin{cases}
   \eta_k& \text{if } m\leq k \leq 2^L-1 \\
    \omega_{m,k+1}  \eta_k +(1-\omega_{m,m-k}) \eta_{2^L+k}  & \text{if } 0 \leq k \leq m-1
\end{cases}
\]
and
\[\norm{ f_0-\restr{\sum_{k=0}^{2^L+m-1}\eta_{k} 2^L  \chi^{*(m+1)}(2^L\cdot -(k-m)) }{ [0;1) } }_{\infty} \leq C 2^{-\alpha L}.\]
\end{lemma}
\begin{proof}
Let's introduce the B-spline functions of order $m+1$ on the interval $[-m2^{-L}; 1+m2^{-L}]$ corresponding to the knots $i2^{-L},\ -m\leq i\leq 2^{L}+ m$, denoted $B_{1,m+1},\dots,B_{2^{L}+3m ,m+1}$. Figure \ref{bsplin} depicts these basis functions in the particular case $L=3$ and $m=3$.

The Cox-de Boor recursion formula ensures that B-splines whose supports are far enough from the edges $-m2^{-L}$ and $1+m2^{-L}$ are actually Cardinal splines with suitable scaling. As shown in \cite{de1986b} (Section 10), 
\begin{equation}
\label{from_b_to_c}
B_{k,m+1} = \chi^{*(m+1)}\left(2^L\cdot - (k-2m-1)\right),\qquad m+1 \leq k \leq 2^{L}+2m.
\end{equation}
Also, $B_{i,m+1} $ is supported in an interval of length at most $(m+1)2^{-L}$ included in $[(i-(2m+1))2^{-L};(i-m)2^{-L}]$, i.e.,
\begin{equation}\label{small_support} \forall x\notin[(i-(2m+1))2^{-L};(i-m)2^{-L}],\qquad B_{i,m+1}(x)=0.\end{equation}
As $f_0 \in \Sigma(\alpha, [0,1) )$, according to Lemma \ref{ext_Hold}, there exists a map 
\begin{align*}
h:\ [-(m+1)2^{-L}; 1+(m+1)2^{-L}] \to \mathbb{R}\text{ such that, }\\
h\in \Sigma\left(\alpha, [-m2^{-L}; 1+m2^{-L}]\right),\ \norm{h}_{\Sigma(\alpha)}=\norm{f_0}_{\Sigma(\alpha)},\ \restr{h}{[0;1)}=f_0.
\end{align*} 
Also, for $L$ large enough, $\norm{h^{(\lfloor\alpha\rfloor)}}_\infty\leq2\norm{f_0^{(\lfloor\alpha\rfloor)}}_\infty$ by continuity. Using Lemma \ref{spline_sub_vdv} and \eqref{small_support}, there exists $C$ depending only on $m$ and $\alpha$, and reals $\theta_k, \ m+1 \leq k\leq 2^L+2m $, bounded by $\norm{f_0}_{\Sigma(\alpha)}+2\norm{f_0^{(\lfloor\alpha\rfloor)}}_\infty$, such that for $L$ large enough
\begin{align}
\label{first_bnd}
 \norm{ \restr{h}{ [0;1) }-\restr{\sum_{k=m+1}^{2^L+2m}\theta_k B_{k,m+1}}{ [0;1) } }_{\infty} &\leq C 2^{-\alpha L} \left(\norm{h^{(\lfloor\alpha\rfloor)}}_\infty+\norm{h}_{\Sigma(\alpha)}\right)\nonumber\\
 &\leq C 2^{-\alpha L} \left(\norm{f_0^{(\lfloor\alpha\rfloor)}}_\infty+\norm{f_0}_{\Sigma(\alpha)}\right).
\end{align}
In addition, thanks to the small support of $\chi^{*(m+1)}$ from Lemma \ref{splin_supp_small} and the equality \eqref{from_b_to_c}, the maps $\sum_{k=m+1}^{2^L+2m}\theta_k B_{k,m+1}$ and, with $\tilde{k}(k)=(k-2m-1\bmod 2^L)$,  \[\sum_{k=m+1}^{2^L+m}\theta_k \sum_{i\in\Z}\chi^{*(m+1)}\left(2^L\cdot - (k+i2^L-2m-1)\right)=\sum_{k=m+1}^{2^L+m}\theta_k 2^{-L}S_{\tilde{k}(k),2^{L},m}(\cdot)\] are equal on the interval $ \left[2^{-L}m;1-2^{-L}m\right)$. The latter map then satisfies the inequality in the first part of the theorem according to \eqref{first_bnd} and belongs to $\Tilde{\Pi}_{m+1, \mathbb{Z}/q}(\mathbb{R})$ following \eqref{period_spline_bas_def}.\\

Let's now dwell on the second part of the Lemma. For $L$ large enough, as $f_0>\rho>0$, then $h>\rho/2>0$ by continuity of $h$ on $[-m2^{-L}; 1+m2^{-L}]$. Therefore, Lemma \ref{spline_sub_vdv} also ensures the existence of a constant $c(\rho)$ such that $\theta_k>c(\rho)>0, \ m+1 \leq k\leq 2^L+2m$ in \eqref{first_bnd} for $L$ large enough. From \eqref{from_b_to_c}, Lemma \ref{symmetry_convol} and for $\omega_{m,l}$ as in \eqref{omegas_def}, we see that \[2^L \int_0^1 \sum_{k=m+1}^{2^L+2m}\theta_k B_{k,m+1}(t) dt =  \sum_{i=2m+1}^{2^L+m}\theta_i + \sum_{i=m+1}^{2m}\left(\omega_{m,i-m}\theta_i+\omega_{m,2m+1-i}\theta_{i+2^L}\right)\eqqcolon \Omega_m.\]
For $f_0$ a density, by integration on $[0;1)$, the inequality \eqref{first_bnd} gives 
\[ \left| 2^{-L}\Omega_m- 1 \right| \leq C 2^{-\alpha L} \left(\norm{f_0^{(\lfloor\alpha\rfloor)}}_\infty+\norm{f_0}_{\Sigma(\alpha)}\right).\]Define \[\widetilde{\theta}_i\coloneqq \frac{\theta_{i}}{ 2^{-L}\Omega_m}.\] The two last displays ensure that the $\widetilde{\theta}_i$'s are all bounded by $2\left(\norm{f_0}_{\Sigma(\alpha)}+2\norm{f_0^{(\lfloor\alpha\rfloor)}}_{\infty}\right)$ for $L$ large enough. 
From this inequality and \eqref{first_bnd}, we now write, for a constant $C$ depending on $m$, $\alpha$, $\norm{f_0^{(\lfloor\alpha\rfloor)}}_\infty$ and $\norm{f_0}_{\Sigma(\alpha)}$, that 

\begin{align}
\label{snd_bnd}
\norm{ f_0 -\restr{\sum_{k=m+1}^{2^L+2m}\widetilde{\theta}_k B_{k,m+1}}{ [0;1) } }_{\infty}  &\leq    
 \norm{ f_0 - \restr{\sum_{k=m+1}^{2^L+2m}\theta_k B_{k,m+1}}{ [0;1) } }_{\infty} + \nonumber\\
 &  \left| 1- \left(2^{-L}\Omega_m\right)^{-1}\right| \norm{\restr{\sum_{k=m+1}^{2^L+2m}\theta_k B_{k,m+1}}{ [0;1) }}_\infty \nonumber\\
 &\leq C 2^{-\alpha L}
 \end{align}
since $f_0$ is bounded on $[0;1)$ as a Hölderian density, the bound depending on $\alpha$ and $ \norm{f}_{\Sigma(\alpha)} $ only (see \cite{tsyb_npe}, p. 9).
Let's now define
\[
   \Theta_k\coloneqq 
\begin{cases}
   \widetilde{ \theta}_k& \text{if } 2m+1\leq k \leq 2^L+ m, \\
    \omega_{m,k-m}  \widetilde{ \theta}_k +(1-\omega_{m,k-m}) \widetilde{ \theta}_{2^L+k}  & \text{if } m+1 \leq k \leq 2m, \\
    \omega_{m,2^L+2m+1-k} \widetilde{ \theta}_k + (1-\omega_{m,2^L+2m+1-k}) \widetilde{ \theta}_{k-2^L}  & \text{if } 2^L+m+1 \leq k \leq 2^L+2m. \\
\end{cases}
\]
As pointed out in Lemma \ref{symmetry_convol} , $\omega_{m,l}=1-\omega_{m,m+1-l}$ and, as a consequence, for $m+1 \leq k \leq 2m$, $\Theta_k=\Theta_{2^L+k}$. Then, by definition,
\[ \left(2^{-L}\Theta_k \right)_{2m+1\leq k \leq 2^L+2m} \in S^{2^L}.\]
It remains to introduce the application  $\sum_{i=0}^{2^L-1} 2^{-L}\Theta_{i+2m+1} S_{i,2^L,m}$ which satisfies, using once again the small support of B-splines (see also Part \ref{spline_plots}), that \[\restr{\sum_{i=0}^{2^L-1} 2^{-L}\Theta_{i+2m+1}  S_{i,2^L,m}}{ [2^{-L}m;1-2^{-L}m) } =\restr{ \sum_{k=m+1}^{2^L+2m}\widetilde{\theta}_k B_{k,m+1}}{ [2^{-L}m;1-2^{-L}m) },\] which, along equation \eqref{snd_bnd}, finally brings about the conclusion

\begin{equation*}
 \norm{ \restr{f_0}{  [2^{-L}m;1-2^{-L}m) }-\restr{\sum_{i=0}^{2^L-1} 2^{-L}\Theta_{i+2m+1} S_{i,2^L,m}}{  [2^{-L}m;1-2^{-L}m) } }_{\infty} \leq C 2^{-\alpha L}.
\end{equation*}
Therefore, $g=\sum_{i=0}^{2^L-1} 2^{-L}\Theta_{i+2m+1} S_{i,2^L,m}$ satisfies all the conditions from the lemma.

\end{proof}

\section*{Acknowledgements}
I sincerely thank Pr. Ismaël Castillo (Sorbonne Université), whose guidance and insights into the study of Bayesian tree-based methods were of tremendous help in the making of this paper.

%% file: Supplements.tex
\section{Supplementary materials.}

We present here additional elements used in the derivation of the main results in \textit{Smoothing and adaption of shifted Pólya Tree ensembles} paper. First, we present results on iterated convolutions of $\mathds{1}_{[0;1)}$ and spline functions. At the end of this section, we expound on the link between these functions and recall a classic result in Approximation Theory for splines. Indeed, our priors involve iterated convolutions in their definitions, so that a spline approximation theory is helpful here. This allows us to obtain simpler conditions for the derivation of posterior contraction rates for the prior presented in the article. This is the object of additional lemmas that build on a classical result from \cite{ghosal2000} that we recall. Also available, a numerical study of our results presents simulations, focusing on a simplified version of the priors we introduced. We end with some technical results.

\subsection{Results on iterated convolutions of the indicator function and spline functions.}


\subsubsection{Iterated convolution of the indicator function.}

\begin{lemma}
\label{splin_supp_small}
For $m\in\N^*$, $\norm{\chi^{*m}}_\infty\leq1$ and
\[ \forall x\notin[0;m],\quad \chi^{*m}(x)=0.\]
\end{lemma}

\begin{proof}
For $m=1$, it is straightforward. Then, by induction, from the positivity of $\chi$, for any $x\in \R$,
\begin{align*}
\begin{split}
0\leq\chi^{*(m+1)}(x)&=\int_\R \chi^{*m}(t)\chi(x-t)dt\\
&\leq \int_\R \mathds{1}_{[0;m]}(t)\mathds{1}_{[x-1;x]}(t)dt\\
&\leq \mathds{1}_{]0;m+1[}(x)
\end{split}
\end{align*}
which concludes the proof.
\end{proof}

\begin{lemma}
\label{compact_formula}
Let $m\in\N^*$ and $s>0$. Then
\[ \chi_s^{*m}=s^{m-1}\chi^{*m}(\cdot/s).\]
\end{lemma}

\begin{proof}
The result is straightforward for $m=1$. Now, by induction, assuming that the property for a given $m$ is proved, for $t\in\R$,

\begin{align*}
\begin{split}
\chi_s^{*(m+1)}(t) &= \int_\R  \chi_s\left(u\right) \chi_s^{*m}\left(t-u\right)du\\
&= s^{m-1} \int_\R  \chi_s\left(u\right) \chi^{*m}\left(\frac{t-u}{s}\right)du\\
&= s^{m} \int_\R  \chi\left(v\right) \chi^{*m}\left(\frac{t}{s}-v\right)dv\quad\text{with $v=u/s$}\\
&= s^{m} \chi^{*(m+1)} \left(\frac{t}{s} \right). \\
\end{split}
\end{align*}

\end{proof}

\begin{lemma}
\label{int_chi_formula}
Let $m\in\N^*$ and $s>0$. Then
\begin{equation*}
\label{int_chi}
\int_\R \chi^{*m}(t) dt = 1 \qquad\text{and}\qquad s^{-m}\int_\R \chi_s^{*m}(t) dt = 1.
\end{equation*}
\end{lemma}

\begin{proof}
The result is straightforward for $m=1$. For $m\geq2$, by positivity of $\chi$ and with the change of variable $v=t-u$,
\[ \int_\R \chi^{*m}(t) dt = \int_\R\left[ \int_\R \chi(u)\chi^{*(m-1)}(t-u)du \right] dt=  \int_\R \chi(u) \left[ \int_\R \chi^{*(m-1)}(v)dv \right] du = 1\]
by induction. Then, from Lemma \ref{compact_formula} and with the change of variable $u=t/s$
\[ s^{-m}\int_\R \chi_s^{*m}(t) dt = s^{-1}\int_\R \chi^{*m}(t/s) dt= \int_\R \chi^{*m}(u) du = 1.\]

\end{proof}

\begin{lemma}
\label{symmetry_convol}
Let $m\in\N^*$ and $t\in\R$. Then $\chi^{*m}(m-t)=\chi^{*m}(t)$. As a consequence, for $0\leq l\leq m+1$, $\omega_{m,l}=1-\omega_{m,m+1-l}$.

\end{lemma}

\begin{proof}
For $m=1$, $\chi(1-t)=\mathds{1}_{0\leq 1-t\leq 1}=\mathds{1}_{0\leq t\leq 1}=\chi(t)$. By induction, if the theorem is valid until $m\in\N^*$,

\begin{align*}
\label{symetry_it_conv}
\begin{split}
\chi^{*(m+1)}(m+1-t) &= \int_R \chi(u) \chi^{*m}(m+1-t-u) du\\
&= \int_\R \chi(1-v) \chi^{*m}(v+m-t) dv\quad\text{with the change of variable $v=1-u$}\\
&=  \int_\R \chi(v) \chi^{*m}(t-v) dv\\
&= \chi^{*(m+1)}(t).
\end{split}
\end{align*}
The second part of the Lemma comes from a change of variable in \eqref{omegas_def} and Lemma \ref{int_chi_formula}.

\end{proof}


\subsubsection{Splines periodicity.}

\begin{lemma}
\label{period_splines_seq}
Let $\left(\theta_i\right)_{i\in \mathbb{Z}}\in \R^\Z$, $h=1/N,\ N\in\mathbb{N}^*$ and $m\in\N^*$. The map 
\[ x \to  \sum_{i\in\mathbb{Z}} \theta_i \chi^{*(m+1)}\left(\frac{x}{h}-i\right) \]
is $1$-periodic if and only if the sequence  $\left(\theta_i\right)_{i\in \mathbb{Z}}$ is $N$-periodic. 
\end{lemma}
\begin{proof}
With $r\in \R$ and $p\in\Z$, 
\[  \sum_{i\in\mathbb{Z}} \theta_i \chi^{*(m+1)}\left(\frac{r+p}{h}-i\right)= \sum_{j\in\mathbb{Z}} \theta_{j+pN} \chi^{*(m+1)}\left(\frac{r}{h}-j\right). \]
As $\left\{\chi^{(m+1)}\big(h^{-1}\cdot-i\big)\big|i\in\Z\right\}$ is a basis for the space $\Pi_{m+1,h\Z}\left(\R\right)$ (see Section \ref{spline_introduction}), the above map satisfies
\[ \sum_{i\in\mathbb{Z}} \theta_i \chi^{*(m+1)}\left(\frac{\cdot}{h}-i\right) \text{ is $1$-periodic }\iff  \theta_{i} =\theta_{i+pN} \text{ for any $i\in\Z$ and $p\in\Z$}.\]

\end{proof}

\begin{lemma}
\label{density_prior_int}
Let $m\in\N$, $L\in\N$ and $\left(\Theta_i\right)_{i=0}^{2^L-1}$ be positive real numbers such that 
\[ \sum_{i=0}^{2^L-1}\Theta_i=1\]
and introduce, for $i=2^L-m,\dots,2^L-1$, \[\theta_i\in \left[0;  \frac{\Theta_i}{\omega_{m,2^L-i}} \right]\] where $\omega_{m,i}$ is defined by \eqref{omegas_def}.
Then, if $(u_i)_{i\in\Z}$ is an $(2^L+m)$-periodic sequence, such that
\[
   u_i=
\begin{cases}
   \Theta_i& \text{if } 0\leq i \leq 2^L-m-1 \\
    \theta_i  & \text{if } 2^L-m \leq i \leq 2^L-1 \\
    \frac{\Theta_{i-m}-\omega_{m,2^L+m-i}\theta_{i-m}}{1-\omega_{m,2^L+m-i}} & \text{if } 2^L \leq i \leq 2^L+m-1 \\
\end{cases},
\]the restriction of the map $f^m_{\infty,2^{-L}}$ on $[0;1]$, with $f=\sum_{i\in\Z} u_i H_{Li}$, is a probability density function on $[0;1]$.
\end{lemma}
\begin{proof}
From their definition, the $u_i$'s are positive real numbers such that  $f$ and $f^m_{^{\infty,2^{-L}}}$ are themselves positive as it can be seen in \eqref{convol_cont_aggreg}. It remains to compute the integral. Beforehand, we recall that for any $n\geq 1$, $\chi^{*n}$ is supported on $[0;n]$, so that $\chi^{*n}(2^{L}\cdot-i)$ is supported on $[2^{-L}i; 2^{-L}(i+n)]$, for any $i\in\Z$. The intersection of this last interval with $[0;1]$ has its interior non-empty if and only if $-n+1\leq i\leq 2^L-1$. Hence,

\begin{align*}
\int_{[0;1]}f^m_{\infty,2^{-L}}(v)dv &= \int_{[0;1]} \sum_{i\in\mathbb{Z}} u_i  2^{L} \chi^{*(m+1)}\left(2^Lv-i\right) dv \quad\text{according to \eqref{step_func_can} and \eqref{spline_new_def},}\\
&= \sum_{i=-m}^{2^L-1} u_i  2^{L} \int_{[0;1]}  \chi^{*(m+1)}\left(2^Lv-i\right) dv\quad\text{according to the above remark,}\\
&=  \sum_{i=-m}^{2^L-1} u_i  \int_{[-i;2^L-i]}  \chi^{*(m+1)}\left(r\right) dr\quad\text{with $r=2^Lv-i$,}\\
&= \sum_{i=-m}^{-1} u_i\int_{-i}^{m+1}\chi^{*(m+1)}\left(r\right) dr+\sum_{i=0}^{2^L-m-1} u_i \int_{0}^{m+1}\chi^{*(m+1)}\left(r\right) dr \\
&\quad +\sum_{i=2^L-m}^{2^L-1} u_i\int_{0}^{2^L-i}\chi^{*(m+1)}\left(r\right) dr\\
 &\quad\quad\text{according to the above discussion on the supports,}\\
 &= \sum_{i=-m}^{-1} u_i (1-\omega_{m,-i})+\sum_{i=0}^{2^L-m-1} u_i +\sum_{i=2^L-m}^{2^L-1} u_i\omega_{m,2^L-i}\quad\text{using Lemma \ref{int_chi_formula},}\\
 &= \sum_{i=0}^{2^L-m-1} u_i +\sum_{i=2^L-m}^{2^L-1} u_i\omega_{m,2^L-i} +  \sum_{i=2^L}^{2^L+m-1} u_i (1-\omega_{m,2^L+m-i})\quad\text{by periodicity,}\\
 &= \sum_{i=0}^{2^L-m-1} \Theta_i +\sum_{i=2^L-m}^{2^L-1} \theta_i\omega_{m,2^L-i} + \\
 &\qquad \sum_{i=2^L}^{2^L+m-1} (1-\omega_{m,2^L+m-i}) \frac{\Theta_{i-m}-\omega_{m,2^L+m-i}\theta_{i-m}}{1-\omega_{m,2^L+m-i}} \\
 &=\sum_{i=0}^{2^L-1} \Theta_i + \sum_{i=2^L-m}^{2^L-1} \theta_i\omega_{m,2^L-i} -  \sum_{i=2^L}^{2^L+m-1}\omega_{m,2^L+m-i}\theta_{i-m}\\
 &= 1 +  \sum_{i=2^L-m}^{2^L-1} \theta_i\omega_{m,2^L-i} - \sum_{j=2^L-m}^{2^L-1} \theta_j\omega_{m,2^L-j}\\
 &= 1.
\end{align*}
\end{proof}

In the above proof, we have also proven that the $u_i$'s from the definition of the DPA and CPA distributions are such that
\begin{equation}\label{norma_prior_coeff} \sum_{i=0}^{2^L-m-1}u_i + \sum_{i=2^L-m}^{2^L-1}\left(\omega_{m,2^L-i}u_i+\omega_{m,i-(2^L-m-1)}u_{i+m}\right)=1,\end{equation}
where we used Lemmas \ref{int_chi_formula} and \ref{symmetry_convol} to express the last terms of the sum.\\


\subsubsection{Definition of spline functions with iterated convolutions.}

\begin{lemma}
\label{map_to_splines}
Let $h>0$ and $m\in\N$. Then, if $\phi\colon\R\longrightarrow\R$ is a right continuous piecewise constant map with breaks at points $jh,\ j\in\Z$, then $\phi_{\infty,h}^m$ is a spline function of order $m+1$ with knots $\mathbf{t}=(jh)_{j\in\Z}$.
\end{lemma}

\begin{proof}
Propositions 6.7.1 and 6.7.2 from \cite{aubin2011applied} show that $\chi^{*(m+1)}$ is a spline of order $m+1$ with knots $i=0,\dots,m+1$. Also, we write $\phi=\sum_{j\in\Z}\theta_j \mathds{1}_{[jh;(j+1)h)}(\cdot)$ for some sequence $(\theta_j)_{j\in\Z}$. We conclude the proof with \eqref{convol_cont_aggreg}, Subsection \ref{spline_introduction} and 
\begin{align*}
\phi_{\infty,h}^m(x) &=  \sum_{j\in \mathbb{Z}} h^{-1} \theta_j   \Big[\chi^{*m}(h^{-1}\cdot)*  \mathds{1}_{[jh;(j+1)h)}\Big](x) \\
&=  \sum_{j\in \mathbb{Z}} \theta_j  h^{-1}   \Big[\chi^{*m}(h^{-1}\cdot)*  \chi(h^{-1}\cdot - j)\Big](x)\\
&= \sum_{j\in \mathbb{Z}}  \theta_j  h^{-1}  \int_\R  \chi^{*m}(h^{-1} s) \chi\left(h^{-1} (x-s) - j\right) ds \\
&= \sum_{j\in\mathbb{Z}} \theta_j  h^{-1}  \chi^{*(m+1)}\left(h^{-1} x-j\right)\quad\text{with the change of variable $v=h^{-1}s$.}
\end{align*}
\end{proof}

\subsubsection{Spline approximation.}

\begin{lemma}
\label{spline_sub_vdv}
Suppose $k \geq \alpha > 0$ and $\mathbf{t}$ is a finite knot sequence of step at most $T^{-1}$, included in a closed bounded interval $I\subset\R$. There exists a constant $C$ depending only on $k$ and $\alpha$ such that for every $f_0\in \Sigma\left(\alpha,I\right)$ and $T$ large enough, there exists $\theta \in \R^{k+\#\mathbf{t}-1}$ with $\norm{\theta}_\infty < \norm{f_0^{(\lfloor\alpha\rfloor)}}_\infty+\norm{f_0}_{\Sigma(\alpha)}$ and
\[\norm{\sum_{i=1}^{k+\#\mathbf{t}-1}\theta_i B_{i,k}-f_0}_\infty \leq C T^{-\alpha}\left(\norm{f_0^{(\lfloor\alpha\rfloor)}}_\infty+\norm{f}_{\Sigma(\alpha)}\right)\] where the $B_{i,k}$'s form the B-spline basis of $\Pi_{k, \mathbf{t}}(I)$.
Furthermore,  if $f_0$ is strictly positive, for $T$ large enough, the vector $\theta$ can be chosen to have strictly positive coordinates. The $\theta_i$'s can indeed be lower bounded by a strictly positive constant depending on the lower bound on $f_0$.
\end{lemma}

\begin{proof}
This is Lemma E.4 from \cite{fundamentalsBNP}.
\end{proof}


\subsubsection{Plots}\label{spline_plots}

On Figure \ref{cycl_plot}, we see (in the particular case $L=3$ and $m=3$) that on the interval $ [2^{-L}m;1-2^{-L}m) $, the basis functions $S_{i,2^L,m}$ are equal to the basis functions $B_{i+2m+1,m+1}$ from Figure \ref{bsplin}. 

\begin{figure}[H]
\centering
\includegraphics[width=.9\textwidth]{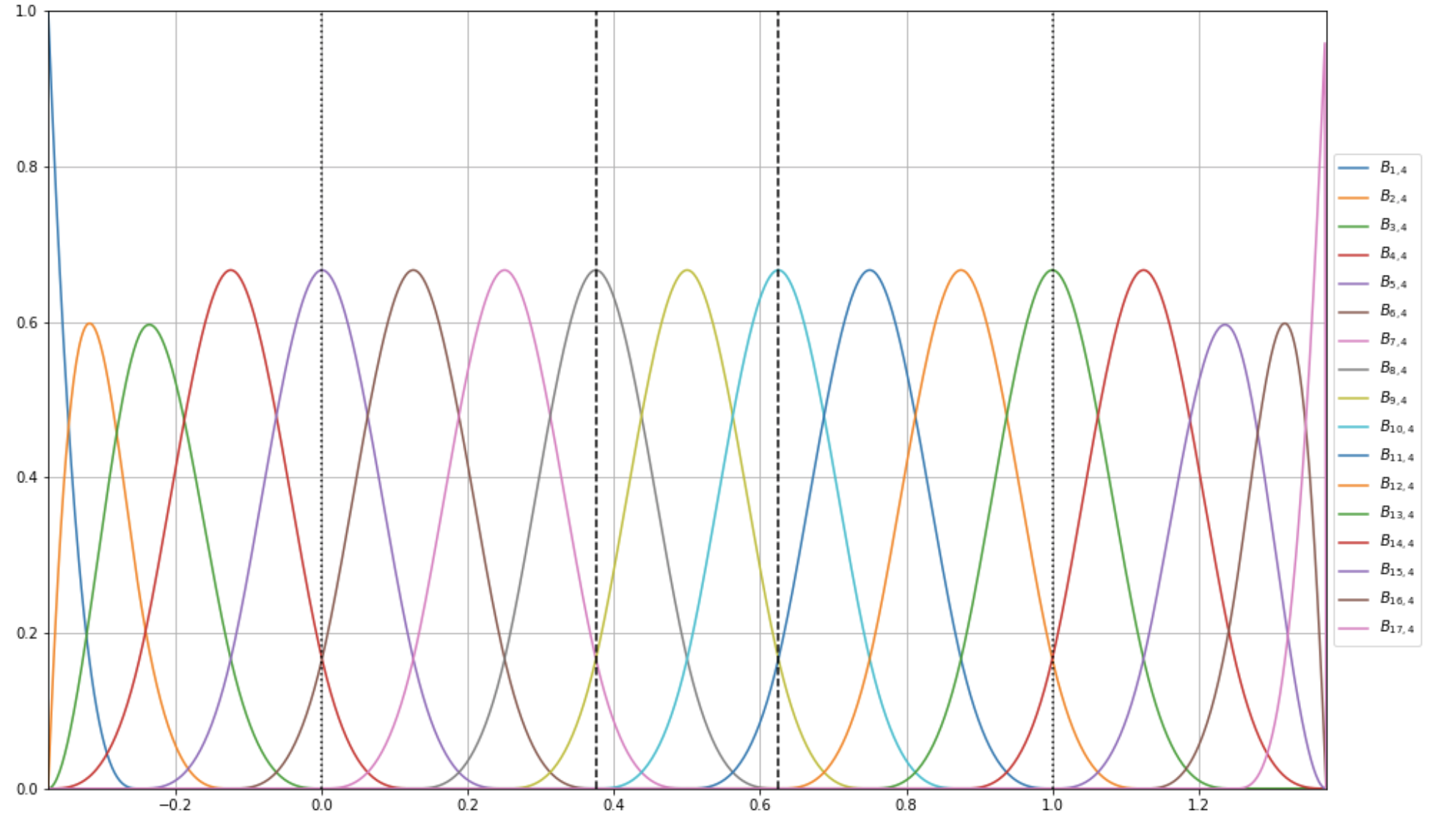}
\caption{B-splines of order $4$ with knots $t_i=i/8$ for $-4\leq i\leq 8+4$ as introduced in the proof of Lemma \ref{approx_cyclicSplines}.}
\label{bsplin}
\end{figure}

\begin{figure}[H]
\centering
\includegraphics[width=.9\textwidth]{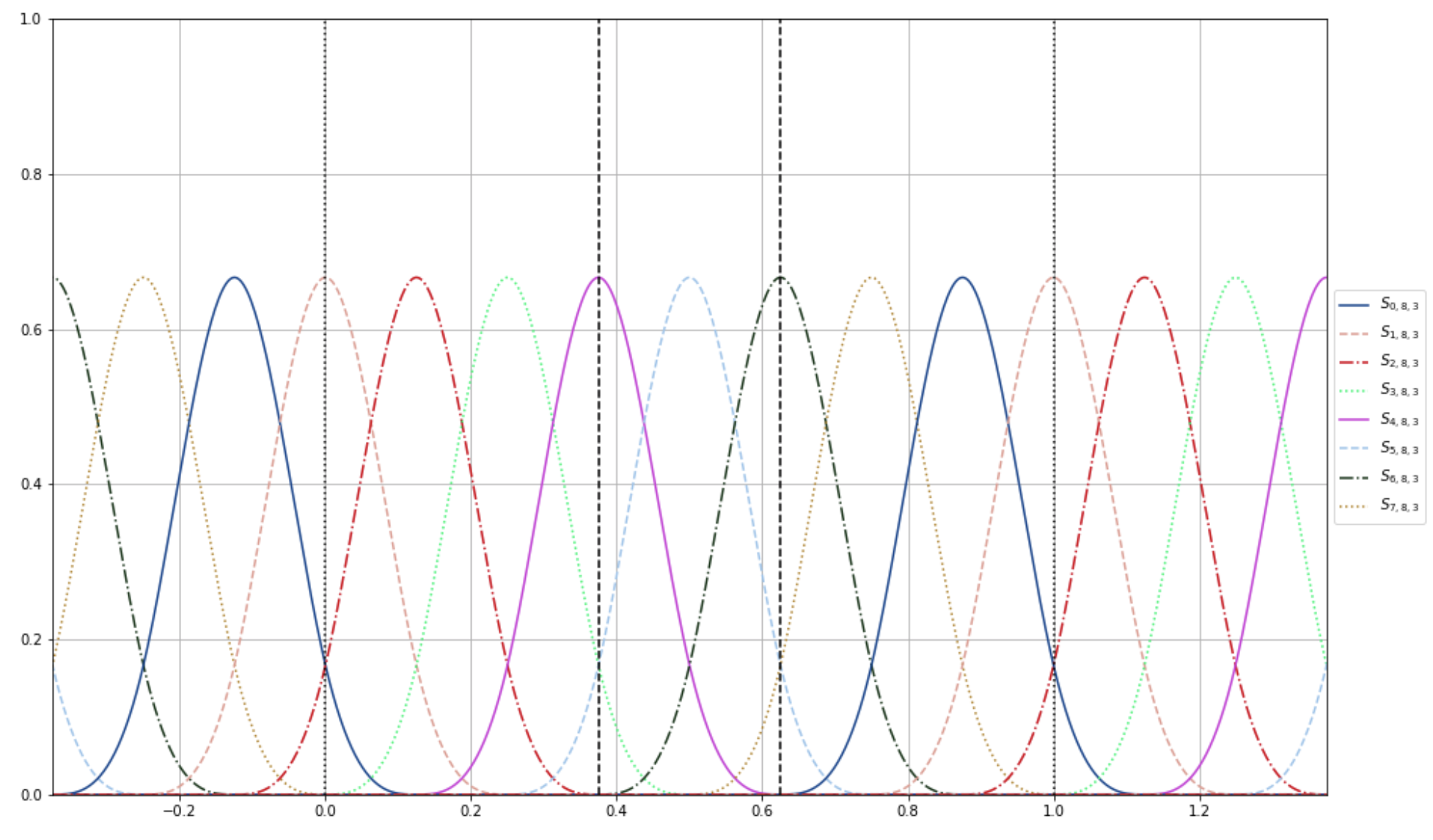}
\caption{Base functions $S_{i,8,3}$.}
\label{cycl_plot}
\end{figure}

Figure \ref{fpf_draw} depicts a $1$-step aggregation of a TPT sample as introduced in Section \ref{DPA_subsection}. One sees that it smoothes the histogram, as it tends to a piecewise linear density. Even though it remains a histogram function, this added smoothness accounts for better estimation performance with the priors introduced in the paper.

\begin{figure}[H]
\begin{subfigure}{.45\textwidth}
  \centering
  \includegraphics[width=.9\linewidth]{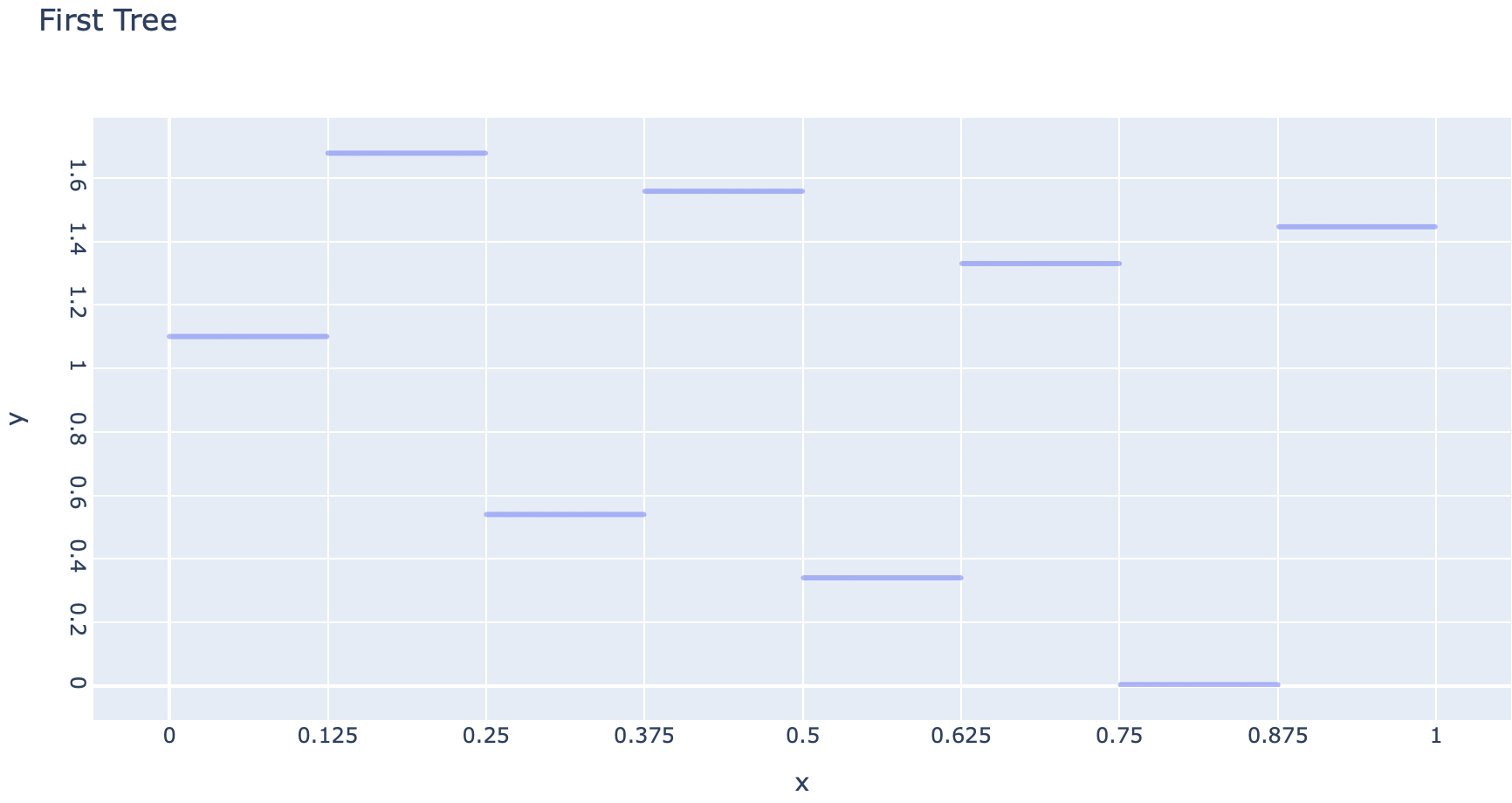}
  \label{fig:sfig1}
\end{subfigure}%
\begin{subfigure}{.45\textwidth}
  \centering
  \includegraphics[width=.9\linewidth]{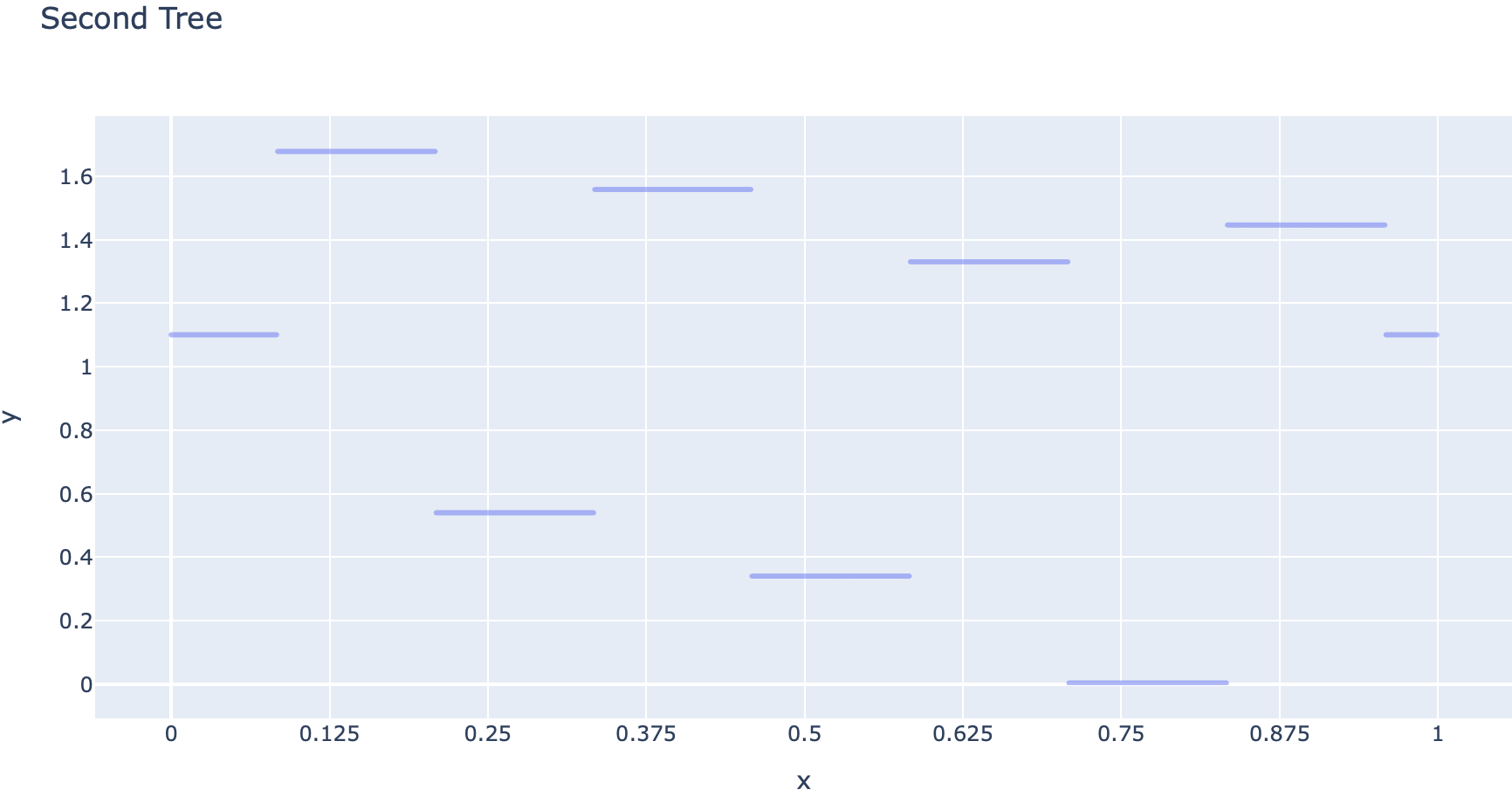}
  \label{fig:sfig2}
\end{subfigure}
\begin{subfigure}{.45\textwidth}
  \centering
  \includegraphics[width=.9\linewidth]{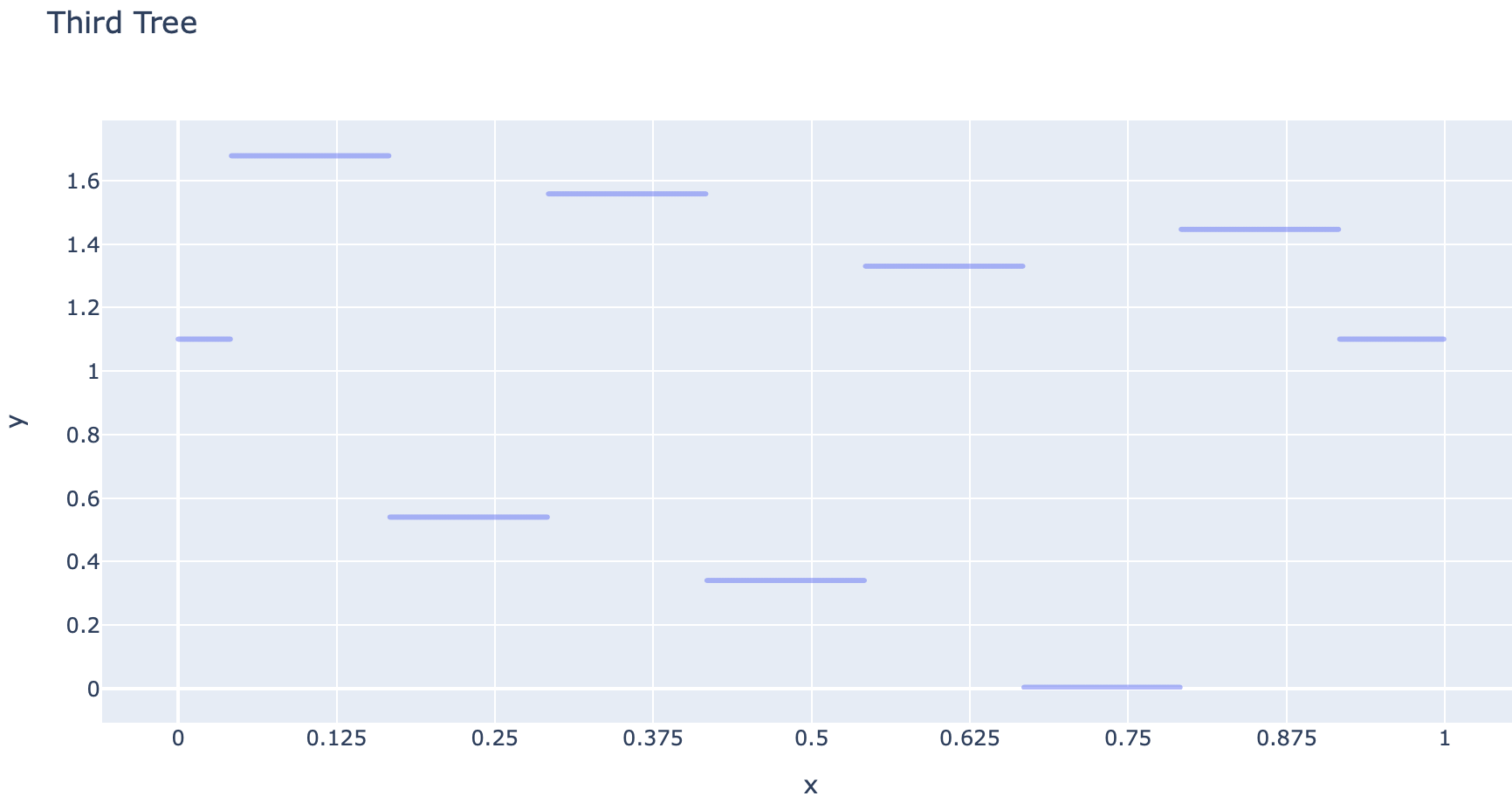}
  \label{fig:sfig3}
\end{subfigure}
\begin{subfigure}{.45\textwidth}
  \centering
  \includegraphics[width=.9\linewidth]{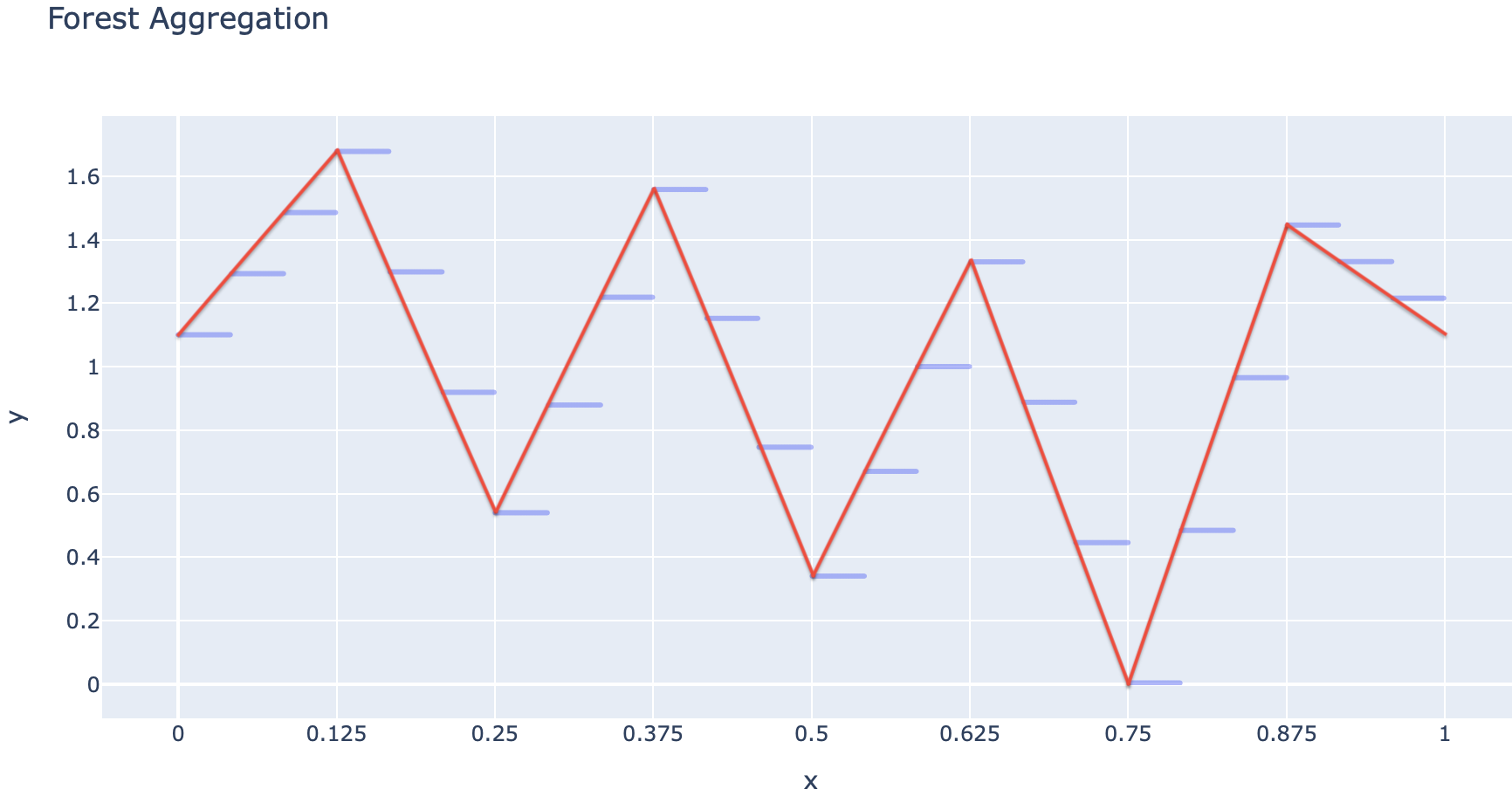}
  \label{fig:sfig4}
\end{subfigure}
\caption{"Naïve" aggregation $f^1_{3,2^{-3}}$ where $f$ is the periodic extension of a sample from $\text{TPT}_3\left(\mathcal{A}\right)$. The first plots represents the shifted trees. The map in red is $f_{\infty,2^{-3}}^1$ when $q\to\infty$.}
\label{fpf_draw}
\end{figure}

\begin{figure}[H]
  \includegraphics[width=\linewidth]{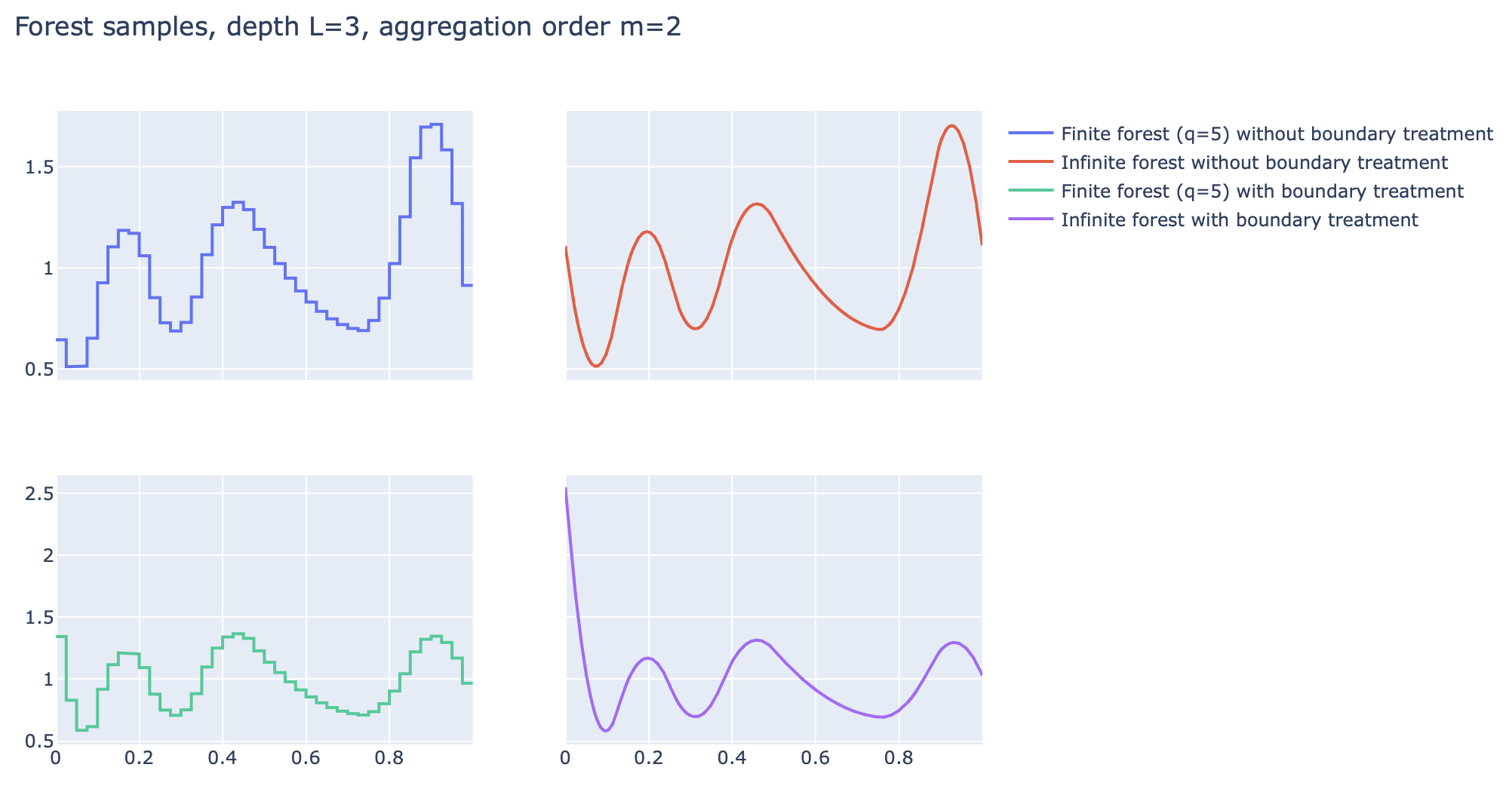}
  \caption{Draws from the DPA and CPA priors and their equivalents without the draw of uniform random variable to modify the behaviour near the frontier of $[0;1)$, with $L=3$ and $m=2$.}
  \label{fpf_draw2}
\end{figure}

\begin{figure}[H]
  \includegraphics[width=\linewidth]{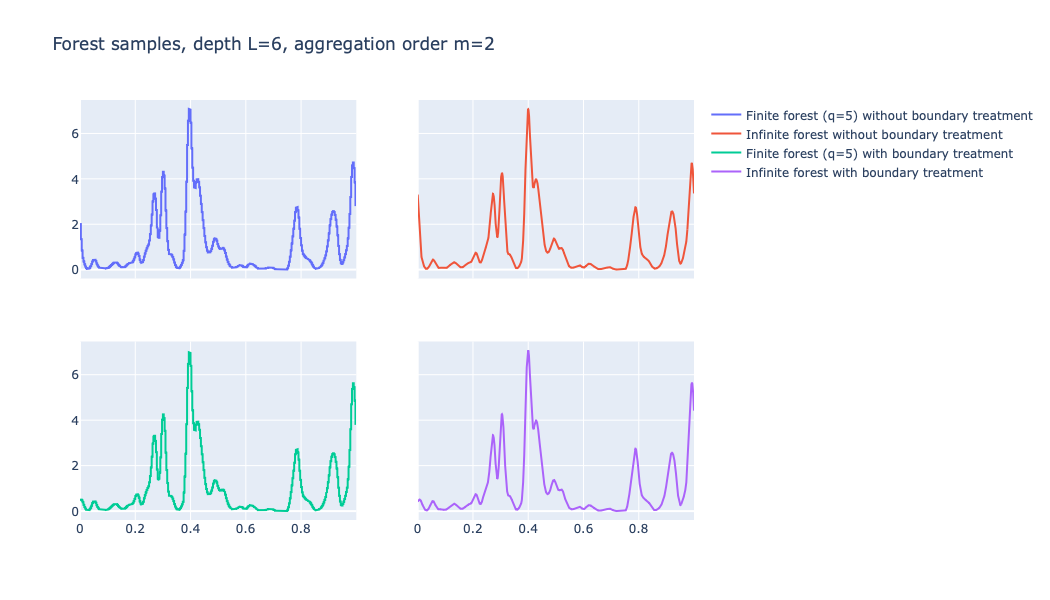}
  \caption{Draws from the DPA and CPA priors and their equivalents without the draw of uniform random variable to modify the behaviour near the frontier of $[0;1)$, with $L=6$ and $m=2$.} 
  \label{fpf_draw3}
\end{figure}


\subsection{Numerical simulations}\label{simulations}

Though sampling from the CPA or DPA posterior may be possible via usual MCMC methods, the modification of the samples near the frontier brought by the uniform variables makes it difficult to explicit the posterior or to come up with more efficient sampling algorithms.

However, if we discard the uniform variables from the definition of the prior, it becomes possible to derive an explicit formula of the prior. Namely, for $L>0$, $q>0$ and $m\geq0$, let's focus on the image prior of the $TPT_L(\mathcal{A})$ distribution by the map $f\to f^m_{q,2^{-L}}$ (using definitions from Section \ref{DPA_subsection}). Then, observing an i.i.d. sample $X\in[0;1)^n$, $n>0$, it is possible to show that the posterior is the image measure by $f\to f^m_{q,2^{-L}}$ of a mixture of TPT distribution, which makes it possible to sample directly from the posterior.

Indeed, for $\mathcal{A}=\left(a_l\right)_{0<l\leq L}$ and $Y=\left\{Y_{\kappa0},\ 0\leq |\kappa|<L\right\}$, we have that the posterior on $Y$ is
\[ \Pi\left[Y|X\right] \propto f\left(X_1,\dots,X_n|Y\right)\prod_{|\kappa|=0}^{L-1} \left[Y_{\kappa0}(1-Y_{\kappa0})\right]^{a_{|\kappa|+1}-1},\]
where the likelihood is
\begin{align*}
f\left(X_1,\dots,X_n|Y\right)&=\prod_{i=1}^n \left[q^{-m} \sum_{(j_1,\dots,j_m)\in[\![0;q-1]\!]^m}  \sum_{\kappa, |\kappa|=L} \mathds{1}_{X_i-2^{-L}(j_1+\dots+j_m)/q\in I_\kappa} 2^L \prod_{j=1}^{L} Y_{\kappa^{[j]}} \right]\\
&= q^{-mn}\sum_{\substack{(j_{1,1},\dots,j_{1,m})\in[\![0;q-1]\!]^m,\dots,\\(j_{n,1},\dots,j_{n,m})\in[\![0;q-1]\!]^m}} \prod_{i=1}^{n}  \sum_{\kappa, |\kappa|=L} \mathds{1}_{X_i-2^{-L}(j_1+\dots+j_m)/q\in I_\kappa} 2^L \prod_{j=1}^{L} Y_{\kappa^{[j]}}.
\end{align*}

For given $(j_{1,1},\dots,j_{1,m}),\dots,(j_{n,1},\dots,j_{n,m})$, all in $[\![0;q-1]\!]^m$, let's note $N_{X,(j_{11},\dots,j_{n,m})}\left(I_\kappa\right)=\sum_{i=1}^n \mathds{1}_{X_i-2^{-L}(j_1+\dots+j_m)/q\in I_\kappa}$ for any $\kappa, |\kappa|\geq0$, so that
\begin{align*}
&\prod_{i=1}^{n}  \sum_{\kappa, |\kappa|=L} \mathds{1}_{X_i-2^{-L}(j_1+\dots+j_m)/q\in I_\kappa}  \prod_{j=1}^{L} Y_{\kappa^{[j]}} = \\
&\qquad\qquad \prod_{\kappa, 0\leq|\kappa|\leq L-1}Y_{\kappa0}^{N_{X,(j_{11},\dots,j_{n,m})}\left(I_{\kappa0}\right) }(1-Y_{\kappa0})^{N_{X,(j_{11},\dots,j_{n,m})}\left(I_{\kappa1}\right) }.
\end{align*}
Finally, the posterior on $Y$ is proportional to 
\[ q^{-mn}\sum_{\substack{(j_{1,1},\dots,j_{1,m})\in[\![0;q-1]\!]^m,\dots,\\(j_{n,1},\dots,j_{n,m})\in[\![0;q-1]\!]^m}}   \prod_{\kappa, 0\leq|\kappa|\leq L-1}Y_{\kappa0}^{N_{X,(j_{11},\dots,j_{n,m})}\left(I_{\kappa0}\right) +a_{|\kappa|+1}-1}(1-Y_{\kappa0})^{N_{X,(j_{11},\dots,j_{n,m})}\left(I_{\kappa1}\right) +a_{|\kappa|+1}-1}. \]
The distribution on $Y$ is then a mixture of $TPT_L$ distributions with parameter sets $\mathcal{A}_{(j_{11},\dots,j_{n,m})}=\left\{ N_{X,(j_{11},\dots,j_{n,m})}\left(I_{\kappa}\right) +a_{|\kappa|}-1,\  1\leq|\kappa|\leq L\right\}$ (see \cite{fundamentalsBNP}, Chapter 3).

 At the end of this section, we present a theoretical result for this posterior, but we present some simulations first. The density $f_0: x\in[0;1)\mapsto 1+0.5*\sin\left(2\pi x\right)$ satisfies the assumptions of the mentioned theorem with $\alpha=1.5$. We simulated $10^4$ samples from this density and drew $10$ samples from the posterior, with parameters tuned as in the below theorem. We also compare these with samples of the TPT posterior with parameters tuned as in \cite{MR3729648}. In this paper, sup-norm posterior contraction rates are proven for $\alpha\leq1$. On Figure \ref{post_sim}, we can see that the modified DPA prior is associated with smoother posterior samples and leverage the larger regularity of the signal.

\begin{figure}[H]
  \includegraphics[width=\linewidth]{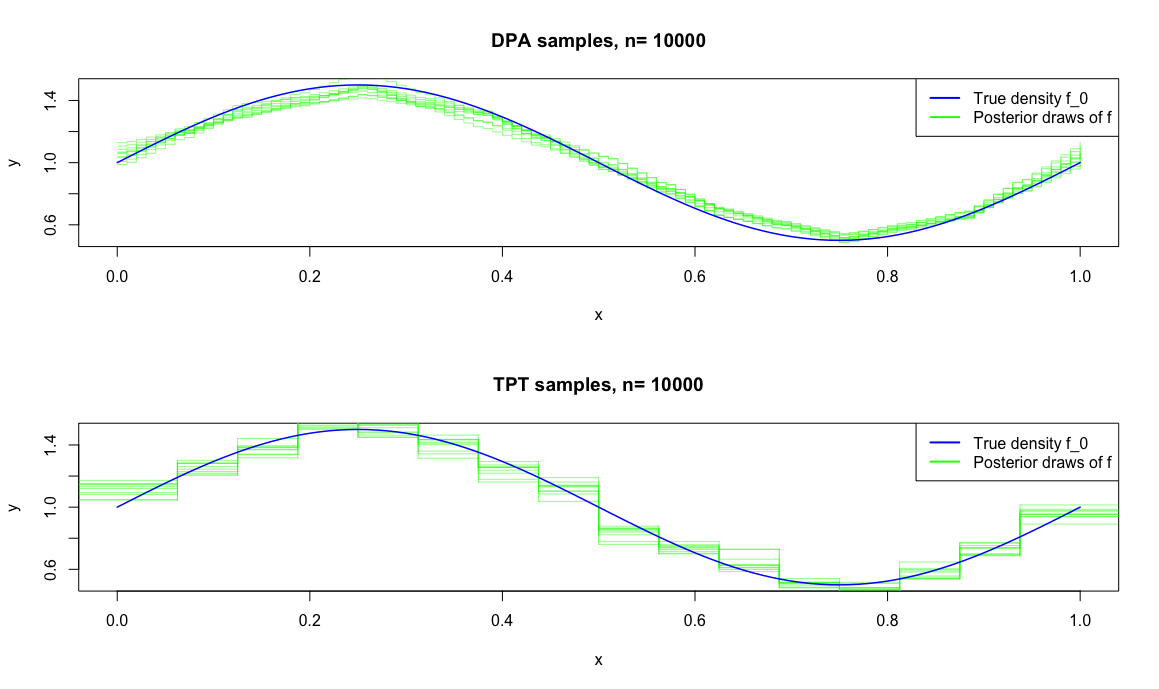}
  \caption{Posterior samples for the simplified DPA prior and the TPT prior, with sine sampling density and sample size $n=10^4$.}
  \label{post_sim}
\end{figure}

We also studied what happens with a density $f_0$ whose behavior near the frontier renders this simplified prior inadequate. Namely, on Figure \ref{post_sim2}, we analyze the situation where the sampling density $f_0$ is increasing, $3/2$-smooth, but has different limits towards $0$ and $1$ (it was obtained as the re-scaled exponential of an integrated Brownian motion). With a sample of size $n=10^5$ from $f_0$ and comparing samples from their posteriors, the simplified forest prior still outperforms the single-tree prior far for the frontier. However, as we did not include a modification near the frontier, it behaves badly towards $0$ and $1$.

\begin{figure}[H]
  \includegraphics[width=\linewidth]{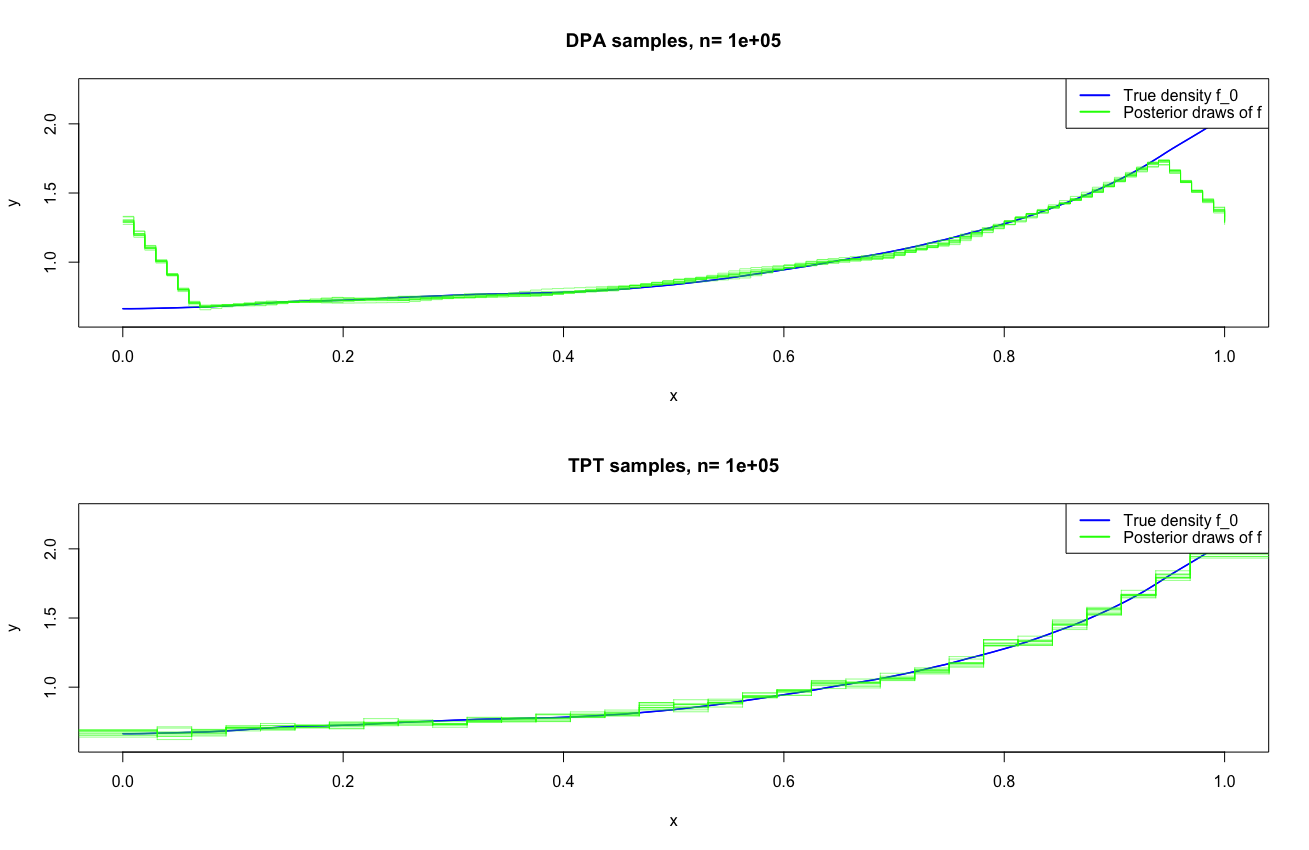}
  \caption{Posterior samples for the simplified DPA prior and the TPT prior, with integrated Brownian sampling density and sample size $n=10^5$.}
  \label{post_sim2}
\end{figure}

As discussed in the paper, this prior is not well-suited for the estimation of general smooth densities as it behaves badly near the frontier of $\Omega$. However, if we make additional assumptions on the true densities to be recovered, it is possible to obtain an analog of Theorem 1.

\begin{theorem}
Suppose $f_0\in  \Sigma(\alpha, [0,1))$, $\alpha>0$, $f_0\geq\rho$ for some $\rho>0$ and $f_0^{(i)}(0)=f_0^{(i)}(1^{-})$ for $i=0,\dots,\lfloor\alpha\rfloor$. If $\Pi$ is the probability distribution of $f=g^{\lfloor\alpha\rfloor}_{q_n,2^{-L_n}}$, $g\sim TPT_L\left(\left(a_l\right)_{0<l\leq L_n}\right)$, such that for some $\beta>0,R\geq1, \delta>0$,
\begin{itemize}
\item $2^{L_n}\asymp \left(\frac{n}{\log n}\right)^{\frac{1}{2\alpha+1}}$,
\item $q_n \geq 2^{\alpha L_n}$,
\item $\forall a\in\mathcal{A},\ a\in[r;R]$ with $R>r>0$,
\end{itemize}
Then, for $M>0$ depending on $\rho, \alpha$, $\|f_0\|_{\Sigma(\alpha)}$, $\beta$ and $R$, and $d(f,g)=\|f-g\|_1$ or $d(f,g)=h(f,g)$, as $n\rightarrow\infty$, \[\mathbb{E}_{f_0}\Pi\Bigg[d(f_0,f)>M\left(\frac{\log n}{n}\right)^{\frac{\alpha}{2\alpha +1}}\big| X\Bigg]\rightarrow0.\]
\end{theorem}

The proof is similar to the one we provide in the paper, although quite simpler as we do not have to take care of what happens near the frontier.


\subsection{Contraction rate derivation}

\begin{theorem}[Ghosal, Ghosh, Van Der Vaart, 2000]
\label{GGV_theo}

Suppose $d$ is either the Hellinger or the $L_1$ distance and  $\Pi$ is an \textit{a priori} probability distribution on the space of probability densities. Also,
\begin{align*}
\begin{split}
B_{KL}(f_0,\epsilon)=\bigg\{f:[0,1)\mapsto\mathbb{R}\ |\ &K(f_0, f)\coloneqq \int f_0\log\frac{f_0}{f}\leq \epsilon^2,\\
&V(f_0, f)\coloneqq \int f_0\left(\log\frac{f_0}{f} - K(f_0, f)\right)^2\leq \epsilon^2 \bigg\}.
\end{split}
\end{align*}
If the positive sequence $(\epsilon_n)_{n\geq0}$ satisfies $\epsilon_n\rightarrow 0$, $n\epsilon_n^2\rightarrow +\infty$ and there exist sets $\mathcal{F}_n$ such that the three following conditions are satisfied for some $c>0$, $D>0$

\begin{enumerate}
\item $\Pi[B_{KL}(f_0,\epsilon_n)]\geq e^{-cn\epsilon_n^2}$,
\item $\log N\left(\epsilon_n, \mathcal{F}_n, d\right)\leq Dn\epsilon_n^2$,
\item $\Pi[\mathcal{F}_n^c]\leq e^{-(c+4)n\epsilon_n^2}$,
\end{enumerate}
it then follows, for a constant $M>0$ sufficiently large, depending on $c$ and $D$, that the posterior satisfies, as $n\to\infty$,

$$\mathbb{E}_{f_0}\Pi[d(f_0,f)>M\epsilon_n|X]\longrightarrow0.$$

\end{theorem}


\subsection{Forest priors $DPA$ and $CPA$.}

\subsubsection{Bounds on the Kullback-Leibler divergence}

\begin{lemma}
\label{KL_cont_unif}
Suppose $f_0\in  \Sigma(\alpha, [0,1))$, $\alpha>0$, and $f_0\geq\rho$ for some $\rho>0$. For $m\geq\lfloor\alpha\rfloor$, take $(\eta_i )_{0\leq i\leq 2^L+m+1}$ as the sequence from Lemma \ref{approx_cyclicSplines}. Then, for $(u_i )_{i\in\Z}$ a $2^L+m$-periodic sequence satisfying \eqref{norma_prior_coeff}, $L\in\N$, \[f=\sum_{i\in\Z} u_i H_{Li},\]we have that there exists a constant $C$ depending only on $\rho,\ m,\ \alpha$ and $\norm{f_0}_{\Sigma(\alpha)}$ such that, for  $m\geq\lfloor\alpha\rfloor$,
 \[ K\left(f_0, \restr{ f^m_{\infty,2^{-L}}}{[0;1)}\right)\vee V\left(f_0, \restr{ f^m_{\infty,2^{-L}}}{[0;1)}\right) \leq C\left( 2^{-2\alpha L} + 2^{2L} \underset{0\leq i \leq 2^L+m-1}{\max} \left| \eta_{i} - u_{i-m}\right|^2  \right) \]
 for $L$ large enough and $2^L\underset{0\leq i \leq 2^L+m-1}{\max} \left| \eta_{i} - u_{i-m}\right|$ small enough.
 Also, if $q$ is a large enough integer,
  \[ V\left(f_0, \restr{\frac{f^m_{q,2^{-L}}}{\int_0^1 f^m_{q,2^{-L}}(t)dt}}{[0;1)}\right) \leq C\left( 2^{-2\alpha L} + 2^{2L} \underset{0\leq i \leq 2^L+m-1}{\max} \left| \eta_{i} - u_{i-m}\right|^2  + \left(\frac{m\omega_{m,1}^{-1}2^L}{q}\right)^2 \right) \]
 and
 \[ K\left(f_0, \restr{\frac{f^m_{q,2^{-L}}}{\int_0^1 f^m_{q,2^{-L}}(t)dt}}{[0;1)}\right) \leq C\left( 2^{-2\alpha L} + 2^{2L} \underset{0\leq i \leq 2^L+m-1}{\max} \left| \eta_{i} - u_{i-m}\right|^2  + \left(\frac{m\omega_{m,1}^{-1}2^L}{q}\right)^2 \right). \]
\end{lemma}

\begin{proof}
From Lemma \ref{approx_cyclicSplines}, the map  $\Tilde{f}_0^{L}=  \sum_{i=0}^{2^L+m-1}\eta_i 2^L \chi^{*(m+1)}\left(2^L\cdot -(i-m)\right)$ is such that

\begin{equation}\label{partial_bias_2} \norm{ f_0-\restr{\Tilde{f}_0^{L}}{ \left[0;1\right)  } }_{\infty} \leq C 2^{-\alpha L} \end{equation} 
with $C$ depending on $m$, $\alpha$ and $\norm{f_0}_{\Sigma(\alpha)}$. Let's first write the decomposition, given $A$ and $B$ exist (it will be shown later),
\begin{equation*}
\begin{array}{rcl}
K\left(f_0, \restr{ f^m_{\infty,2^{-L}}}{[0;1)}\right) &=& \int_0^1 f_0(t)\log\left(\frac{f_0(t)}{ f^m_{\infty,2^{-L}}(t) }\right) dt\\
      &=& \underbrace{\int_0^1 f_0(t)\log\left(\frac{f_0(t)}{ \Tilde{f}_0^{L}(t) }\right) dt}_\text{= A} + \underbrace{\int_0^1 f_0(t)\log\left(\frac{\Tilde{f}_0^{L}(t)}{ f^m_{\infty,2^{-L}}(t) }\right) dt}_\text{= B}.
\end{array}
\end{equation*}
Focusing on the first term, we have, as $\log(1+u)\leq u$ for $u>-1$,
\begin{equation*}
\label{sndBoundA}
\begin{array} {lcl} 
A &\leq& \int_0^1 f_0(t)\frac{f_0(t)- \Tilde{f}_0^{L}(t)}{ \Tilde{f}_0^{L(t)}}dt \\
&=& \int_0^1 \frac{\left(f_0(t)- \Tilde{f}_0^{L}(t)\right)^2}{ \Tilde{f}_0^{L(t)}}dt + 1 - \int_0^1 \Tilde{f}_0^{L}(t)dt\\
&=&\int_0^1 \frac{\left(f_0(t)- \Tilde{f}_0^{L}(t)\right)^2}{ \Tilde{f}_0^{L(t)}}dt + 1 - \left[\sum_{k=m}^{2^L-1} \eta_k +\sum_{k=0}^{m-1}  \big(\omega_{m,k+1}  \eta_k +(1-\omega_{m,m-k}) \eta_{2^L+k}\big)\right]\\
&\leq& \frac{2}{\rho}\norm{f_0- \restr{\Tilde{f}_0^{L}}{ \left[0;1\right)  } }_\infty^2 \quad \text{for $L$ large enough,}\\
&\leq& \frac{C}{\rho} 2^{-2\alpha L} ,
\end{array}
\end{equation*}
where we lower bounded  $\Tilde{f}_0^{L}$ by $\rho/2$ on $[0;1)$ for $L$ large enough as $f_0$ is lower bounded by $\rho>0$. On the other hand, since $\Tilde{f}_0^{L}$ and $f^m_{\infty,2^{-L}}$ have unit integral on $[0;1)$ according to their definitions, we have the upper bound 
\begin{equation*}
\label{sndBoundB}
\begin{array} {lcl} 
B &\leq& \int_0^1 f_0(t)\frac{\Tilde{f}_0^{L}(t)-f^m_{\infty,2^{-L}}(t)} {f^m_{\infty,2^{-L}}(t)} dt \\
&=& \int_0^1 \frac{\left(\Tilde{f}_0^{L}(t)-f^m_{\infty,2^{-L}}(t)\right)\left(f_0(t)-f^m_{\infty,2^{-L}}(t)\right)}{f^m_{\infty,2^{-L}}(t)}dt\\
&=& \int_0^1 \frac{\left(\Tilde{f}_0^{L}(t)-f^m_{\infty,2^{-L}}(t)\right)\left(f_0(t)-\Tilde{f}_0^{L}(t)\right)}{f^m_{\infty,2^{-L}}(t)}dt+\int_0^1 \frac{\left(\Tilde{f}_0^{L}(t)-f^m_{\infty,2^{-L}}(t)\right)^2}{f^m_{\infty,2^{-L}}(t)}dt\\
&\leq& \frac{1}{2}\int_0^1 \frac{\left(\Tilde{f}_0^{L}(t)-f_0(t)(t)\right)^2}{f^m_{\infty,2^{-L}}(t)}dt + \frac{3}{2}\int_0^1 \frac{\left(\Tilde{f}_0^{L}(t)-f^m_{\infty,2^{-L}}(t)\right)^2}{f^m_{\infty,2^{-L}}(t)}dt\ \text{using that $2ab\leq a^2+b^2$ for real numbers.}
\end{array}
\end{equation*} 
At this point, we develop from \eqref{spline_new_def} and the fact that the maps $\chi^{*(m+1)}(2^L\cdot -k)$ make a partition of the unity (see section E.2 from \cite{fundamentalsBNP})

\begin{align}
\begin{split}
\label{bound_in_coef_spli}
\norm{\restr{\Tilde{f}_0^{L}}{[0;1)}-\restr{f^m_{\infty,2^{-L}}}{[0;1)}}_\infty &\leq \norm{\sum_{i=0}^{2^L+m-1} (\eta_{i}-u_{i-m})  2^L\chi^{*(m+1)}(2^L\cdot -(i-m)) }_\infty\\
&\leq 2^L\underset{0\leq i \leq 2^L+m-1}{\max} \left| \eta_{i} - u_{i-m}\right|.
\end{split}
\end{align}
Therefore, for $L$ large enough and $2^L\underset{0\leq i \leq 2^L+m-1}{\max} \left| \eta_{i} - u_{i-m}\right|$ small enough, we also lower bound $f^m_{\infty,2^{-L}}$ by $\rho/4$. Under such conditions, we finally give the bounds
\[ B\leq \frac{2}{\rho}\norm{f_0- \restr{\Tilde{f}_0^{L}}{ \left[0;1\right)  } }_\infty^2 + \frac{6}{\rho} \norm{\Tilde{f}_0^{L}-f^m_{\infty,2^{-L}}}_\infty^2 \leq   C\left( 2^{-2\alpha L} + 2^{2L} \underset{0\leq i \leq 2^L+m-1}{\max} \left| \eta_{i} - u_{i-m}\right|^2  \right) \]
and 
\begin{equation*}
        K(f_0,  f^m_{\infty,2^{-L}}) \leq C\left( 2^{-2\alpha L} + 2^{2L} \underset{0\leq i \leq 2^L+m-1}{\max} \left| \eta_{i} - u_{i-m}\right|^2  \right).
\end{equation*}
For the discrete version,
\[ K\left(f_0, \restr{\frac{f^m_{q,2^{-L}}}{\int_0^1 f^m_{q,2^{-L}}(t)dt}}{[0;1)}\right) \leq K(f_0,  f^m_{\infty,2^{-L}})  + \int_{[0;1)} f_0\log\left( \frac{ f^m_{\infty,2^{-L}}}{ f^m_{q,2^{-L}}\left(\int_0^1 f^m_{q,2^{-L}}(t)dt\right)^{-1}}\right) d\lambda. \]
It then remains to use Lemma \ref{Discretiz_control} to see that for $q$ large enough, since $0\leq f\leq 2^L \omega_{m,1}^{-1}$ from \eqref{norma_prior_coeff} and $\underset{1\leq l \leq m}{\inf}\omega_{m,l}=\omega_{m,1}$, 
\begin{align*}
 &\int_{[0;1)} f_0\log\left( \frac{ f^m_{\infty,2^{-L}}}{ f^m_{q,2^{-L}}\left(\int_0^1 f^m_{q,2^{-L}}(t)dt\right)^{-1}}\right) d\lambda\hfill\\
 &\leq    \int_{[0;1)} f_0\frac{\left( f^m_{\infty,2^{-L}}-f^m_{q,2^{-L}}\left(\int_0^1 f^m_{q,2^{-L}}(t)dt\right)^{-1}\right)}{ f^m_{q,2^{-L}}\left(\int_0^1 f^m_{q,2^{-L}}(t)dt\right)^{-1} }\\
 &=  \int_{[0;1)} \left(f_0-\Tilde{f}_0^{L}+\Tilde{f}_0^{L}-f^m_{\infty,2^{-L}}+f^m_{\infty,2^{-L}}-f^m_{q,2^{-L}}\left(\int_0^1 f^m_{q,2^{-L}}(t)dt\right)^{-1}\right)\\
 &\pushright{ \frac{\left( f^m_{\infty,2^{-L}}-f^m_{q,2^{-L}}\left(\int_0^1 f^m_{q,2^{-L}}(t)dt\right)^{-1}\right)}{ f^m_{q,2^{-L}}\left(\int_0^1 f^m_{q,2^{-L}}(t)dt\right)^{-1} } d\lambda}\\
 &\leq    \frac{4}{\rho} \norm{f_0-\Tilde{f}_0^{L}}_\infty^2               +\frac{4}{\rho} \norm{\Tilde{f}_0^{L}-f^m_{\infty,2^{-L}}}_\infty^2  +    \frac{16}{\rho}\norm{ f^m_{\infty,2^{-L}}-f^m_{q,2^{-L}}\left(\int_0^1 f^m_{q,2^{-L}}(t)dt\right)^{-1} }_\infty^2 \\
 &\quad\text{since } 2ab\leq a^2+b^2\text{ and } (a+b+c)^2\leq 3(a^2+b^2+c^2),\\
 &\leq C\left( 2^{-2\alpha L} + 2^{2L} \underset{0\leq i \leq 2^L+m-1}{\max} \left| \eta_{i} - u_{i-m}\right|^2  + \left(\frac{m\omega_{m,1}^{-1}2^L}{q}\right)^2 \right),
\end{align*}
where we used that for $L,q$ large enough and $2^L\underset{0\leq i \leq 2^L+m-1}{\max} \left| \eta_{i} - u_{i-m}\right|$ small enough, \[f^m_{q,2^{-L}}\left(\int_0^1 f^m_{q,2^{-L}}(t)dt\right)^{-1}\geq \rho/8.\]

On the other hand, $f_0$ belongs to an interval of the form $\left[\rho/2;M\right]$ as $f_0$ is upper bounded by a constant depending on $\alpha$ and $ \norm{f_0}_{\Sigma(\alpha)}$ only since it is a Hölderian density (see \cite{tsyb_npe}, p.9). Also, from equations \eqref{partial_bias_2} and \eqref{bound_in_coef_spli}, $f^m_{\infty,2^{-L}}\in \left[\rho/4;2M\right]$ for $L$ large enough and $2^L\underset{0\leq i \leq 2^L+m-1}{\max} \left| \eta_{i} - u_{i-m}\right|$ small enough. We then use Taylor's inequality to write that

\begin{equation*}
\begin{array}{rcl}
V(f_0,  f^m_{\infty,2^{-L}}) &=& \int f_0\left(\log f_0- \log f^m_{\infty,2^{-L}}\right)^2 d\lambda\\
	&\lesssim & \norm{ f_0 - \restr{f^m_{\infty,2^{-L}}}{[0;1)}}_\infty^2
\end{array}
\end{equation*}
as well as \[ V\left(f_0, \restr{\frac{f^m_{q,2^{-L}}}{\int_0^1 f^m_{q,2^{-L}}(t)dt}}{[0;1)}\right)  \lesssim 	\norm{ f_0 -\restr{\frac{f^m_{q,2^{-L}}}{\int_0^1 f^m_{q,2^{-L}}(t)dt}}{[0;1)}}_\infty^2 \] from Lemma \ref{Discretiz_control} and $q$ large enough.
We conclude with the triangular inequality, equations \eqref{partial_bias_2}, \eqref{bound_in_coef_spli} and Lemma \ref{Discretiz_control}.
\end{proof}

\subsubsection{Bounds on the Hellinger distance.}
\begin{lemma}
\label{Hell_unif_cont}
Let $\left(\theta_i\right)_{i\in\Z}$ and $\left(\zeta_i\right)_{i\in\Z}$ be two $\left(2^L+m\right)$-periodic sequences of real positive numbers in $\mathcal{H}_{m, L}$ from \eqref{sieve_CPA}, where $L\in\N^*,m\in\N$ are such that $m<2^{L}-1$, and verifying
\[  \norm{\mathbf{\theta}}_\infty\vee \norm{\mathbf{\zeta}}_\infty\leq M\in\R_+^*.\]
If $f=\sum_{i\in\Z} \theta_i H_{Li}$ and $g=\sum_{i\in\Z} \zeta_i H_{Li}$, then, for $q\in\N^*$, $q>2^{L+1}mM$,
\begin{align*} &h\left( \restr{\frac{f^m_{q,2^{-L}}}{\int_0^1 f^m_{q,2^{-L}}(t)dt}}{[0;1)}, \restr{\frac{g^m_{q,2^{-L}}}{\int_0^1 g^m_{q,2^{-L}}(t)dt}}{[0;1)} \right) \\ &\leq 3^{1/4}\left(2^L+m\right)^{1/4} \norm{\mathbf{\theta}-\mathbf{\zeta}}_2^{1/2}+ 6^{1/4}\sqrt{\left(\frac{2^{L+1}Mm}{q-2^{L+1}Mm}\right)\left(1+mM\right)}, \end{align*}
as well as 
\[ h\left( \restr{ f^m_{\infty,2^{-L}}}{[0;1)}, \restr{g^m_{\infty,2^{-L}}}{[0;1)} \right) \leq  \left(2^L+m\right)^{1/4} \norm{\mathbf{\theta}-\mathbf{\zeta}}_2^{1/2}.\]

\end{lemma}

\begin{proof}
First, the same computation as in the proof of Lemma \ref{density_prior_int} shows that
\[
 \int_0^1 f^m_{\infty,2^{-L}}(t)dt= 1.
\]
Then, according to Lemma \ref{Discretiz_control}, there exists $V$ such that $|V|\leq \frac{2^{L+1}Mm}{q}$ and
 \[ \int_0^1 f^m_{q,2^{-L}}(t)dt = 1+V\] as $f$ takes values in $[0;2^LM]$.
 The same properties are verified with $g$, which allows us to write
 
 \begin{align*} 
&h\left( \restr{\frac{f^m_{q,2^{-L}}}{\int_0^1 f^m_{q,2^{-L}}(t)dt}}{[0;1)}, \restr{\frac{g^m_{q,2^{-L}}}{\int_0^1 g^m_{q,2^{-L}}(t)dt}}{[0;1)} \right)\\
 &=  \Bigg(\int_0^1 \left( \sqrt{\frac{1}{(1+V)q^m}\sum_{(i_1,\dots,i_m)\in [\![0;q-1]\!]^m} f\Bigg(u-\frac{i_1+\dots+i_m}{q2^L}\Bigg)}\right. \\
 &\qquad -\left. \sqrt{\frac{1}{(1+V')q^m}\sum_{(i_1,\dots,i_m)\in [\![0;q-1]\!]^m} g\Bigg(u-\frac{i_1+\dots+i_m}{q2^L}\Bigg)} \right)^2 du \Bigg)^{1/2}  \\
&\leq  \Bigg(\int_0^1 \left| \sum_{(i_1,\dots,i_m)\in [\![0;q-1]\!]^m}  \frac{f\big(u-\frac{i_1+\dots+i_m}{q2^L}\big)}{(1+V)q^m} - \frac{g\big(u-\frac{i_1+\dots+i_m}{q2^L}\big)}{(1+V')q^m} \right| du \Bigg)^{1/2} \\
&\leq  \Bigg( \frac{1}{q^m} \sum_{(i_1,\dots,i_m)\in [\![0;q-1]\!]^m}   \int_{-\frac{i_1+\dots+i_m}{q2^L}}^{1-\frac{i_1+\dots+i_m}{q2^L}} \left|  \frac{f\big(u\big)}{(1+V)} - \frac{g\big(u\big)}{(1+V')} \right| du \Bigg)^{1/2} \\
&\leq \Bigg(  \int_{-m2^{-L}}^{1} \left|  \frac{f\big(u\big)}{(1+V)} - \frac{g\big(u\big)}{(1+V')} \right| du \Bigg)^{1/2}\\
&\leq \Bigg(  (1+m2^{-L})^{1/2} \Bigg[\int_{-m2^{-L}}^{1} \left(  \frac{f\big(u\big)}{(1+V)} - \frac{g\big(u\big)}{(1+V')} \right)^2 du\Bigg]^{1/2} \Bigg)^{1/2} \\
&= \Bigg(  (1+m2^{-L})^{1/2} \Bigg[\sum_{i=0}^{2^L+m-1} 2^L\left( \frac{\theta_i}{(1+V)} - \frac{\zeta_i}{(1+V')} \right)^2\Bigg]^{1/2} \Bigg)^{1/2} \\
&\leq 3^{1/4}2^{L/4}(1+m2^{-L})^{1/4}\Bigg[\sum_{i=0}^{2^L+m-1} \left( \theta_i-\zeta_i \right)^2+\sum_{i=0}^{2^L+m-1} \left( \frac{V}{1+V}\theta_i\right)^2+\sum_{i=0}^{2^L+m-1} \left(\frac{V'}{1+V'} \zeta_i \right)^2\Bigg]^{1/4}\\
&\leq 3^{1/4}2^{L/4}(1+m2^{-L})^{1/4} \Bigg[\norm{\mathbf{\theta}-\mathbf{\zeta}}_2^2+\left(\frac{V}{1+V}\right)^2\left(\sum_{i=0}^{2^L+m-1} \theta_i\right)^2+\left(\frac{V'}{1+V'}\right)^2\left(\sum_{i=0}^{2^L+m-1} \zeta_i\right)^2\Bigg]^{1/4}\\
&\leq 3^{1/4}2^{L/4}(1+m2^{-L})^{1/4}\Bigg[\norm{\mathbf{\theta}-\mathbf{\zeta}}_2^2+2\left(\frac{2^{L+1}Mm}{q-2^{L+1}Mm}\right)^2\left(1+mM\right)^2\Bigg]^{1/4}\\
&=3^{1/4}\left(2^L+m\right)^{1/4} \norm{\mathbf{\theta}-\mathbf{\zeta}}_2^{1/2}+ 6^{1/4}\sqrt{\left(\frac{2^{L+1}Mm}{q-2^{L+1}Mm}\right)\left(1+mM\right)}.\end{align*}
Above, we have used that, since $\omega_{m,l}=1- \omega_{m,m+1-l}$ from Lemma \ref{int_chi_formula}, 
  \begin{align*} 
  \sum_{i=0}^{2^L+m-1} \theta_i&= 1+  \sum_{i=2^L-m}^{2^L-1}\left(\omega_{m,m+1-\left(2^L-i\right)}\theta_i+\omega_{m,2^L-i}\theta_{i+m}\right)\\
  &\leq 1+ \sum_{i=2^L-m}^{2^L-1}\left(\omega_{m,m+1-\left(2^L-i\right)}+\omega_{m,2^L-i}\right)M\\
  &\leq 1 + mM.
  \end{align*}
Also,
  \begin{align*} 
h\left( \restr{ f^m_{\infty,2^{-L}}}{[0;1)}, \restr{g^m_{\infty,2^{-L}}}{[0;1)} \right)&=  \Bigg(\int_0^1 \left( \sqrt{\sum_{i\in\Z} \theta_i 2^L\chi^{*m}\left(2^Lu-i\right)}  - \sqrt{\sum_{i\in\Z} \zeta_i 2^L\chi^{*m}\left(2^Lu-i\right)} \right)^2 du \Bigg)^{1/2}  \\
&\leq  \Bigg(\int_0^1 \left| \sum_{i\in\Z} (\theta_i -\zeta_i) 2^L\chi^{*m}\left(2^Lt-i\right) \right| du \Bigg)^{1/2} \\
&\leq \Bigg( \sum_{i\in\Z} \left|\theta_i -\zeta_i \right|  2^L \int_0^1  \chi^{*m}\left(2^Lt-i\right)du \Bigg)^{1/2} \\
&\leq \Bigg( \sum_{i=0}^{2^L+m-1}  \left|\theta_i -\zeta_i \right|   \Bigg)^{1/2} \\
&\leq \left(2^L+m\right)^{1/4}\Bigg( \sum_{i=0}^{q+m-1}  \left|\theta_i -\zeta_i \right|^2   \Bigg)^{1/4} \\
&\leq \left(2^L+m\right)^{1/4} \norm{\mathbf{\theta}-\mathbf{\zeta}}_2^{1/2}.
\end{align*}
 
\end{proof}


\subsection{Spline prior $SPT$.}

\subsubsection{Bounds on the Kullback-Leibler divergence}


\begin{lemma}
\label{KL_ball_control}
Suppose $f_0\in  \Sigma(\alpha, [0,1))$, for $\alpha>0$, is a probability density and $f_0\geq\rho$ for some $\rho>0$. Define $(\eta_i )_{0\leq i\leq 2^L-1}$ as the sequence from Lemma \ref{bias_bnd_ext_pol}. Then, for $\tau>0$ small enough, $L\in\N^*$ large enough, $(\Theta_i )_{0\leq i\leq 2^L-1}\in S^{2^L}$ and \[f=SD_{\tau, \lfloor \alpha \rfloor, 2^{-L}}\left(\sum_{i=0}^{2^{L_n}-1} \Theta_i H_{Li}\right),\]we have that there exists a constant $C$ depending only on $\rho,\ \alpha$, $\norm{f_0}_{\infty}$ and $\norm{f_0}_{\Sigma(\alpha)}$ such that
 \[ K(f_0, f)\vee V(f_0, f) \leq C\left( 2^{-2\alpha L} + \tau^2 + \left(2^L\tau^{-1} \vee 2^{2L}\tau^{-2}\right)^2 \underset{0\leq i \leq 2^L-1}{\max} \left| \eta_{i} - \Theta_{i}\right|^2  \right) \]
 for $\left(2^L\tau^{-1} \vee 2^{2L}\tau^{-2}\right)\underset{0\leq i \leq 2^L-1}{\max} \left| \eta_{i} - \Theta_{i}\right|$ small enough.
\end{lemma}
\begin{proof}

Let $f_0^L$ be the map obtained from Lemma \ref{bias_bnd_ext_pol}, then
\[ f_0^L= SD_{\tau, \lfloor \alpha \rfloor, 2^{-L}}\left(\sum_{i=0}^{2^{L}-1} \eta_i  H_{Li}\right)\]
for some $(\eta_i )_{0\leq i\leq 2^L-1}\in S^{2^L}$. We now give the decomposition, given $A$ and $B$ exist (it will be shown later),
\begin{equation*}
\begin{array}{rcl}
K(f_0, f) &=& \int f_0\log\left(\frac{f_0}{f}\right) d\lambda\\
      &=& \underbrace{\int f_0\log\left(\frac{f_0}{ f_0^{L} }\right)d\lambda}_\text{= A} + \underbrace{\int f_0\log\left(\frac{ f_0^{L} }{f}\right)d\lambda}_\text{= B}.
\end{array}
\end{equation*}
Focusing on the first term,  the bound $\log(1+u)\leq u$ for $u>-1$ results in 
\begin{equation}
\label{firstBoundA}
\begin{array} {lcl} 
A &\leq& \int_0^1 f_0\frac{f_0-f_0^{L}}{f_0^{L}}d\lambda \\
&=& \int_0^1 \frac{(f_0-f_0^{L})^2}{f_0^{L}}d\lambda + 1 - \int_0^1 f_0^{L}d\lambda\\
&\leq& \frac{2}{\rho}||f_0-f_0^{L}||_\infty^2 \quad \text{for $L$ large enough,}
\end{array}
\end{equation}
where we used that, by construction,  $\int_0^1 f_0^{L}(t) dt =1$.
Also, we have lower bounded  $f_0^{L}$ by $\rho/2$ for $L$ large enough and $\tau$ small enough, as a consequence from our assumption on $f_0$ and Lemma \ref{bias_bnd_ext_pol}. On the other hand, using once again that $f_0^{L}$ is a density, we have the upper bound 
\begin{equation}
\label{firstBoundB}
\begin{array} {lcl} 
B &\leq& \int_0^1 f_0\frac{f_0^{L}-f}{f}d\lambda \\
&=& \int_0^1 \frac{(f_0^{L}-f)(f_0-f)}{f}d\lambda\\
&=& \int_0^1 \frac{(f_0^{L}-f)(f_0-f_0^L)}{f}d\lambda + \int_0^1 \frac{(f_0^{L}-f)^2}{f}d\lambda\\
&\leq& \frac{1}{2}\int_0^1 \frac{(f_0-f_0^L)^2}{f}d\lambda + \frac{3}{2}\int_0^1 \frac{(f_0^{L}-f)^2}{f}d\lambda\ \text{using that $2ab\leq a^2+b^2$ for $a,b$ real numbers.}
\end{array}
\end{equation} 
Along with the lower bound on $f_0^L$, Lemma \ref{sup_control} ensures that there exists a constant $C$ depending on $\rho$ and $\alpha$ only such that, when $L$ is large enough and  $\left(2^L\tau^{-1} \vee 2^{2L}\tau^{-2}\right)\underset{0\leq i \leq 2^L-1}{\max} \left| \eta_{i} - \Theta_{i}\right|$ is small enough,

\[ B \leq C \Bigg( \norm{f_0 - f_0^{L} }_\infty^2 +  \left(2^L\tau^{-1} \vee 2^{2L}\tau^{-2}\right)^2 \underset{0\leq i \leq 2^L-1}{\max} \left| \eta_{i} - \Theta_{i}\right|^2   \Bigg). \]
In the end, \eqref{firstBoundA}, \eqref{firstBoundB} and Lemma \ref{bias_bnd_ext_pol} lead to the bound
\begin{equation*}
        K(f_0, f) \leq C\left( 2^{-2\alpha L} + \tau^2 + \left(2^L\tau^{-1} \vee 2^{2L}\tau^{-2}\right)^2 \underset{0\leq i \leq 2^L-1}{\max} \left| \eta_{i} - \Theta_{i}\right|^2  \right)
\end{equation*}
where the constant $C$ only depends on $\alpha$, $\rho$, $\norm{f_0}_{\infty}$ and $\norm{f_0}_{\Sigma(\alpha)}$.\\

\noindent In a second part, we write that 
\begin{equation*}
\begin{array}{rcl}
V(f_0, f) &=& \int f_0\log\left(\frac{f_0}{f}\right)^2 d\lambda\\
      &\leq& 2\underbrace{\int f_0\log\left(\frac{f_0}{f_0^{L}}\right)^2d\lambda}_\text{= A'} + 2\underbrace{\int f_0\log\left(\frac{f_0^{L}}{f}\right)^2d\lambda}_\text{= B'}.
\end{array}
\end{equation*}
For the first term, introducing the Lebesgue-measurable event $G=\left\{x\in[0;1)\ \Big|\ f_0(x)>f_0^{L}(x)\right\}$, we use similar arguments as above to write that
\begin{equation}
\label{bndAprime}
\begin{array} {lcl} 
A'&=& \int_G f_0\log\left(\frac{f_0}{f_0^{L}}\right)^2d\lambda + \int_{G^c} f_0\log\left(\frac{f_0^{L}}{f_0}\right)^2d\lambda\\
&\leq& \int_G f_0\left(\frac{f_0-f_0^{L}}{f_0^{L}}\right)^2d\lambda + \int_{G^c} f_0\left(\frac{f_0^{L}-f_0}{f_0}\right)^2d\lambda \\
&\leq& C\norm{f_0-f_0^{L}}_\infty^2
\end{array}
\end{equation} 
where $C$ depends on $\rho$ only and the last inequality is valid for $L$ large enough and $\tau$ small enough.
Similarly, for the second term, introducing the Lebesgue-measurable event $H=\left\{x\in[0;1)\ \big|\ f_0^{L}(x)>f(x)\right\}$ and using Lemma \ref{sup_control}, it follows that
\begin{equation}
\label{bndBprime}
\begin{array} {lcl} 
B'&\leq& \int_H f_0\left(\frac{f_0-f}{f}\right)^2d\lambda + \int_{H^c} f_0\left(\frac{f_0^{L}-f}{f_0^{L}}\right)^2 d\lambda\\
&=& C  \left(2^L\tau^{-1} + 2^{2L}\tau^{-2}\right)^2 \underset{0\leq i \leq 2^L-1}{\max} \left| \eta_{i} - \theta_{i}\right|^2 
\end{array}
\end{equation} for some $C$, which is valid for $L$ large enough and $\left(2^L\tau^{-1} \vee 2^{2L}\tau^{-2}\right)\underset{0\leq i \leq 2^L-1}{\max} \left| \eta_{i} - \Theta_{i}\right|$ small enough.
Finally, we obtain from  \eqref{bndAprime}, \eqref{bndBprime} and Lemma \ref{bias_bnd_ext_pol}
\begin{equation*}
V(f_0, f) \leq C\left( 2^{-2\alpha L} + \tau^2 + \left(2^L\tau^{-1} \vee 2^{2L}\tau^{-2}\right)^2 \underset{0\leq i \leq 2^L-1}{\max} \left| \eta_{i} - \Theta_{i}\right|^2  \right)
\end{equation*}
with $C$ depending on $\alpha$, $\rho$, $\norm{f_0}_{\infty}$ and $\norm{f_0}_{\Sigma(\alpha)}$.
\end{proof}


\subsubsection{Bound on the Hellinger distance}
\begin{lemma}
\label{Hellinger_control}
Let $\tau>0$, $m\in\N$, $L\in\N^*$ and two density functions $g_1, g_2$ on $[0;1)$ which are piecewise constant on the grid $[i2^{-L}; (i+1)2^{-L})$, $0\leq i\leq 2^L-1$. Then, for  $f_i= SD_{\tau, m, 2^{-L}}(g_i)$,
\begin{equation*}
h\left(f_1, f_2 \right)  \leq 2\left(1+\sqrt{1+2\left(m +1\right)^3 e^{\sqrt{6\left(m +1\right)}m}}\right)\tau^{-1/2}\norm{g_{1} -  g_{2} }_2^{1/2}.
\end{equation*}
\end{lemma}
\begin{proof}
By definition, \[f_i= \frac{A^2_{m, 2^{-L}}(g_i)_++\tau}{\int_{[0;1)} (A^2_{m, 2^{-L}}(g_i)_++\tau)d\lambda}.\]Then the triangle inequality, its reversed version and simple algebra give
\begin{align*}\begin{split}
h\left(f_1, f_2 \right) &= \norm{ \frac{\sqrt{A^2_{m, 2^{-L}}(g_1)_++\tau}}{\norm{ \sqrt{A^2_{m, 2^{-L}}(g_1)_++\tau} }_2} - \frac{\sqrt{A^2_{m, 2^{-L}}(g_2)_++\tau}}{\norm{ \sqrt{A^2_{m, 2^{-L}}(g_2)_++\tau} }_2}  }_2 \\
&\leq 2\frac{\norm{ \sqrt{A^2_{m, 2^{-L}}(g_1)_++\tau} - \sqrt{A^2_{m, 2^{-L}}(g_2)_++\tau} }_2}{\norm{ \sqrt{A^2_{m, 2^{-L}}(g_1)_++\tau} }_2}.
\end{split}\end{align*}
The numerator on the right hand side is then bounded by 
\begin{align*}\begin{split}
&\norm{A^2_{m, 2^{-L}}(g_1)_+ -A^2_{m, 2^{-L}}(g_2)_+}_1^{1/2}\leq  \norm{A^1_{m, 2^{-L}}(g_1) -A^1_{m, 2^{-L}}(g_2)}_1^{1/2} \\
&\qquad\qquad\qquad+\norm{\left(A^2_{m, 2^{-L}}(g_1)_+-A^1_{m, 2^{-L}}(g_1)\right) -\left(A^2_{m, 2^{-L}}(g_2)_+-A^1_{m, 2^{-L}}(g_2)\right)}_1^{1/2} 
\end{split}\end{align*}
as for any $a,b\geq0$, $|\sqrt{a}-\sqrt{b}|\leq \sqrt{|a-b|}$. For the first term, with $\Bar{g}_i$ the $1$-periodic extension of $g_i$, we develop

\begin{equation*}
\begin{array} {lcl} 
\norm{A^1_{m, 2^{-L}}(g_1) -A^1_{m, 2^{-L}}(g_2)}_1&=&\int_0^1 \left|  \Bar{g}_{1,\infty,2^{-L}}^{m}(t)  -  \Bar{g}_{2,\infty,2^{-L}}^{m}(t)  \right| dt  \\
&=&  \int_0^1 \left|  \chi^{m}_{2^{-L}} *(\Bar{g}_1-\Bar{g}_2) (t)   \right| dt \\
&\leq& \int_0^1 \chi^{m}_{2^{-L}} *(|\Bar{g}_1-\Bar{g}_2|) (t)  dt  \\
&=& \int_{\mathbb{R} } \chi^{m}_{2^{-L}}(x)  \int_0^1 |\Bar{g}_1-\Bar{g}_2| (t-x)  dt dx \\
&=& \int_{\mathbb{R} } \chi^{m}_{2^{-L}}(x)  \int_0^1 |g_1-g_2| (t)  dt dx  \\
&=& \norm{g_{1} -  g_{2} }_1,
\end{array}
\end{equation*}
following Lemma \ref{int_chi_formula}. Then, according to the link between $A^1_{m, 2^{-L}}(g_i)$ and $A^2_{m, 2^{-L}}(g_i)$,  the square of the second term is equal to 
\begin{equation*}
\label{decomp_ext}
\begin{array} {lcl} 
&\bigintss_{\left[0;\frac{m +1}{2^L}\right)\bigcup\left[1-\frac{m -1}{2^L};1\right)}  \Bigg|\left(A^2_{m, 2^{-L}}(g_1)_+-A^1_{m, 2^{-L}}(g_1)\right) -\left(A^2_{m, 2^{-L}}(g_2)_+-A^1_{m, 2^{-L}}(g_1)\right)\Bigg| d\lambda\\
&\leq  \norm{A^1_{m, 2^{-L}}(g_1) -A^1_{m, 2^{-L}}(g_2)}_1 +  \bigintss_{\left[0;\frac{m +1}{2^L}\right)\bigcup\left[1-\frac{m -1}{2^L};1\right)}  \Bigg| A^2_{m, 2^{-L}}(g_1) -A^2_{m, 2^{-L}}(g_2)\Bigg| d\lambda.
\end{array}
\end{equation*}
Now, writing $P_1,\ P_2$ the polynomials of degree $m$ such that $A^1_{m, 2^{-L}}(g_i)(t) = P_i(t) \text{ for } t \in \left[\frac{m}{2^L}; \frac{m +1}{2^L}\right)$. The polynomials $Q_i=P_i\circ(\cdot / \frac{m +1}{2^L})$ have degree $m$ and, by definition, we have the two equalities, 
\[ \int_0^{\frac{m +1}{2^L}} \left|A^2_{m, 2^{-L}}(g_1)(t)-A^2_{m, 2^{-L}}(g_2)(t)\right|dt = 2^{-L}\left(m +1\right) \int_0^{1} |(Q_1-Q_2 )(t)|dt,  \]
\[ \int_{\frac{m}{2^L}}^{\frac{m +1}{2^L}} \left|A^2_{m, 2^{-L}}(g_1)(t)-A^2_{m, 2^{-L}}(g_2)(t)\right|dt = 2^{-L}\left(m +1\right) \int_{\frac{m}{m + 1}}^{1} |(Q_1-Q_2 )(t)|dt . \]
It remains to apply Lemma \ref{ineq_norms} to obtain
\begin{align*}
 &\pushleft{\int_0^{\frac{m +1}{2^L}} \left|A^2_{m, 2^{-L}}(g_1)(t)-A^2_{m, 2^{-L}}(g_2)(t)\right|dt} \\
 &\leq 2\left(m +1\right)^3 e^{\sqrt{6\left(m +1\right)}m}\int_{\frac{m}{2^L}}^{\frac{m +1}{2^L}} \left|A^2_{m, 2^{-L}}(g_1)(t)-A^2_{m, 2^{-L}}(g_2)(t)\right|dt\\
 &\leq 2\left(m +1\right)^3 e^{\sqrt{6\left(m +1\right)}m}\int_{\frac{m}{2^L}}^{\frac{m +1}{2^L}} \left|A^1_{m, 2^{-L}}(g_1)(t)-A^1_{m, 2^{-L}}(g_2)(t)\right|dt \\
 &\leq 2\left(m +1\right)^3 e^{\sqrt{6\left(m +1\right)}m} \norm{g_{1} -  g_{2} }_1.
\end{align*}
 Finally, we obtain the bound
\begin{align*}\begin{split}
h\left(f_1, f_2 \right)  &\leq 2\frac{\left(1+\sqrt{1+2\left(m +1\right)^3 e^{\sqrt{6\left(m +1\right)}m}}\right)\norm{g_{1} -  g_{2} }_1^{\frac{1}{2}}}{\norm{ \sqrt{A^2_{m, 2^{-L}}(g_1)_++\tau}}_2}\\
& \leq2\left(1+\sqrt{1+2\left(m +1\right)^3 e^{\sqrt{6\left(m +1\right)}m}}\right)\tau^{-1/2}\norm{g_{1} -  g_{2} }_2^{\frac{1}{2}}.
\end{split}\end{align*}
\end{proof}


\subsubsection{Bounds on sup-norm distance.}

\begin{lemma}
\label{sup_control}
Let $L\in\mathbb{N}$ and $\left(\Theta_{1,i}\right)_{0\leq i\leq 2^L-1}$, $\left(\Theta_{2,i}\right)_{0\leq i\leq 2^L-1}$ be two elements in $S^{2^L}$. Then, for $0\leq m\leq 2^L-1$ and $\epsilon>0$, there exists a contant $C$ depending only on $m$ such that

\begin{align*}
\begin{split}
&\norm{SD_{\tau,m,2^{-L}}\left(\sum_{0\leq i \leq 2^L-1}  \Theta_{1,i} H_{Li} \right) - SD_{\tau, m,2^{-L}}\left(\sum_{0\leq i \leq 2^L-1}  \Theta_{2,i} H_{Li}\right)  }_\infty\\
&\quad \leq C \left(2^L\tau^{-1} \vee 2^{2L}\tau^{-2}\right) \underset{0\leq i \leq 2^L}{\max} \left| \Theta_{1,i} - \Theta_{2,i}\right|.
\end{split}
\end{align*}

\end{lemma}

\begin{proof}
From \eqref{convol_cont_aggreg} and Lemma \ref{int_chi_formula}, it is straightforward that 

\begin{align}
\label{convol_sup_cons}
\begin{split}
&\norm{A^1_{m,2^{-L}}\left( \sum_{0\leq i \leq 2^L-1}  \Theta_{1,i} H_{Li} \right) - A^1_{m,2^{-L}}\left( \sum_{0\leq i \leq 2^L-1}  \Theta_{2,i} H_{Li} \right)  }_\infty \\
&\qquad \leq \norm{\sum_{0\leq i \leq 2^L-1}  \Theta_{1,i} H_{Li} - \sum_{0\leq i \leq 2^L-1}  \Theta_{2,i} H_{Li}  }_\infty\\
&\qquad=2^L \underset{0\leq i \leq 2^L-1}{\max} \left| \Theta_{1,i} - \Theta_{2,i}\right|.
\end{split}
\end{align}
Now, we point out that, for $J\neq \varnothing$ an interval in $[0;1)$,  \[\norm{f}_{\infty, J}\coloneqq \underset{x\in J}\sup |f(x)|\] defines a norm on the space of polynomials of degree at most $m$, which is of finite dimension. Therefore there exist constants $C_{1,m}\geq 1$ and $C_{2,m}\geq 1$ such that

\[ \norm{\cdot}_{\infty, [0; 1)} \leq C_{1,m} \norm{\cdot}_{\infty, \left[0; (m+1)^{-1}\right)} \text{ and } \norm{\cdot}_{\infty, [0; 1)} \leq C_{2,m} \norm{\cdot}_{\infty, \left[m(m+1)^{-1};1\right)}.\]
Also, let's write $P_j$ (resp. $Q_j$) the polynomial of degree $m$ such that \[A^1_{m,2^{-L}}\left( \sum_{0\leq i \leq 2^L-1}  \Theta_{j,i} H_{Li} \right) = P_j(2^L(m+1)^{-1},\ \cdot)\quad\text{ resp. $Q_j\left(2^L(m+1)^{-1}\Big(\cdot-1+ (m+1)2^{-L}\Big)\right)$,}\] on the  interval $ \left[m2^{-L}; (m+1)2^{-L}\right)$ (resp. $\left[1- (m+1)2^{-L};  1-m2^{-L}\right)$). These polynomials exist according to Lemma \ref{map_to_splines}. It follows that, by definition,

\begin{align*}
\begin{split}
&\underset{t\in\left[0;(m+1)2^{-L}\right)}{\sup}\Bigg| A^2_{m,2^{-L}}\left( \sum_{0\leq i \leq 2^L-1}  \Theta_{1,i} H_{Li} \right)(t) - A^2_{m,2^{-L}}\left( \sum_{0\leq i \leq 2^L-1}  \Theta_{2,i} H_{Li} \right)(t)  \Bigg|\\
&\qquad=\underset{t\in\left[0;(m+1)2^{-L}\right)}{\sup} \left|  P_1(2^L(m+1)^{-1}t) -  P_2(2^L(m+1)^{-1}t )  \right|\\
&\qquad=\underset{t\in[0;1)}{\sup}\left|  P_1(t) -  P_2(t)  \right|\\
&\qquad\leq C_{2,m} \underset{t\in\left[m(m+1)^{-1};1\right)}{\sup}\left|  P_1(t) -  P_2(t)  \right|\\
&\qquad= C_{2,m} \underset{t\in\left[m2^{-L}; (m+1)2^{-L}\right)}{\sup}\left|  P_1(2^L(m+1)^{-1}t) -  P_2(2^L(m+1)^{-1}t)  \right|\\
&\qquad= C_{2,m} \underset{t\in\left[m2^{-L}; (m+1)2^{-L}\right)}{\sup}\Bigg| A^2_{m,2^{-L}}\left( \sum_{0\leq i \leq 2^L-1}  \Theta_{1,i} H_{Li} \right)(t) - A^2_{m,2^{-L}}\left( \sum_{0\leq i \leq 2^L-1}  \Theta_{2,i} H_{Li} \right)(t)  \Bigg|\\
&\qquad= C_{2,m} \underset{t\in\left[m2^{-L}; (m+1)2^{-L}\right)}{\sup}\Bigg| A^1_{m,2^{-L}}\left( \sum_{0\leq i \leq 2^L-1}  \Theta_{1,i} H_{Li} \right)(t) - A^1_{m,2^{-L}}\left( \sum_{0\leq i \leq 2^L-1}  \Theta_{2,i} H_{Li} \right)(t)  \Bigg|.
\end{split}
\end{align*}
Therefore, using the same arguments on$\left[1- (m+1)2^{-L};  1\right)$ and the fact that \[A^1_{m,2^{-L}}\left( \sum_{0\leq i \leq 2^L-1}  \Theta_{j,i} H_{Li} \right)\] and \[A^2_{m,2^{-L}}\left( \sum_{0\leq i \leq 2^L-1}  \Theta_{j,i} H_{Li} \right)\] are equal on $\left[(m+1)2^{-L}; 1-(m+1)2^{-L}\right)$,

\begin{align}
\label{pol_ext_cons}
\begin{split}
&\norm{A^2_{m,2^{-L}}\left( \sum_{0\leq i \leq 2^L-1}  \Theta_{1,i} H_{Li} \right) - A^2_{m,2^{-L}}\left( \sum_{0\leq i \leq 2^L-1}  \Theta_{2,i} H_{Li} \right) }_\infty \\
&\quad \leq \max\left(C_{1,m} , C_{2,m} \right) \norm{A^1_{m,2^{-L}}\left( \sum_{0\leq i \leq 2^L-1}  \Theta_{1,i} H_{Li} \right) - A^1_{m,2^{-L}}\left( \sum_{0\leq i \leq 2^L-1}  \Theta_{2,i} H_{Li} \right)  }_\infty.
\end{split}
\end{align}
Then, from Lemma \ref{int_chi_formula} and the hypothesis on the $\left(\Theta_{j,i}\right)_{0\leq i\leq 2^L-1}$ sequences, we also notice that 
\[ \norm{A^1_{m,2^{-L}}\left( \sum_{0\leq i \leq 2^L-1}  \Theta_{j,i} H_{Li} \right) }_\infty \leq 2^L.\]
Therefore, the same arguments as above gives
\[ \norm{A^2_{m,2^{-L}}\left( \sum_{0\leq i \leq 2^L-1}  \Theta_{j,i} H_{Li} \right)  }_\infty \leq 2^L\max\left(C_{1,m} , C_{2,m} \right). \]
Finally, denoting 
\[I_j =\int_0^1 \left[A^2_{m,2^{-L}}\left( \sum_{0\leq i \leq 2^L-1}  \Theta_{j,i} H_{Li} \right)_+ +\tau\right] d\lambda, \]
\eqref{sd_map_def2} gives
\begin{align}
\label{norm_sup_cons}
\begin{split}
&\norm{SD_{\tau,m,2^{-L}}\left(\sum_{0\leq i \leq 2^L-1}  \Theta_{1,i} H_{Li} \right) - SD_{\tau, m,2^{-L}}\left(\sum_{0\leq i \leq 2^L-1}  \Theta_{2,i} H_{Li}\right)  }_\infty\\
&\quad \leq \tau^{-1}\norm{A^2_{m,2^{-L}}\left( \sum_{0\leq i \leq 2^L-1}  \Theta_{1,i} H_{Li} \right) - A^2_{m,2^{-L}}\left( \sum_{0\leq i \leq 2^L-1}  \Theta_{2,i} H_{Li} \right)  }_\infty\\
&\qquad + \norm{A^2_{m,2^{-L}}\left( \sum_{0\leq i \leq 2^L-1}  \Theta_{2,i} H_{Li} \right)  }_\infty  \left|\frac{1}{I_1 }-\frac{1}{I_2}\right|\\
&\quad \leq \left( \tau^{-1} + 2^L\max\Big(C_{1,m} , C_{2,m} \Big)\tau^{-2}\right)\\
&\qquad \norm{A^2_{m,2^{-L}}\left( \sum_{0\leq i \leq 2^L-1}  \Theta_{1,i} H_{Li} \right) - A^2_{m,2^{-L}}\left( \sum_{0\leq i \leq 2^L-1}  \Theta_{2,i} H_{Li} \right) }_\infty.
\end{split}
\end{align}
Combining \eqref{convol_sup_cons}, \eqref{pol_ext_cons} and \eqref{norm_sup_cons} concludes the proof.

\end{proof}

\begin{lemma}
\label{bias_bnd_ext_pol}
Let $f_0\in  \Sigma(\alpha, [0,1))$, for $\alpha>0$, be a probability density function such that $f_0\geq\rho$ for some $\rho>0$. Then, for $L\in\N$ large enough, there exists $(\eta_i )_{0\leq i\leq 2^L-1}\in S^{2^L}$ and a constant $C$ depending only on $\alpha$, $\norm{f_0}_{\infty}$ and $\norm{f_0}_{\Sigma(\alpha)}$ such that
\[ \norm{f_0 -  SD_{\tau, \lfloor \alpha \rfloor, 2^{-L}}\left(\sum_{i=0}^{2^{L}-1} \eta_i  H_{Li}\right)	}_\infty \leq C \left(2^{-\alpha L} + \tau \right)\]
for any $\tau>0$ small enough.
\end{lemma}

\begin{proof}

Let's write  $\Tilde{f}_0^{L}= \sum_{i=0}^{2^{L}-1}\eta_i S_{i,2^{L},\lfloor \alpha \rfloor}$ the application from Lemma \ref{approx_cyclicSplines} such that

\begin{equation}\label{partial_bias} \norm{ \restr{f_0}{ \left[\lfloor \alpha \rfloor 2^{-L};1-\lfloor \alpha \rfloor 2^{-L}\right) }-\restr{\Tilde{f}_0^{L}}{ \left[\lfloor \alpha \rfloor 2^{-L};1-\lfloor \alpha \rfloor 2^{-L}\right)  } }_{\infty} \leq C2^{-\alpha L} \end{equation} 
with $(\theta_i)_{0\leq i\leq 2^{L}-1} \in S^{2^{L}}$.
Then, we see, from definitions \eqref{first_step_map}, \eqref{period_spline_bas_def} and equation \eqref{spline_new_def} that \[\Tilde{f}_0^{L}=A^1_{\lfloor \alpha \rfloor, 2^{-L}}\left( \sum_{0\leq i \leq 2^L-1}  \eta_{i} H_{Li} \right)\] and we introduce \[ f_0^{L} =  SD_{\tau, \lfloor \alpha \rfloor, 2^{-L}}\left(\sum_{0\leq i \leq 2^L-1}  \eta_{i} H_{Li}\right). \] 
Besides, by construction (see Subsection \ref{ext_new_priors}) and from \eqref{sd_map_def}, there exists a polynomial $P$ of degree $\lfloor \alpha \rfloor$  such that
\begin{equation}\label{pol_H2} \forall t \in \left[0; \frac{\lfloor \alpha \rfloor+1}{2^L}\right),\qquad A^2_{\lfloor \alpha \rfloor, 2^{-L}}\left( \sum_{0\leq i \leq 2^L-1}  \eta_{i} H_{Li} \right)(t) = P\left( \frac{2^Lt}{\lfloor \alpha \rfloor+1}\right).\end{equation}
We point out that we also have, for any $t\in\left[\frac{\lfloor \alpha \rfloor}{2^L}; \frac{\lfloor \alpha \rfloor+1}{2^L}\right)$,
\begin{equation}\label{H1_H2} A^2_{\lfloor \alpha \rfloor, 2^{-L}}\left( \sum_{0\leq i \leq 2^L-1}  \eta_{i} H_{Li} \right)(t) = A^1_{\lfloor \alpha \rfloor, 2^{-L}}\left( \sum_{0\leq i \leq 2^L-1}  \eta_{i} H_{Li} \right)(t).\end{equation}
Also, if $\mathbb{R}_{\lfloor \alpha \rfloor}\left[X\right]$ is the space of all polynomials of degree at most $\lfloor \alpha \rfloor$, it is well-known fact that there exists a polynomial $Q_1\in\mathbb{R}_{\lfloor \alpha \rfloor}\left[X\right]$  such that $\forall t \in \left[0; \frac{\lfloor \alpha \rfloor+1}{2^L}\right)$, $\big|f_0(t) - Q_1(t)\big|\leq C 2^{-\alpha L}$, with $C$ a constant depending only on $\alpha$ and $\norm{f_0}_{\Sigma(\alpha)}$. Let $Q(t)=Q_1\left( \frac{\lfloor \alpha \rfloor+1}{2^L}t\right)$ for $t\in [0;1)$.
Let's define on $\mathbb{R}_{\lfloor \alpha \rfloor}\left[X\right]$, for $J$ an interval in $[0;1)$, the norm
\[ ||g||_{\infty, J} \coloneqq \underset{x\in J}{\sup} \left|g(x)\right|.\]
By equivalence of norms on $\mathbb{R}_{\lfloor \alpha \rfloor}\left[X\right]$, $||g||_{\infty, [0;1)}  \leq C||g||_{\infty, \left[\lfloor \alpha \rfloor(\lfloor \alpha \rfloor+1)^{-1};1\right)}$, with $C$ a constant depending on $\lfloor \alpha \rfloor$ only.

Now, using that, by definition, $\Tilde{f}_0^{L}$ and $Q_1$ are both close to $f_0$ on some intervals as shown above, we deduce from the last paragraph, \eqref{pol_H2}, \eqref{H1_H2} and the bound \eqref{partial_bias}
\begin{align*}
 ||Q-P||_{\infty, \left[0;1\right)} &= \underset{x\in \left[0;1\right]}{\sup} \left|Q(x)-P(x)\right|\\
&= \underset{t\in\left[0; \frac{\lfloor \alpha \rfloor+1}{2^L}\right]}{\sup} \left|Q\Bigg( \frac{2^Lt}{\lfloor \alpha \rfloor+1}\Bigg)-P\Bigg( \frac{2^Lt}{\lfloor \alpha \rfloor+1}\Bigg)\right| \\
&\leq  \underset{t \in \left[0; \frac{\lfloor \alpha \rfloor+1}{2^L}\right)}{\sup} \left|f_0(t)-Q_1(t)\right| + \underset{t \in \left[0; \frac{\lfloor \alpha \rfloor+1}{2^L}\right)}{\sup} \Bigg|f_0(t)-P\left( \frac{2^Lt}{\lfloor \alpha \rfloor+1}\right)\Bigg|\\
& \leq C 2^{-\alpha L}
 \end{align*}
where $C$ depends on $\alpha$, $\norm{f_0}_\infty$ and $\norm{f_0}_{\Sigma(\alpha)}$ only. Hence, from \eqref{pol_H2},
\begin{align*}
\begin{split}
\underset{t \in \left[0; \frac{\lfloor \alpha \rfloor+1}{2^L}\right)}{\sup} \Bigg|  f_0(t) - A^2_{\lfloor \alpha \rfloor, 2^{-L}}\left(\sum_{i=0}^{2^{L}-1} \theta_i 2^{L}  \mathds{1}_{\left[i2^{-L}; (i+1)2^{-L}\right)}\right)(t)    \Bigg|  &\leq  \underset{t \in \left[0; \frac{\lfloor \alpha \rfloor+1}{2^L}\right)}{\sup}| f_0(t) - Q_1(t)| +\\
&\qquad \underset{t \in \left[0; \frac{\lfloor \alpha \rfloor+1}{2^L}\right)}{\sup}\left| Q_1(t) - P\Bigg( \frac{2^Lt}{\lfloor \alpha \rfloor+1}\Bigg)\right|\\
&\leq C 2^{-\alpha L} + ||Q-P||_{\infty, \left[0;1\right)}\\
&\leq C 2^{-\alpha L}
\end{split}
\end{align*}
with $C$ a constant depending only on $f_0$, $\norm{f_0}_\infty$ (which exists as $f_0$ is a Hölderian density, see for instance \cite{tsyb_npe}, p.9) and $\norm{f_0}_{\Sigma(\alpha)}$.
Using the same reasoning to control the distance on $\left[1-\frac{\lfloor \alpha \rfloor+1}{2^L}; 1\right)$ and the equality \eqref{H1_H2}, we conclude that
\begin{equation}\label{closeness_first} \norm{f_0 - A^2_{\lfloor \alpha \rfloor, 2^{-L}}\left( \sum_{0\leq i \leq 2^L-1}  \eta_{i} H_{Li} \right)}_\infty \leq C  2^{-\alpha L} \end{equation}
and $\Big|1-\int_0^1 A^2_{\lfloor \alpha \rfloor, 2^{-L}}\left( \sum_{0\leq i \leq 2^L-1}  \eta_{i} H_{Li} \right) (u)du\Big|\leq C  2^{-\alpha L} $. For $L$ large enough, as $f_0$ is lower bounded by a strictly positive constant, we thus have that $A^2_{\lfloor \alpha \rfloor, 2^{-L}}\left( \sum_{0\leq i \leq 2^L-1}  \eta_{i} H_{Li} \right)$ is positive for $L$ large enough, leading to the simplification of the definition \eqref{sd_map_def2} (we remove the positive parts from the formula)
\[ f_0^{L} =  SD_{\tau, \lfloor \alpha \rfloor, 2^{-L}}\left(\sum_{0\leq i \leq 2^L-1}  \eta_{i} H_{Li}\right)= \frac{A^2_{\lfloor \alpha \rfloor, 2^{-L}}\left( \sum_{0\leq i \leq 2^L-1}  \eta_{i} H_{Li} \right)+\tau}{\int_0^1 \Bigg(A^2_{\lfloor \alpha \rfloor, 2^{-L}}\left( \sum_{0\leq i \leq 2^L-1}  \eta_{i} H_{Li} \right)(u) + \tau\Bigg)du}.\]
We also have that $A^2_{\lfloor \alpha \rfloor, 2^{-L}}\left( \sum_{0\leq i \leq 2^L-1}  \eta_{i} H_{Li} \right)$ is upper bounded by a constant depending only on $\alpha$, $\norm{f_0}_\infty$ and $\norm{f}_{\Sigma(\alpha)}$ as a consequence from \eqref{closeness_first}. Finally, 

\begin{align*}
\norm{f_0 - f_0^{L} }_\infty &\leq C2^{-\alpha L} + \tau+\\
&\left|1-\frac{1}{\int_0^1 \Bigg(A^2_{\lfloor \alpha \rfloor, 2^{-L}}\left( \sum_{0\leq i \leq 2^L-1}  \eta_{i} H_{Li} \right)(u) + \tau\Bigg)du}\right|\times\\
&\norm{A^2_{\lfloor \alpha \rfloor, 2^{-L}}\left( \sum_{0\leq i \leq 2^L-1}  \eta_{i} H_{Li} \right)(u) + \tau}_\infty\\
&\leq C\left(2^{-\alpha L}  + \tau\right).
\end{align*}

\end{proof}


\subsection{Miscellaneous.}

\subsubsection{Extension of Hölderian maps.}

\begin{lemma}
\label{ext_Hold}
Let $E\subset\mathbb{R}$ be a non-empty interval and let $f_0\in \Sigma\left(\alpha, E\right)$ where $\alpha>0$. Then there exists a function $\widetilde{f}\in \Sigma\left(\alpha, \mathbb{R}\right)$ so that $\restr{\widetilde{f}}{E}=f_0$ and $\norm{\widetilde{f}}_{\Sigma(\alpha)}=\norm{f_0}_{\Sigma(\alpha)}$.
\end{lemma}
\begin{proof}
First, let's assume that $\alpha\leq 1$. We use the fact that
\begin{equation}
\label{base_ineg}
|x_{1}-x_{2}|^{\alpha}\leq(|x_{1}-x_{3}|+|x_{2}-x_{3}|)^{\alpha}\leq
|x_{1}-x_{3}|^{\alpha}+|x_{2}-x_{3}|^{\alpha}.
\end{equation}
First, for $f_0\in \Sigma\left(\alpha, E\right)$, we have $|f_0(x)-f_0(y)|\leq \norm{f}_{\Sigma(\alpha)}|x-y|^{\alpha}$ for all $x,y\in E$ and we define
\[
h(x)\coloneqq\inf\left\{  f_0(y)+\norm{f_0}_{\Sigma(\alpha)}|x-y|^{\alpha}:\,y\in E\right\}  ,\quad
x\in\mathbb{R}^{n}.
\]
If $x\in E$, then taking $y=x$ we get that $h(x)\leq f_0(x)$. To prove that $h(x)$ is finite for every $x\in\mathbb{R}^{n}$, fix $y_{0}\in
E$. If $y\in E$ then, from \eqref{base_ineg},
\[
f_0(y)-f_0(y_{0})+\norm{f_0}_{\Sigma(\alpha)}|x-y|^{\alpha}\geq-\norm{f_0}_{\Sigma(\alpha)}|y-y_{0}|^{\alpha}+\norm{f_0}_{\Sigma(\alpha)}|x-y|^{\alpha}%
\geq-\norm{f_0}_{\Sigma(\alpha)}|x-y_{0}|^{\alpha},
\]
and so
\begin{align*}
h(x)   =\inf\left\{  f_0(y)+\norm{f_0}_{\Sigma(\alpha)}|x-y|^{\alpha}:\,y\in E\right\}  
  \geq f_0(y_{0})-\norm{f_0}_{\Sigma(\alpha)}|x-y_{0}|^{\alpha}>-\infty.
\end{align*}
Note that if $x\in E$, then we can choose $y_{0}=x$ in the previous
inequality to obtain $h(x)\geq f\left(  x\right)  $. Thus $h$ extends $f_0$. 
Next we prove that
\[\left\vert h(x_{1})-h\left(  x_{2}\right)  \right\vert \leq \norm{f_0}_{\Sigma(\alpha)}|x_{1}-x_{2}|^{\alpha}\]
for any $x_{1}$,$\,x_{2}\in\mathbb{R}$. Given $\varepsilon>0$, by the definition of $h$ there exists
$y_{1}\in E$ such that
\[
h(x_{1})\geq f_0(y_{1})+\norm{f_0}_{\Sigma(\alpha)}|x_{1}-y_{1}|^{\alpha}-\varepsilon.
\]
Since $h\left(  x_{2}\right)  \leq f_0(y_{1})+\norm{f_0}_{\Sigma(\alpha)}|x_{2}-y_{1}|^{\alpha}$, we get from \eqref{base_ineg}
\begin{align*}
h(x_{1})-h\left(  x_{2}\right)   &  \geq \norm{f_0}_{\Sigma(\alpha)}|x_{1}-y_{1}|^{\alpha}-\norm{f_0}_{\Sigma(\alpha)}|x_{2}-y_{1}|^{\alpha}-\varepsilon\\
&  \geq-\norm{f_0}_{\Sigma(\alpha)}|x_{1}-x_{2}|^{\alpha}-\varepsilon.
\end{align*}
Letting $\varepsilon\rightarrow0$ gives $h(x_{1})-h\left(  x_{2}\right)\geq-\norm{f_0}_{\Sigma(\alpha)}|x_{1}-x_{2}|^{\alpha}$. It remains to reverse the roles of $x_{1}$ and $x_{2}$ to prove that $h$ is Hölder continuous with $\norm{h}_{\Sigma(\alpha)}=\norm{f_0}_{\Sigma(\alpha)}$.\\
Now, if $\alpha>1$, we have that $f_0\in \Sigma\left(\alpha, E\right)$ implies \[f_0^{(\lfloor\alpha\rfloor)}\in \Sigma\left(\alpha-\lfloor\alpha\rfloor, E\right),\quad \norm{f_0^{(\lfloor\alpha\rfloor)}}_{\Sigma(\alpha-\lfloor\alpha\rfloor)}=\norm{f_0}_{\Sigma(\alpha)}.\]The above proof ensures that there exist $g\in  \Sigma\left(\alpha-\lfloor\alpha\rfloor, \mathbb{R}\right)$ such that
$\restr{g}{E}=f_0^{(\lfloor\alpha\rfloor)}$. As this last application as well as $g$ are continuous, it suffices to take the successive primitives of $g$ (with equality constraint ensuring that these are also derivatives of $f_0$) to obtain the result.
\end{proof}


\subsubsection{Control of discretization error.}

\begin{lemma}
\label{Discretiz_control}
Let $f$ be a piecewise constant map on the intervals $[i/l;(i+1)/l)$, $i=0,\dots,l-1$, with $l\in\N^*$. Let's assume that $f$ takes values in $[0;l']$. Then, if $g$ is the $1$-periodic extension of $f$ on $\R$, we have
\[ \norm{g_{\infty,s}^{m}  - g_{q,s}^{m} }_\infty \leq \frac{ml'(ls+1)}{q}\]
for $m\in\N$, $0<s<1/2$.
\end{lemma}

\begin{proof}
If $m=0$, the result is straightforward. Otherwise, we start  by noting that, as $g$ is $1$-periodic, \eqref{rec_def} gives that $g_{\infty,s}^{m}$, as well as $g_{q,s}^{m}$,  are themselves  $1$-periodic. It is therefore sufficient to control $\left|g_{\infty,s}^{m}(x)  - g_{q,s}^{m}(x) \right|$ for $x$ in $[0;1)$.\\
Before going further, let's first show that for any $m\geq0$,  $g_{\infty,s}^{m}$ is a function of bounded variation over any interval $[x-s;x]$, with a bound on its total variation independent of $x\in\R$. This means that there exists a constant $V=V(l',l,s)>0$, such that, for any $x\in\R$ and subdivision $\sigma = \left\{x-s=x_1<x_2\dots x_{n-1}<x_n=x \given n \geq 2 \right\}$, \[ \sum_{i=1}^{n} \left|g_{\infty,s}^{m}(x_{i+1})-g_{\infty,s}^{m}(x_i)\right|\leq V.\]
By assumptions on $f$, on an interval of length $s$, $g$ is piecewise constant with at most $ls+1$ discontinuity points and takes values in $[0;l']$. Therefore, $g$ has bounded variation at most $V=l'(ls+1)$.
Then we show that the convolution of $s^{-m} \chi_s^{*m}$ with $g$ is also a function with bounded variation on any interval $[x-s;x]$ with $x\in[0;1)$. Indeed, for a subdivision $\sigma$ of $[x-s;x]$ and $m$ such that $s(m+1)<1$, using \eqref{rec_def} and Lemma \ref{int_chi_formula}
\begin{align}
\label{rec_bnd_var}
 \sum_{i=1}^{n} \left|g_{\infty,s}^{m}(x_{i+1})-g_{\infty,s}^{m}(x_i)\right| &=  \sum_{i=1}^{n} \left| s^{-m} \chi_s^{*m} * g(x_{i+1}) - s^{-m} \chi_s^{*m} * g(x_{i}) \right|\nonumber\\
 &\leq \int_\R  \left[\sum_{i=1}^{n}\Big| g(x_{i+1}-u) -g(x_{i}-u) \Big|\right]  \left|s^{-m} \chi_s^{*m}(u)\right| \ du\nonumber\\
 &\leq l'(ls+1) \int_\R  s^{-m} \chi_s^{*m}(u)\ du\\
 &= V.\nonumber
\end{align}
Let's assume $m=1$, in which case  $g_{q,s}^{1}(x)$ is a Riemann sum with converges to $g_{\infty,s}^{1}(x)= s^{-1}\int_{x-s}^x g(t) dt $ as $q\to\infty$. More precisely, with the bounded variation property above, it is a common result that in this case
\[ \left|sg_{\infty,s}^{1}(x)  - sg_{q,s}^{1}(x) \right| \leq  V\frac{s}{q}\leq \frac{l's(ls+1)}{q}.\]
This proves the lemma for $m=1$.\\
For the general case, let's assume that the lemma is true for some $m\in\N^*$. We use the same argument, since the above equation translates in
\[ \left|g_{\infty,s}^{m+1}(x)  - \Big(g_{\infty,s}^{m}\Big)_{q,s}^{1} (x) \right| \leq  \frac{l'(ls+1)}{q},\]
according to \eqref{rec_bnd_var}. Also, the property at level $m$ ensures that
\begin{align*}
\begin{split}
\left| g_{q,s}^{m+1}(x) - \Big(g_{\infty,s}^{m}\Big)_{q,s}^{1} (x) \right| &\leq \frac{1}{q}\sum_{i=0}^{q-1} \left| g_{q,s}^{m}\Big(x-\frac{is}{q}\Big) -  g_{\infty,s}^{m}\Big(x-\frac{is}{q}\Big)\right|\\
&\leq \frac{ml'(ls+1)}{q}.
\end{split}
\end{align*}
Using the triangular inequality on $\left|g_{\infty,s}^{m+1}(x)  - g_{q,s}^{m+1}(x) \right| $ to obtain a bound from the sum of the two terms above allows to conclude the proof.
\end{proof}

\subsubsection{Equivalence of norms on spaces of polynomials.}

\begin{lemma}
\label{ineq_norms}
Let $\mathbb{R}_n[X]$ be the space of polynomials of degree at most $n$ and with real coefficients. For $P\in\mathbb{R}_n[X]$, let's write
\[ \norm{P}_{\infty, J} \coloneqq \underset{t\in J}{\max}|P(t)|, \quad \text{with }J\text{ an interval in }\R. \]
Then, for all $P\in\mathbb{R}_n[X]$ and $s\in[0;1)$, we have
\[ \norm{P}_{\infty, [0;1]} \leq e^{\sqrt{6s^{-1}}n} \norm{P}_{\infty, [0;s]}\]
as well as
\[ \norm{P}_{\infty, [0;1]} \leq e^{\sqrt{6s^{-1}}n} \norm{P}_{\infty, [1-s;1]}.\]
It also follows that 
\[ \int_{[0;1]} \left|P(u)\right|du \leq e^{\sqrt{6s^{-1}}n} \frac{2}{s}(n+1)^2\int_{[0;s]} \left|P(u)\right|du, \]
\[ \int_{[0;1]} \left|P(u)\right|du \leq e^{\sqrt{6s^{-1}}n} \frac{2}{s}(n+1)^2\int_{[1-s;1]} \left|P(u)\right|du. \]
\end{lemma}

\begin{proof}
For $n=0$, the result is straightforward. Let's then delve into the case $n\geq 1$. First, as
\[ x^k-y^k=(x-y)\sum_{i=0}^{k-1}x^iy^{k-1-i} \]
for any $x,y$ in $[-1;1]$,
\[|x^k-y^k|\leq k|x-y|.\]
Then, if $P(x)=\sum_{k=0}^n a_k x^k$, we have 
\[|P(x)-P(y)|\leq \sum_{k=1}^n|a_k||x^k-y^k|\leq|x-y|\sum_{k=1}^nk|a_k|.\]
Before going further, we point out that necessarily $a_k = P^{(k)}(0)/(k!)$. Now, let's recall Markov's brother inequality (that can be found in \cite{boas1969inequalities} for instance) which states that, for $Q\in\mathbb{R}_n[X]$ and any nonnegative integers $k$

\[  \max_{-1 \leq x \leq 1} |Q^{(k)}(x)| \leq \frac{n^2 (n^2 - 1^2) (n^2 - 2^2) \cdots (n^2 - (k-1)^2)}{1 \cdot 3 \cdot 5 \cdots (2k-1)} \max_{-1 \leq x \leq 1} |Q(x)|. \]
The constant appearing in the above inequality is equal to $T_n^{(2k)}(1)$, where $T_n$ is the $n$-th Chebyshev polynomial. Let's apply this inequality to $Q = P \circ \left(\frac{s}{2}(X+1)\right)$. First, we have

\[ \norm{P}_{\infty, [0;s]} = \max_{-1 \leq x \leq 1} |Q(x)|. \]
And then, for $x \in [-1;1]$, 
\[ Q^{(k)}(x) = \left(\frac{s}{2}\right)^{k} P^{(k)}\left(\frac{s}{2}(x+1)\right),\]
so that
\[ \max_{-1 \leq x \leq 1} |Q^{(k)}(x)| = \left(\frac{s}{2}\right)^{k} \max_{0 \leq x \leq s} | P^{(k)}\left(x\right) | \geq \left(\frac{s}{2}\right)^{k} |P^{(k)}(0)|.\]
Combining these results finally gives us, for $0\leq x,y\leq 1$,
\begin{align*}
\begin{split}
 |P(x)-P(y)| &\leq |x-y|\sum_{k=1}^nk|a_k|\\
 &= |x-y|\sum_{k=1}^n\frac{|P^{(k)}(0)|}{(k-1)!}\\
 &\leq |x-y|\sum_{k=1}^n\left(\frac{2}{s}\right)^{k} \frac{T_n^{(2k)}(1) \norm{P}_{\infty, [0;s]} }{(k-1)!}.
\end{split}
\end{align*}
It follows readily that 

\[  \norm{P}_{\infty, [0;1]} \leq \Bigg(1+ \sum_{k=1}^n\left(\frac{2}{s}\right)^{k} \frac{T_n^{(2k)}(1) }{(k-1)!} \Bigg) \norm{P}_{\infty, [0;s]}.\]
The multiplicative factor above is then bounded by
\begin{align*}
\begin{split}
1+ \sum_{k=1}^n \frac{ 6^ks^{-k} n^{2k}}{(2k)!} &\leq   1+ \sum_{k=1}^n\frac{ (\sqrt{6s^{-1}}n)^{2k}}{(2k)!}\\
&\leq e^{\sqrt{6s^{-1}}n}
\end{split}
\end{align*}
as $k\leq (3/2)^k, \forall k\geq 1$. This concludes the proof of the first inequality. For the second inequality, it suffices to remark that for $P\in\mathbb{R}_n[X]$, 
\[ \norm{P}_{\infty, [1-s;1]}  = \norm{R}_{\infty, [0;s]},\qquad   \norm{P}_{\infty, [0;1]}  =  \norm{R}_{\infty, [0;1]}, \]
with $R(t)=P(1-t)$ for any $t$ a real number, defining $R$ as an element of $\mathbb{R}_n[X]$.\\

Finally, for the last claim, we introduce the primitive $p(\cdot)=\int_0^\cdot P(t)dt$ which is a polynomial of degree at most $n+1$ verifying \[ \norm{p}_{\infty,[0;s]}\leq \int_0^s |P(t)|dt.\]
With Markov's brother inequality and rescaling, we have
\[ \norm{P}_{\infty,[0;s]}\leq \frac{2}{s}(n+1)^2\norm{p}_{\infty,[0;s]}\]
which allows us to conclude
\[  \int_{[0;1]} \left|P(u)\right|du \leq \norm{P}_{\infty, [0;1)}\leq e^{\sqrt{6s^{-1}}n} \norm{P}_{\infty, [0;s]} \leq e^{\sqrt{6s^{-1}}n} \frac{2}{s}(n+1)^2\int_0^s |P(u)|du.\]
\end{proof}